\documentclass[11pt,a4paper]{article}

\pdfoutput=1 
\usepackage{epsf,epsfig,amsfonts,amsgen,amsmath,amstext,amsbsy,amsopn,amsthm,cases,listings,color
}
\usepackage{ebezier,eepic}
\usepackage{color}
\usepackage{multirow}
\usepackage{epstopdf}
\usepackage{graphicx}
\usepackage{pgf,tikz}
\usepackage{mathrsfs}
\usepackage[marginal]{footmisc}
\usepackage{enumitem}
\usepackage[titletoc]{appendix}
\usepackage{booktabs}
\usepackage{url}
\usepackage{mathtools}

\usepackage{pgfplots}
\usepackage{authblk}
\usepackage{amssymb}

\usepackage{wasysym}

\usepackage{empheq}

\usepackage{dsfont}
\usepackage{subfigure}
\usepackage{mathrsfs}
\usepackage{wasysym} 
\usetikzlibrary{arrows}
\usepackage{aligned-overset}

\usepackage{algorithm}
\usepackage{algpseudocode}

\usepackage{mathtools}

\DeclarePairedDelimiter{\floor}{\lfloor}{\rfloor}
\usepackage{algcompatible}
\usepackage[backref=page]{hyperref}

\allowdisplaybreaks[1]

\definecolor{uuuuuu}{rgb}{0.27,0.27,0.27}
\definecolor{sqsqsq}{rgb}{0.1255,0.1255,0.1255}

\newtheorem{definition}{Definition} [section]
\newtheorem{theorem}[definition]{Theorem}
\newtheorem{lemma}[definition]{Lemma}
\newtheorem{proposition}[definition]{Proposition}

\newtheorem{corollary}[definition]{Corollary}

\newtheorem{claim}[definition]{Claim}

\newtheorem{fact}[definition]{Fact}
\newenvironment{pf}{\noindent{\bf Proof}}

\newcommand{\uproduct}{\mathbin{\;{\rotatebox{90}{\textnormal{$\small\Bowtie$}}}}}

\newcommand{\etal}{\textit{et al}. }
\newcommand{\ie}{\textit{i}.\textit{e}., }

\newcommand{\RomanNumeralCaps}[1]{\MakeUppercase{\romannumeral #1}}

\begin{document}
\title{\bf\Large Exact results for some extremal problems on expansions \RomanNumeralCaps{1}}

\date{\today}
\author[1]{Xizhi Liu\thanks{Research was supported by ERC Advanced Grant 101020255 and Leverhulme Research Project Grant RPG-2018-424. Email: \texttt{xizhi.liu.ac@gmail.com}}}
\author[2]{Jialei Song\thanks{Research was supported by Science and Technology Commission of Shanghai Municipality (No. 22DZ2229014). Email: \texttt{jlsong@math.ecnu.edu.cn}}}
\author[3]{Long-Tu Yuan\thanks{Research was supported by National Natural Science
Foundation of China (No. 12271169 and 12331014)  and Science and Technology Commission of Shanghai Municipality (No. 22DZ2229014). Email: \texttt{ltyuan@math.ecnu.edu.cn}}}

\affil[1]{Mathematics Institute and DIMAP,
            University of Warwick, 
            Coventry, CV4 7AL, UK}
\affil[2,3]{
School of Mathematical Sciences,  Key Laboratory of Mathematics and Engineering Applications (Ministry of Education) \& Shanghai Key Laboratory of PMMP,  East China Normal University, Shanghai 200241, China
}
\maketitle
\begin{abstract}
The expansion of a graph $F$, denoted by $F^3$, is the $3$-graph obtained from $F$ by adding a new vertex to each edge such that different edges receive different vertices. 
We establish a stability version of a theorem by Kostochka--Mubayi--Verstra\"{e}te~\cite{KMV17b} and demonstrate two applications of it by establishing tight upper bounds for large $n\colon$
\begin{itemize}
    \item The maximum number of edges in an $n$-vertex $3$-graph that does not contain $T^3$ for certain class $\mathcal{T}$ of trees, thereby (partially) sharpening the asymptotic result of Kostochka--Mubayi--Verstra\"{e}te.  
    \item The minimum number of colors needed to color the complete $n$-vertex $3$-graph to ensure the existence of a rainbow copy of $F^3$ when $F$ is a graph obtained from some tree $T\in \mathcal{T}$ by adding a new edge, thereby extending anti-Ramsey results on $P_{2t}^3$ by Gu--Li--Shi and $C_{2t}^3$ by Tang--Li--Yan.
\end{itemize}

We introduce a framework that utilizes tools from Extremal Set Theory for solving certain generalized Tur\'{a}n problems. 
More specifically, we establish a parallel of the stability theorem above in generalized Tur\'{a}n problems. 
Using this stability theorem, we determine, for large $n$, the maximum number of triangles in an $n$-vertex graph that does not contain the shadow of $C_{k}^3$ or $T^3$ for $T\in \mathcal{T}$, thus answering a question of Lv \etal on generalized Tur\'{a}n problems. 

\medskip

\end{abstract}
\section{Introduction}\label{SEC:Introduction}
Fix an integer $r\ge 2$, an $r$-graph $\mathcal{H}$ is a collection of $r$-subsets of some finite set $V$. We identify a hypergraph $\mathcal{H}$ with its edge set and use $V(\mathcal{H})$ to denote its vertex set. The size of $V(\mathcal{H})$ is denoted by $v(\mathcal{H})$. 
Given a family $\mathcal{F}$ of $r$-graphs, we say $\mathcal{H}$ is \textbf{$\mathcal{F}$-free}
if it does not contain any member of $\mathcal{F}$ as a subgraph.
The \textbf{Tur\'{a}n number} $\mathrm{ex}(n,\mathcal{F})$ of $\mathcal{F}$ is the maximum
number of edges in an $\mathcal{F}$-free $r$-graph on $n$ vertices.
The study of $\mathrm{ex}(n,\mathcal{F})$ and its variant has been a central topic in extremal graph and hypergraph theory since the seminal work of Tur\'{a}n~\cite{T41}.  
 
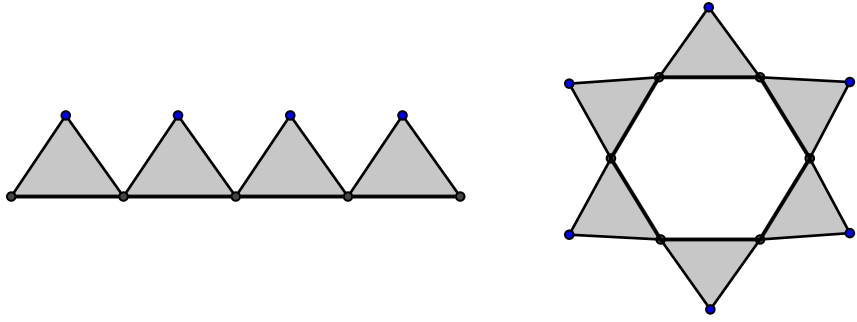
\begin{figure}[htbp]
\centering
\tikzset{every picture/.style={line width=0.75pt}} 
\begin{tikzpicture}[x=0.75pt,y=0.75pt,yscale=-1,xscale=1,scale=0.8]

\draw[line width=1.5pt,color=sqsqsq]  (44,163) -- (114,163) -- (184,163) -- (254,163) -- (324,163) ;
\draw[line width=1pt, fill=sqsqsq,fill opacity=0.25]   (78,112) -- (114,163) -- (44,163) -- cycle ;
\draw[line width=1pt, fill=sqsqsq,fill opacity=0.25]  (148,112) -- (184,163) -- (114,163) -- cycle ;
\draw[line width=1pt, fill=sqsqsq,fill opacity=0.25]   (218,112) -- (254,163) -- (184,163) -- cycle ;
\draw[line width=1pt, fill=sqsqsq,fill opacity=0.25]  (288,112) -- (324,163) -- (254,163) -- cycle ;
\draw [fill=uuuuuu] (44,163) circle (2pt);
\draw [fill=uuuuuu]  (114,163) circle (2pt);
\draw [fill=uuuuuu]  (184,163) circle (2pt);
\draw [fill=uuuuuu]  (254,163) circle (2pt);
\draw [fill=uuuuuu]  (324,163)  circle (2pt);
\draw [fill=blue]   (78,112)  circle (2pt);
\draw [fill=blue]   (148,112) circle (2pt);
\draw [fill=blue]  (218,112)  circle (2pt);
\draw [fill=blue]   (288,112)  circle (2pt);

\draw[line width=1.5pt,color=sqsqsq]  (542,139) -- (511,190) -- (449,190) -- (418,139) -- (448,88) -- (511,88) -- cycle ;

\draw [fill=uuuuuu] (542,139) circle (2pt);
\draw [fill=uuuuuu]  (511,190) circle (2pt);
\draw [fill=uuuuuu]  (449,190) circle (2pt);
\draw [fill=uuuuuu]  (418,139) circle (2pt);
\draw [fill=uuuuuu]  (448,88)  circle (2pt);
\draw [fill=uuuuuu] (511,88)  circle (2pt);

\draw[line width=1pt, fill=sqsqsq,fill opacity=0.25]   (479,44) -- (511,88) -- (448,88) -- cycle ;
\draw[line width=1pt, fill=sqsqsq,fill opacity=0.25]   (567,91) -- (542,139) -- (511,88) -- cycle ;
\draw[line width=1pt, fill=sqsqsq,fill opacity=0.25]   (480,234) -- (449,190) -- (511,190) -- cycle ;
\draw[line width=1pt, fill=sqsqsq,fill opacity=0.25]   (567,186) -- (511,190) -- (542,139) -- cycle ;
\draw[line width=1pt, fill=sqsqsq,fill opacity=0.25]   (392,92) -- (448,88) -- (418,139) -- cycle ;
\draw[line width=1pt, fill=sqsqsq,fill opacity=0.25]   (392,187) -- (418,139) -- (449,190) -- cycle ;
\draw [fill=blue]  (479,44)  circle (2pt);
\draw [fill=blue]   (567,91)  circle (2pt);
\draw [fill=blue]  (480,234)   circle (2pt);
\draw [fill=blue]   (567,186)  circle (2pt);
\draw [fill=blue]  (392,92)  circle (2pt);
\draw [fill=blue]   (392,187)  circle (2pt);

\end{tikzpicture}
\caption{Expansions of $P_4$ and $C_6$ for $r=3$.}
\label{fig:Path-Cycle-Expansion}
\end{figure}

Given a graph $F$, 
the \textbf{expansion} $F^{r}$ of $F$ is the $r$-graph obtained from $F$ by adding a set of $r-2$ new vertices to each edge such that different edges receive disjoint sets (see Figure~\ref{fig:Path-Cycle-Expansion}).
Expansions are important objects in Extremal set theory and Hypergraph Tur\'{a}n problems. 
It was introduced by Mubayi in~\cite{M06} as a way to extend Tur\'{a}n's theorem to hypergraphs. 
There has been a significant amount of progress in the study of expansion over the last few decades, and we refer the reader to the survey~\cite{MV15Survey} by Mubyai--Verstra\"{e}te for more related results.

In the present work, our primary focus will be extremal problems related to the expansion of trees and cycles when $r\in \{2,3\}$. 
The following parameters will be crucial for us. 

Following Frankl--F\"{u}redi~\cite{FF87}, the \textbf{crosscut number} $\sigma(F)$ of graph $F$ is 
\begin{align*}
    \sigma(F) := \min\left\{|I|+|F-I| \colon \text{$I$ is independent in $F$}\right\}. 
\end{align*}
Recall that the \textbf{covering number} $\tau(F)$ of a graph $F$ is 
    \begin{align*}
                \tau(F) & := 
                \min\left\{|S| \colon S\subseteq V(F) \mathrm{\ such\ that\ } |S\cap e| \ge 1 \mathrm{\ for\ all\ } e\in F\right\}. 
    \end{align*}
Following the definition in~\cite{Yuan22}, the \textbf{independent covering number} $\tau_{\mathrm{ind}}(F)$ of a bipartite graph $F$ is
            \begin{align*}
                \tau_{\mathrm{ind}}(F) & := 
                \min\left\{|S| \colon S\subseteq V(F) \mathrm{\ such\ that\ } |S\cap e| = 1 \mathrm{\ for\ all\ } e\in F\right\}. 
            \end{align*}
It follows trivially from the definitions that 
$\tau(F) \le \sigma(F)$ holds for all graphs $F$, 
and $\tau(F) \le \sigma(F) \le \tau_{\mathrm{ind}}(F)$ holds for all bipartite graphs $F$. 

Extending a definition in~\cite{FJ15}, we say a tree $T$ is \textbf{strongly edge-critical} if
    \begin{itemize}
        \item $\tau(T) = \sigma(T) = \tau_{\mathrm{ind}}(T)$, and 
        \item $T$ contains a \textbf{critical edge}, i.e. there exists an edge $e\in T$ such that $\sigma(T\setminus e) \le \sigma(T)-1$. 
    \end{itemize}
Path and cycle of length $k$ will be denoted by $P_k$ and $C_k$, respectively. 
\subsection{Tur\'{a}n problems for the expansion of trees}
Determining $\mathrm{ex}(n,T)$ when $T$ is a tree or hypertree is a central topic in Extremal combinatorics. 
Two well-known conjectures in this regard, the Erd{\H o}s--S\'{o}s Conjecture on trees~\cite{ES64} and Kalai's conjecture on hypertrees (see~\cite{FF87}), are still open in general. 

For $r \ge 4$ and large $n$, the exact value of $\mathrm{ex}(n,F^r)$ when $F$ is a path, a cycle, or in a special class of trees, is now well-understood due to series works of F\"{u}redi~\cite{F14tree}, F\"{u}redi--Jiang~\cite{FJ15cycle,FJ15}, F\"{u}redi--Jiang--Seiver~\cite{FJS14path}. 
Using methods significantly different from theirs, Kostochka--Mubyai--Verstra\"{e}te~\cite{KMV15a,KMV15c,KMV17b} extended these results to $r=3$,  determining the exact value of $\mathrm{ex}(n,F^3)$ when $F$ is a path or a cycle, and proving the following  asymptotic result for trees.
\begin{theorem}[Kostochka--Mubyai--Verstra\"{e}te~\cite{KMV17b}]\label{THM:KMV-tree}
    Suppose that $T$ is a tree. 
    Then 
    \begin{align*}
        \mathrm{ex}(n,T^3) = \left(\frac{\sigma(T)-1}{2} +o(1)\right)n^2. 
    \end{align*}
\end{theorem}
It is worth mentioning that the main term in the theorem above comes from the following construction. 
Let $n \ge t \ge 0$ be integers. Define the $3$-graph
    \begin{align*}
        \mathcal{S}(n,t)
        := \left\{e\in \binom{[n]}{3} \colon |e\cap [t]| \ge 1\right\}.
    \end{align*}
It is easy to see that $\mathcal{S}(n,t)$ is $T^3$-free for all trees with $\tau(T) \ge t+1$. Hence, $\mathrm{ex}(n,T^3) \ge |\mathcal{S}(n,t)|$ holds for all trees with $\tau(T) \ge t+1$. 
Our main result in this subsection is an exact determination of $\mathrm{ex}(n,F^3)$ when $F$ is a strongly edge-critical tree, thus sharpening (partially) Theorem~\ref{THM:KMV-tree}. 
\begin{theorem}\label{THM:Hypergraph-Tree-Exact}
    Suppose that $T$ is a tree satisfying $\sigma(T) = \tau_{\mathrm{ind}}(T)$ and containing a critical edge.
    Then 
    \begin{align*}
        \mathrm{ex}(n, T^{3})
        \le |\mathcal{S}(n,\sigma(T)-1)|
        \quad\text{for all sufficiently large $n$}.
    \end{align*}
    In particular, $T$ is a strongly edge-critical tree.
    Then 
    \begin{align*}
        \mathrm{ex}(n, T^{3})
        = |\mathcal{S}(n,\sigma(T)-1)|
        \quad\text{for all sufficiently large $n$}. 
    \end{align*}
\end{theorem}
\textbf{Remark.}
\begin{itemize} 
    \item Note that one cannot hope that Theorem~\ref{THM:Hypergraph-Tree-Exact} holds for all trees, as paths of even length would be counterexamples (see~\cite{KMV15a}). 
    \item For $r\ge 4$ F\"{u}redi--Jiang~\cite{FJ15} determined $\mathrm{ex}(n, T^{r})$ for all edge-critical trees $T$, where a tree $T$ is edge-critical if $\tau(T) = \sigma(T)$ and there exists an edge $e\in T$ such that $\sigma(T\setminus e) \le \sigma(T)-1$.
It would be interesting to see if the conclusion of Theorem~\ref{THM:Hypergraph-Tree-Exact} still holds under this weak constraint.
\end{itemize}

Theorem~\ref{THM:Hypergraph-Tree-Exact} is proved by using the stability method developed by Siminovits~\cite{S68}, and one of the key steps is to prove the following stability theorem for trees. 

We say an $n$-vertex $3$-graph $\mathcal{H}$ is \textbf{$\delta$-close} to $\mathcal{S}(n,t)$ if there exists a set $L\subseteq V(\mathcal{H})$ of size $t$ such that
\begin{align*}
    |\mathcal{H}-L| \le \delta n^2 \quad\text{and}\quad
    d_{\mathcal{H}}(v) \ge \left(1/2 -\delta\right)n^2 \text{ for all $v\in L$}. 
\end{align*}
\begin{theorem}[Stability]\label{THM:Hypergraph-Tree-Stability}
    Let $T$ be a tree with $\sigma(T) = \tau_{\mathrm{ind}}(T)$.  
    For every $\delta>0$ there exist $\varepsilon>0$ and $n_0$ such that the following holds for all $n\ge n_0$. 
    Suppose that $\mathcal{H}$ is an $n$-vertex $T^3$-free $3$-graph with
    \begin{align*}
        |\mathcal{H}|
        \ge \left(\frac{\sigma(T) - 1}{2}- \varepsilon\right) n^2.  
    \end{align*}
    Then $\mathcal{H}$ is $\delta$-close to $\mathcal{S}(n,\sigma(T)-1)$. 
\end{theorem}
The proofs for Theorems~\ref{THM:Hypergraph-Tree-Exact} and~\ref{THM:Hypergraph-Tree-Stability} 
will be presented in Sections~\ref{SUBSEC:Proof-Hypergraph-Turan-Exact} and~\ref{SUBSEC:Proof-Hypergraph-Turan-Stability}, respectively. 
\subsection{Anti-Ramsey problems}\label{SUBSEC:Intro:antiRamsey}
Let $F$ be an $r$-graph. 
The \textbf{anti-Ramsey number} $\mathrm{ar}(n,F)$ of $F$ is the minimum number $m$ such that any surjective map $\chi \colon K_n^r \to [m]$ contains a rainbow copy of $F$, i.e. a copy of $F$ in which every edge receives a unique value under $\chi$. 
It is easy to observe from the definition that for every $r$-graph $F$, 
    \begin{align*}
        \mathrm{ex}\left(n, \{F\setminus e\colon e\in F\}\right) + 2 
        \le \mathrm{ar}(n,F)
        \le \mathrm{ex}(n, F)+1. 
    \end{align*}
The study of anti-Ramsey problems was initiated by Erd\H{o}s--Simonovits--S\'{o}s~\cite{ESS75} who proved that $\mathrm{ar}(n,K_{r+1}) = \mathrm{ex}(n, K_{r})+2$ for $r \ge 3$ and sufficiently large $n$. 
There has been lots of progress on this topic since their work and we refer the reader to a survey~\cite{FMO10} by Fujita--Magnant--Ozeki for more details. 
In this subsection we are concerned with $\mathrm{ar}(n,F)$ when $F$ is the expansion of a tree plus one edge.

\begin{figure}[htbp]
\centering
\begin{minipage}{0.15\linewidth}
\centering
\begin{tikzpicture}[x=0.75pt,y=0.75pt,yscale=-1,xscale=1,scale=0.45]
\draw[line width=1pt,color=sqsqsq]   (257.31,70.98) .. controls (270.32,70.88) and (281.09,101.48) .. (281.36,139.32) .. controls (281.63,177.16) and (271.3,207.91) .. (258.29,208.01) .. controls (245.28,208.1) and (234.51,177.5) .. (234.24,139.66) .. controls (233.97,101.82) and (244.3,71.07) .. (257.31,70.98) -- cycle ;
\draw[line width=1pt,color=sqsqsq]   (352.72,44.92) .. controls (369.82,44.8) and (383.98,86.11) .. (384.35,137.19) .. controls (384.71,188.27) and (371.15,229.78) .. (354.04,229.9) .. controls (336.94,230.02) and (322.78,188.71) .. (322.41,137.63) .. controls (322.05,86.55) and (335.62,45.04) .. (352.72,44.92) -- cycle ;
\draw [line width=1pt,color=sqsqsq]   (256,93) -- (350,71) ;
\draw[line width=1pt,color=sqsqsq]    (256,93) -- (350,98) ;
\draw [line width=1pt,color=sqsqsq]   (256,93) -- (350,123) ;
\draw [line width=1pt,color=sqsqsq]   (256,135) -- (350,123) ;
\draw [line width=1pt,color=sqsqsq]   (256,135) -- (350,143) ;
\draw  [line width=1pt,color=sqsqsq]  (256,135) -- (350,167) ;
\draw  [line width=1pt,color=sqsqsq]  (256,160) -- (350,167) ;
\draw[line width=1pt,color=sqsqsq]    (256,160) -- (350,195) ;
\draw [line width=1pt,color=sqsqsq]   (256,187) -- (350,195) ;
\draw [line width=1pt,color=blue]   (256,30) -- (350,30) ;
\draw [fill=uuuuuu] (256,30) circle (2pt);
\draw [fill=uuuuuu]  (350,30) circle (2pt);
\draw [fill=uuuuuu] (256,93) circle (2pt);
\draw [fill=uuuuuu]  (256,135) circle (2pt);
\draw [fill=uuuuuu]  (256,160) circle (2pt);
\draw [fill=uuuuuu]  (256,187) circle (2pt);
\draw [fill=uuuuuu]  (350,71) circle (2pt);
\draw [fill=uuuuuu] (350,98) circle (2pt);
\draw [fill=uuuuuu]  (350,123) circle (2pt);
\draw [fill=uuuuuu]  (350,143)circle (2pt);
\draw [fill=uuuuuu]  (350,167)  circle (2pt);
\draw [fill=uuuuuu]  (350,195)  circle (2pt);
\end{tikzpicture}
\end{minipage}%
\begin{minipage}{0.15\linewidth}
\centering
\begin{tikzpicture}[x=0.75pt,y=0.75pt,yscale=-1,xscale=1,scale=0.45]
\draw[line width=1pt,color=sqsqsq]   (257.31,70.98) .. controls (270.32,70.88) and (281.09,101.48) .. (281.36,139.32) .. controls (281.63,177.16) and (271.3,207.91) .. (258.29,208.01) .. controls (245.28,208.1) and (234.51,177.5) .. (234.24,139.66) .. controls (233.97,101.82) and (244.3,71.07) .. (257.31,70.98) -- cycle ;
\draw[line width=1pt,color=sqsqsq]   (352.72,44.92) .. controls (369.82,44.8) and (383.98,86.11) .. (384.35,137.19) .. controls (384.71,188.27) and (371.15,229.78) .. (354.04,229.9) .. controls (336.94,230.02) and (322.78,188.71) .. (322.41,137.63) .. controls (322.05,86.55) and (335.62,45.04) .. (352.72,44.92) -- cycle ;
\draw [line width=1pt,color=sqsqsq]   (256,93) -- (350,71) ;
\draw[line width=1pt,color=sqsqsq]    (256,93) -- (350,98) ;
\draw [line width=1pt,color=sqsqsq]   (256,93) -- (350,123) ;
\draw [line width=1pt,color=sqsqsq]   (256,135) -- (350,123) ;
\draw [line width=1pt,color=sqsqsq]   (256,135) -- (350,143) ;
\draw  [line width=1pt,color=sqsqsq]  (256,135) -- (350,167) ;
\draw  [line width=1pt,color=sqsqsq]  (256,160) -- (350,167) ;
\draw[line width=1pt,color=sqsqsq]    (256,160) -- (350,195) ;
\draw [line width=1pt,color=sqsqsq]   (256,187) -- (350,195) ;
\draw [line width=1pt,color=blue]   (256,93) -- (350,30) ;
\draw [fill=uuuuuu]  (350,30) circle (2pt);
\draw [fill=uuuuuu] (256,93) circle (2pt);
\draw [fill=uuuuuu]  (256,135) circle (2pt);
\draw [fill=uuuuuu]  (256,160) circle (2pt);
\draw [fill=uuuuuu]  (256,187) circle (2pt);
\draw [fill=uuuuuu]  (350,71) circle (2pt);
\draw [fill=uuuuuu] (350,98) circle (2pt);
\draw [fill=uuuuuu]  (350,123) circle (2pt);
\draw [fill=uuuuuu]  (350,143)circle (2pt);
\draw [fill=uuuuuu]  (350,167)  circle (2pt);
\draw [fill=uuuuuu]  (350,195)  circle (2pt);
\end{tikzpicture}
\end{minipage}%
\begin{minipage}{0.15\linewidth}
\centering
\begin{tikzpicture}[x=0.75pt,y=0.75pt,yscale=-1,xscale=1,scale=0.45]
\draw[line width=1pt,color=sqsqsq]   (257.31,70.98) .. controls (270.32,70.88) and (281.09,101.48) .. (281.36,139.32) .. controls (281.63,177.16) and (271.3,207.91) .. (258.29,208.01) .. controls (245.28,208.1) and (234.51,177.5) .. (234.24,139.66) .. controls (233.97,101.82) and (244.3,71.07) .. (257.31,70.98) -- cycle ;
\draw[line width=1pt,color=sqsqsq]   (352.72,44.92) .. controls (369.82,44.8) and (383.98,86.11) .. (384.35,137.19) .. controls (384.71,188.27) and (371.15,229.78) .. (354.04,229.9) .. controls (336.94,230.02) and (322.78,188.71) .. (322.41,137.63) .. controls (322.05,86.55) and (335.62,45.04) .. (352.72,44.92) -- cycle ;
\draw [line width=1pt,color=sqsqsq]   (256,93) -- (350,71) ;
\draw[line width=1pt,color=sqsqsq]    (256,93) -- (350,98) ;
\draw [line width=1pt,color=sqsqsq]   (256,93) -- (350,123) ;
\draw [line width=1pt,color=sqsqsq]   (256,135) -- (350,123) ;
\draw [line width=1pt,color=sqsqsq]   (256,135) -- (350,143) ;
\draw  [line width=1pt,color=sqsqsq]  (256,135) -- (350,167) ;
\draw  [line width=1pt,color=sqsqsq]  (256,160) -- (350,167) ;
\draw[line width=1pt,color=sqsqsq]    (256,160) -- (350,195) ;
\draw [line width=1pt,color=sqsqsq]   (256,187) -- (350,195) ;
\draw [line width=1pt,color=blue]   (256,30) -- (350,71) ;
\draw [fill=uuuuuu] (256,30) circle (2pt);
\draw [fill=uuuuuu] (256,93) circle (2pt);
\draw [fill=uuuuuu]  (256,135) circle (2pt);
\draw [fill=uuuuuu]  (256,160) circle (2pt);
\draw [fill=uuuuuu]  (256,187) circle (2pt);
\draw [fill=uuuuuu]  (350,71) circle (2pt);
\draw [fill=uuuuuu] (350,98) circle (2pt);
\draw [fill=uuuuuu]  (350,123) circle (2pt);
\draw [fill=uuuuuu]  (350,143)circle (2pt);
\draw [fill=uuuuuu]  (350,167)  circle (2pt);
\draw [fill=uuuuuu]  (350,195)  circle (2pt);
\end{tikzpicture}
\end{minipage}%
\vspace{.5cm}
\begin{minipage}{0.15\linewidth}
\centering
\begin{tikzpicture}[x=0.75pt,y=0.75pt,yscale=-1,xscale=1,scale=0.45]
\draw[line width=1pt,color=sqsqsq]   (257.31,70.98) .. controls (270.32,70.88) and (281.09,101.48) .. (281.36,139.32) .. controls (281.63,177.16) and (271.3,207.91) .. (258.29,208.01) .. controls (245.28,208.1) and (234.51,177.5) .. (234.24,139.66) .. controls (233.97,101.82) and (244.3,71.07) .. (257.31,70.98) -- cycle ;
\draw[line width=1pt,color=sqsqsq]   (352.72,44.92) .. controls (369.82,44.8) and (383.98,86.11) .. (384.35,137.19) .. controls (384.71,188.27) and (371.15,229.78) .. (354.04,229.9) .. controls (336.94,230.02) and (322.78,188.71) .. (322.41,137.63) .. controls (322.05,86.55) and (335.62,45.04) .. (352.72,44.92) -- cycle ;
\draw [line width=1pt,color=sqsqsq]   (256,93) -- (350,71) ;
\draw[line width=1pt,color=sqsqsq]    (256,93) -- (350,98) ;
\draw [line width=1pt,color=sqsqsq]   (256,93) -- (350,123) ;
\draw [line width=1pt,color=sqsqsq]   (256,135) -- (350,123) ;
\draw [line width=1pt,color=sqsqsq]   (256,135) -- (350,143) ;
\draw  [line width=1pt,color=sqsqsq]  (256,135) -- (350,167) ;
\draw  [line width=1pt,color=sqsqsq]  (256,160) -- (350,167) ;
\draw[line width=1pt,color=sqsqsq]    (256,160) -- (350,195) ;
\draw [line width=1pt,color=sqsqsq]   (256,187) -- (350,195) ;
\draw [line width=1pt,color=blue] (350,71) .. controls (370, 100) and (370,130) .. (350,167)  ;
\draw [fill=uuuuuu] (256,93) circle (2pt);
\draw [fill=uuuuuu]  (256,135) circle (2pt);
\draw [fill=uuuuuu]  (256,160) circle (2pt);
\draw [fill=uuuuuu]  (256,187) circle (2pt);
\draw [fill=uuuuuu]  (350,71) circle (2pt);
\draw [fill=uuuuuu] (350,98) circle (2pt);
\draw [fill=uuuuuu]  (350,123) circle (2pt);
\draw [fill=uuuuuu]  (350,143)circle (2pt);
\draw [fill=uuuuuu]  (350,167)  circle (2pt);
\draw [fill=uuuuuu]  (350,195)  circle (2pt);
\end{tikzpicture}
\end{minipage}%
\begin{minipage}{0.15\linewidth}
\centering
\begin{tikzpicture}[x=0.75pt,y=0.75pt,yscale=-1,xscale=1,scale=0.45]
\draw[line width=1pt,color=sqsqsq]   (257.31,70.98) .. controls (270.32,70.88) and (281.09,101.48) .. (281.36,139.32) .. controls (281.63,177.16) and (271.3,207.91) .. (258.29,208.01) .. controls (245.28,208.1) and (234.51,177.5) .. (234.24,139.66) .. controls (233.97,101.82) and (244.3,71.07) .. (257.31,70.98) -- cycle ;
\draw[line width=1pt,color=sqsqsq]   (352.72,44.92) .. controls (369.82,44.8) and (383.98,86.11) .. (384.35,137.19) .. controls (384.71,188.27) and (371.15,229.78) .. (354.04,229.9) .. controls (336.94,230.02) and (322.78,188.71) .. (322.41,137.63) .. controls (322.05,86.55) and (335.62,45.04) .. (352.72,44.92) -- cycle ;
\draw [line width=1pt,color=sqsqsq]   (256,93) -- (350,71) ;
\draw[line width=1pt,color=sqsqsq]    (256,93) -- (350,98) ;
\draw [line width=1pt,color=sqsqsq]   (256,93) -- (350,123) ;
\draw [line width=1pt,color=sqsqsq]   (256,135) -- (350,123) ;
\draw [line width=1pt,color=sqsqsq]   (256,135) -- (350,143) ;
\draw  [line width=1pt,color=sqsqsq]  (256,135) -- (350,167) ;
\draw  [line width=1pt,color=sqsqsq]  (256,160) -- (350,167) ;
\draw[line width=1pt,color=sqsqsq]    (256,160) -- (350,195) ;
\draw [line width=1pt,color=sqsqsq]   (256,187) -- (350,195) ;
\draw [line width=1pt,color=blue]   (256,160) -- (350,71) ;
\draw [fill=uuuuuu] (256,93) circle (2pt);
\draw [fill=uuuuuu]  (256,135) circle (2pt);
\draw [fill=uuuuuu]  (256,160) circle (2pt);
\draw [fill=uuuuuu]  (256,187) circle (2pt);
\draw [fill=uuuuuu]  (350,71) circle (2pt);
\draw [fill=uuuuuu] (350,98) circle (2pt);
\draw [fill=uuuuuu]  (350,123) circle (2pt);
\draw [fill=uuuuuu]  (350,143)circle (2pt);
\draw [fill=uuuuuu]  (350,167)  circle (2pt);
\draw [fill=uuuuuu]  (350,195)  circle (2pt);
\end{tikzpicture}
\end{minipage}%
\begin{minipage}{0.15\linewidth}
\centering
\begin{tikzpicture}[x=0.75pt,y=0.75pt,yscale=-1,xscale=1,scale=0.45]
\draw[line width=1pt,color=sqsqsq]   (257.31,70.98) .. controls (270.32,70.88) and (281.09,101.48) .. (281.36,139.32) .. controls (281.63,177.16) and (271.3,207.91) .. (258.29,208.01) .. controls (245.28,208.1) and (234.51,177.5) .. (234.24,139.66) .. controls (233.97,101.82) and (244.3,71.07) .. (257.31,70.98) -- cycle ;
\draw[line width=1pt,color=sqsqsq]   (352.72,44.92) .. controls (369.82,44.8) and (383.98,86.11) .. (384.35,137.19) .. controls (384.71,188.27) and (371.15,229.78) .. (354.04,229.9) .. controls (336.94,230.02) and (322.78,188.71) .. (322.41,137.63) .. controls (322.05,86.55) and (335.62,45.04) .. (352.72,44.92) -- cycle ;
\draw [line width=1pt,color=sqsqsq]   (256,93) -- (350,71) ;
\draw[line width=1pt,color=sqsqsq]    (256,93) -- (350,98) ;
\draw [line width=1pt,color=sqsqsq]   (256,93) -- (350,123) ;
\draw [line width=1pt,color=sqsqsq]   (256,135) -- (350,123) ;
\draw [line width=1pt,color=sqsqsq]   (256,135) -- (350,143) ;
\draw  [line width=1pt,color=sqsqsq]  (256,135) -- (350,167) ;
\draw  [line width=1pt,color=sqsqsq]  (256,160) -- (350,167) ;
\draw[line width=1pt,color=sqsqsq]    (256,160) -- (350,195) ;
\draw [line width=1pt,color=sqsqsq]   (256,187) -- (350,195) ;
\draw [line width=1pt,color=blue]   (256,93) .. controls (245, 110) and (245,140) .. (256,160)  ;
\draw [fill=uuuuuu] (256,93) circle (2pt);
\draw [fill=uuuuuu]  (256,135) circle (2pt);
\draw [fill=uuuuuu]  (256,160) circle (2pt);
\draw [fill=uuuuuu]  (256,187) circle (2pt);
\draw [fill=uuuuuu]  (350,71) circle (2pt);
\draw [fill=uuuuuu] (350,98) circle (2pt);
\draw [fill=uuuuuu]  (350,123) circle (2pt);
\draw [fill=uuuuuu]  (350,143)circle (2pt);
\draw [fill=uuuuuu]  (350,167)  circle (2pt);
\draw [fill=uuuuuu]  (350,195)  circle (2pt);
\end{tikzpicture}
\end{minipage}%
\caption{Possible positions for the new edge.}
\label{fig:tree-plus}
\end{figure}
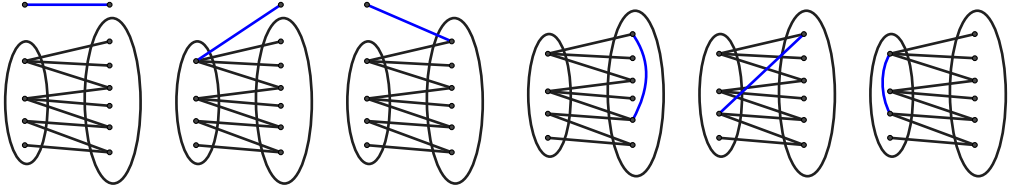
We say a graph $F$ is an  \textbf{augmentation} of a tree $T$ if $F$ can be obtained from $T$ by adding a new edge (we do not require the new edge to be contained in $V(T)$, see Figure~\ref{fig:tree-plus}). 
As another application of the stability theorem (Theorem~\ref{THM:Hypergraph-Tree-Stability}) in the previous subsection, we prove the following anti-Ramsey result for augmentations of trees. 
\begin{theorem}\label{THM:AntiRamsey-Tree-Exact}
    Suppose that $T$ is a tree satisfying $\sigma(T) = \tau_{\mathrm{ind}}(T)$ and containing a critical edge, and $F$ is an augmentation of $T$.  
    Then 
    \begin{align*}
        \mathrm{ar}(n,F^3) 
        \le \binom{n}{3} - \binom{n-\sigma(T)+1}{3} + 2
        \quad\text{for sufficiently large $n$}. 
    \end{align*}
\end{theorem}
\textbf{Remark.} 
In~\cite{GLS20}, Gu--Li--Shi determined $\mathrm{ar}(n,P_k^{r})$ and $\mathrm{ar}(n,C_k^r)$ for certain combinations of $k$ and $r$. 
Later, in~\cite{TLY22}, Tang--Li--Yan extended their result to all $r \ge 3$ and $k \ge 4$. 
Note that odd paths are strongly edge-critical, so Theorem~\ref{THM:AntiRamsey-Tree-Exact} implies (and extends) their results on $P_{2t}^3$  (see {\cite[Theorem~2.1]{GLS20}}) and $C_{2t}^3$ (see {\cite[Theorem~3]{TLY22}}). 

%
Theorem~\ref{THM:AntiRamsey-Tree-Exact}  will be proved in Section~\ref{SEC:Proof-antiRamsey}. 
\subsection{Generalized Tur\'{a}n problems}\label{SUBSEC:Intro:GenTuran}
Let $Q, F$ be two graphs. 
We use $N(Q, G)$ to denote the number of copies of $Q$ in a graph $G$.
The \textbf{generalized Tur\'{a}n number} $\mathrm{ex}(n,Q,F)$ is the maximum number of copies of $Q$ in an $n$-vertex $F$-free graph. 
When $Q = K_2$, this is the ordinary Tur\'{a}n number of $F$. 
The generalized Tur\'{a}n problem was first considered by Erd\H{o}s in~\cite{E62}, in which he determined $\mathrm{ex}(n,K_s, K_t)$ for all $t > s \ge 3$.  
In~\cite{AS16}, Alon--Shikhelman started a systematic study of $\mathrm{ex}(n,Q,F)$, and since then lots of effort has been devoted into this topic. 
Similar to the ordinary Tur\'{a}n problem, determining $\mathrm{ex}(n,Q,F)$ is much harder when $\chi(Q) \ge \chi(F)$, where $\chi(Q)$ and $\chi(F)$ denote the chromatic numbers of $Q$ and $F$, respectively. 
In this subsection, we make some progress in the case $Q=K_3$ and $F$ is the $\triangle$-blowup (defined below) of a cycle or a tree. 

Given two graphs $G_1$ and $G_2$, the \textbf{join} $G_1 \uproduct G_2$ of $G_1$ and $G_2$ is obtained by taking vertex-disjoint copies of $G_1$ and $G_2$ and adding all pairs that have nonempty intersection with both $V(G_1)$ and $V(G_2)$, \ie
\begin{align*}
    G_1 \uproduct G_2
    := \left\{uv \colon u\in V(G_1),\ v\in V(G_2)\right\} \cup G_1 \cup G_2. 
\end{align*}
We use $T(n)$ to denote the balanced complete bipartite graph on $n$ vertices and use $T^{+}(n)$ to denote the $n$-vertex graph obtained from $T(n)$ by adding one edge into the smaller part. 
Given integers $n \ge t \ge 0$ let
\begin{align*}
    S(n,t) := K_{t} \uproduct T(n-t) \quad\text{and}\quad
    S^{+}(n,t) := K_{t} \uproduct T^{+}(n-t). 
\end{align*}
Furthermore, we define two corresponding $3$-graphs as follows:
\begin{align*}
    \mathcal{S}_{\mathrm{bi}}(n,t)
    & := \left\{A \colon \text{$S(n,t)$ induces a copy of $K_3$ on $A$}\right\}, \quad\mathrm{and}\\
    \mathcal{S}^{+}_{\mathrm{bi}}(n,t)
    & := \left\{A \colon \text{$S^{+}(n,t)$ induces a copy of $K_3$ on $A$}\right\}
\end{align*}
%
%
Given a graph $F$ the \textbf{$\triangle$-blowup} $F^{\triangle}$ of $F$ is the graph obtained from $F$ be replacing each edge with a copy of $K_3$ such that no two different copies of $K_3$ use the same new vertex (view Figure~\ref{fig:Path-Cycle-Expansion} as graphs).

Our first result in this subsection is an exact determination of $\mathrm{ex}(n, K_{3}, F^{\triangle})$ when $F$ is a strongly edge-critical tree.
\begin{theorem}\label{THM:GenTuran-Tree-Exact}
     Suppose that $T$ is a tree satisfying $\sigma(T) = \tau_{\mathrm{ind}}(T)$ and containing a critical edge. 
    Then 
    \begin{align*}
        \mathrm{ex}(n, K_{3}, T^{\triangle})
        \le \left|\mathcal{S}_{\mathrm{bi}}\left(n,\sigma(T)-1\right)\right|
        \quad\text{for sufficiently large $n$}.
    \end{align*}
    In particular, if $T$ is a strongly edge-critical tree. Then 
    \begin{align*}
        \mathrm{ex}(n, K_{3}, T^{\triangle})
        = \left|\mathcal{S}_{\mathrm{bi}}\left(n,\sigma(T)-1\right)\right|
        \quad\text{for sufficiently large $n$}. 
    \end{align*}
\end{theorem}
Note that the lower bound for the 'In particular' part comes from the graph $S(n,t)$. 
In general, we have the following asymptotic bound. 
\begin{theorem}\label{THM:GenTuran-Tree-Asymp}
    Suppose that $T$ is a tree. Then 
    \begin{align*}
        \mathrm{ex}(n, K_{3}, T^{\triangle})
        =
\left|\mathcal{S}_{\mathrm{bi}}\left(n,\sigma(T)-1\right)\right| + o(n^2). 
    \end{align*}
\end{theorem}
Since odd paths are strongly edge-critical, we obtain the following corollary of Theorem~\ref{THM:GenTuran-Tree-Exact}. 
\begin{corollary}\label{CORO:GenTuran-Odd-Path-Exact}
    Suppose that $t \ge 1$ is a fixed integer and $n$ is sufficiently large. 
    Then 
    \begin{align*}
        \mathrm{ex}(n, K_{3}, P_{2t+1}^{\triangle}) 
        = \left|\mathcal{S}_{\mathrm{bi}}\left(n,t\right)\right|. 
    \end{align*}
\end{corollary}
For paths of even length and cycles, we prove the following results. 
\begin{theorem}\label{THM:GenTuran-Even-Path-Exact}
    Suppose that $t \ge 1$ is a fixed integer and $n$ is sufficiently large. 
    Then 
    \begin{align*}
        \mathrm{ex}(n, K_{3}, P_{2t+2}^{\triangle}) 
        = \left|\mathcal{S}^{+}_{\mathrm{bi}}\left(n,t\right)\right|. 
    \end{align*}
\end{theorem}
\begin{theorem}\label{THM:GenTuran-Cycle-Exact}
    Suppose that $t \ge 2$ is a fixed integer and $n$ is sufficiently large. 
    Then 
    \begin{align*}
        \mathrm{ex}(n, K_{3}, C_{2t+1}^{\triangle})
         = \left|\mathcal{S}_{\mathrm{bi}}\left(n,t\right)\right|
         \quad\mathrm{and}\quad
        \mathrm{ex}(n, K_{3}, C_{2t+2}^{\triangle})
         = \left|\mathcal{S}^{+}_{\mathrm{bi}}\left(n,t\right)\right|. 
     \end{align*}
\end{theorem}

\textbf{Remark.}
In~\cite{LGHSTVZ22}, Lv \etal determined the exact values of $\mathrm{ex}(n, K_{3}, P_{3}^{\triangle})$ and $\mathrm{ex}(n, K_{3}, C_{3}^{\triangle})$ using a very different method as we used in this paper. 
They also proposed a conjecture for the values of $\mathrm{ex}(n, K_{3}, P_{k}^{\triangle})$ and $\mathrm{ex}(n, K_{3}, C_{k}^{\triangle})$ for $k \ge 4$. 
Corollary~\ref{CORO:GenTuran-Odd-Path-Exact}, Theorem~\ref{THM:GenTuran-Even-Path-Exact}, and Theorem~\ref{THM:GenTuran-Cycle-Exact} confirm their conjecture (except for $C_{4}^{\triangle}$).  
Though the case $C_{4}^{\triangle}$ can be also handled using the same framework as in the proof of Theorem~\ref{THM:GenTuran-Cycle-Exact}, it is a bit technical and needs some extra work. So we include its proof in a separate note~\cite{S24}. 

Similar to the proof of Theorem~\ref{THM:Hypergraph-Tree-Exact}, proofs for the theorems above also use the stability method. A key step is to prove the following stability theorems. 

Given a graph $G$ and a set $L\subseteq V(G)$, we use $G-L$ to denote the induced subgraph of $G$ on $V(G)\setminus L$. 
We say an $n$-vertex graph $G$ is \textbf{$\delta$-close} to $S(n,t)$ if there exists a $t$-set $L\subseteq V(G)$ so that
\begin{itemize}
    \item $d_{G}(v) \ge (1-\delta)n$ for all $v\in L$, 
        \item $N(K_3, G-L) \le \delta n^2$, 
        \item $|G-L| \ge n^2/4 - \delta n^2$, and
        \item $G-L$ can be made bipartite by removing at most $\delta n^2$ edges.
\end{itemize}
%
\begin{theorem}[Stability]\label{THM:GenTuran-Tree-Stability}
    Let $T$ be a tree with $\sigma(T) = \tau_{\mathrm{ind}}(T)$. 
    For every $\delta>0$ there exist $\varepsilon>0$ and $n_0$ such that the following holds for all $n\ge n_0$. 
    Suppose that $G$ is an $n$-vertex $T^{\triangle}$-free graph with
    \begin{align*}
        N(K_{3}, G)
        \ge \left(\frac{\sigma(T)-1}{4}-\varepsilon\right)n^2.   
    \end{align*}
    Then $G$ is $\delta$-close to $S\left(n,\sigma(T)-1\right)$. 
\end{theorem}
Recall from the definition that $\sigma(C_k) = \floor*{(k+1)/2} = \floor*{(k-1)/2}+1$.
\begin{theorem}[Stability]\label{THM:GenTuran-Cycle-Stability}
    Let $k \ge 4$ be a fixed integer. 
    For every $\delta>0$ there exist $\varepsilon>0$ and $n_0\in \mathbb{N}$ such that the following holds for all $n \ge n_0$. 
    Suppose that $G$ is an $n$-vertex $C_{k}^{\triangle}$-free graph with 
    \begin{align*}
        N(K_{3}, G) \ge\left(\frac{\sigma(C_k)-1}{4} - \varepsilon\right)n^2.
    \end{align*}
    Then $G$ is $\delta$-close to $S\left(n,\sigma(C_k)-1\right)$. 
\end{theorem}
Since proofs for Theorems~\ref{THM:GenTuran-Tree-Exact} and~\ref{THM:GenTuran-Tree-Stability} are very similar to proofs of Theorems~\ref{THM:Hypergraph-Tree-Exact} and~\ref{THM:Hypergraph-Tree-Stability}, we include them in the Appendix. 
The proof of Theorem~\ref{THM:GenTuran-Even-Path-Exact} is similar to the proof of Theorem~\ref{THM:GenTuran-Cycle-Exact}, so we include them in the Appendix as well.
The proof of Theorem~\ref{THM:GenTuran-Cycle-Exact} is presented in Section~\ref{SEC:Proof-GenTuran-even-cycle-exact}, and the proof of Theorem~\ref{THM:GenTuran-Cycle-Stability} is included in Section~\ref{SEC:Proof-GenTuran-even-cycle-stability}. 
In the next section, we introduce some definitions and prove some preliminary results. 
\section{Preliminaries}
In this section, we present some definitions and preliminary results related to hypergraphs, trees, and expansions. 
\subsection{Hypergraphs}
Given a $3$-graph $\mathcal{H}$ the \textbf{shadow} $\partial\mathcal{H}$ of $\mathcal{H}$ is a $2$-graph defined by 
\begin{align*}
    \partial\mathcal{H} 
        := \left\{uv\in \binom{V(\mathcal{H})}{2} \colon \mathrm{\ there\ exists\ a \ } w\in V(\mathcal{H}) \mathrm{\ such\ that \ } uvw\in \mathcal{H}\right\}. 
\end{align*}
The \textbf{link} $L_{\mathcal{H}}(v)$ of a vertex $v\in \mathcal{H}$ is 
\begin{align*}
    L_{\mathcal{H}}(v)
    := \left\{uw\in \partial\mathcal{H} \colon uvw \in \mathcal{H}\right\}. 
\end{align*}
The \textbf{degree} of $v$ is $d_{\mathcal{H}}(v):= |L_{\mathcal{H}}(v)|$. 
We use $\delta(\mathcal{H})$ and $\Delta(\mathcal{H})$ to denote the minimum and maximum degree of $\mathcal{H}$, respectively. 
Given a pair of vertices $uv \subseteq V(\mathcal{H})$ the \textbf{neighborhood} of $uv$ is 
\begin{align*}
    N_{\mathcal{H}}(uv) 
        := \left\{  w\in V(\mathcal{H}) \colon uvw \in \mathcal{H}\right\}. 
\end{align*}
The \textbf{codegree} of $uv$ is $d_{\mathcal{H}}(uv) := |N_{\mathcal{H}}(uv)|$. 
We use $\delta_2(\mathcal{H})$ and $\Delta_2(\mathcal{H})$ to denote the minimum and maximum codegree of $\mathcal{H}$, respectively. 
For convenience, for every edge $uvw\in \mathcal{H}$ we let 
\begin{align*}
    \delta_2(uvw) & := \min\left\{d_{\mathcal{H}}(uv),\ d_{\mathcal{H}}(vw),\ d_{\mathcal{H}}(vw)\right\}, \quad\text{and}\quad\\
    \Delta_2(uvw) & := \max\left\{d_{\mathcal{H}}(uv),\ d_{\mathcal{H}}(vw),\ d_{\mathcal{H}}(vw)\right\}.
\end{align*}

Let $k > d \ge 0$ be integers. We say a $3$-graph $\mathcal{H}$ is \textbf{$d$-full} if $d_{\mathcal{H}}(uv) \ge d$ for all $uv\in \partial\mathcal{H}$.
We say $\mathcal{H}$ is \textbf{$(d,k)$-superfull} if it is $d$-full and every edge in $\mathcal{H}$ contains at most one pair of vertices with codegree less than $k$. 

The following simple lemma can be proved by greedily removing shadow edges with small codegree. 
\begin{lemma}[{\cite[Lemma~3.1]{KMV15a}}]\label{LEMMA:KMVa-d-full-subgraph}
    Let $d \ge 0$ be an integer.
    Every $3$-graph $\mathcal{H}$ contains a $(d+1)$-full subgraph $\mathcal{H}'$ with 
    \begin{align*}
        |\mathcal{H}'| \ge |\mathcal{H}| - d|\partial\mathcal{H}|. 
    \end{align*}
\end{lemma}
The following simple fact, which can be proved easily using a greedy argument, will be useful in our proofs. 
\begin{fact}\label{FACT:partial-embedding}
    Let $\mathcal{H}$ be a $3$-graph and $F$ be a graph with $m\ge 1$ edges.
    Suppose that $\{e_1, \ldots, e_m\} \subseteq \partial\mathcal{H}$ is a copy of $F$ and there exists $t \in [m]$ such that 
    \begin{enumerate}[label=(\roman*)]
        \item there exist distinct vertices $w_1, \ldots, w_{t} \in V(\mathcal{H}) \setminus \left(\bigcup_{i\in m}e_i\right)$  with $e_i \cup \{w_i\} \in \mathcal{H}$ for $i\in [t]$, and
        \item $d_{\mathcal{H}}(e_j) \ge 3m$ for all $j\in [t+1, m]$. 
    \end{enumerate}
    Then $F^3 \subseteq \mathcal{H}$. 
\end{fact}
A $3$-graph $\mathcal{H}$ is \textbf{$2$-intersecting} if $|e_1 \cap e_2| = 2$ for all distinct edges $e_1, e_2 \in \mathcal{H}$. 
The following observation on the structure of $2$-intersecting $3$-graphs will be useful. 
\begin{fact}\label{FACT:2-intersecting}
    Suppose that $\mathcal{H}$ is $2$-intersecting. 
    Then either $|\mathcal{H}|\le 4$ or there exists a pair $\{u,v\} \subseteq V(\mathcal{H})$ such that all edges in $\mathcal{H}$ containing $\{u,v\}$. 
\end{fact}
The following analogue for graphs is also useful. 
\begin{fact}\label{FACT:matching-1}
    Suppose that the matching number of a graph $G$ is at most one.  
    Then either $G$ consists of a triangle and possibly some isolated vertices or there exists a vertex $v \in V(G)$ such that all edges in $G$ containing $v$. 
\end{fact}

Given a graph $G$, we associate a $3$-graph with it by letting 
        \begin{align*}
            \mathcal{K}_{G} 
            := \left\{e \in \binom{V(G)}{3} \colon G[e] \cong K_3\right\}. 
        \end{align*}
The following simple fact is clear from the definition.  
\begin{fact}\label{FACT:shadow-Turan}
    Let $F$ be a graph. 
    For every $F^{\triangle}$-free graph $G$, the $3$-graph $\mathcal{K}_{G}$ is $F^{3}$-free. 
    In particular, $\mathrm{ex}(n, K_3, F^{\triangle}) \le \mathrm{ex}(n, F^3)$. 
\end{fact}
\subsection{Trees}
Let $T$ be a tree.  
We say a pair $(I, R)$ with $I \subseteq V(T)$ and $R \subseteq T$ is a \textbf{crosscut pair} if 
\begin{align*}
    \text{$I$ is an independent of $T$, \quad  $R = T-I$, \quad and\quad  $|I|+|R| = \sigma(T)$.}
\end{align*}
Recall that a vertex $v\in V(T)$ is a \textbf{leaf} if $d_{T}(v) = 1$. 
We say an edge $uv \in T$ is a \textbf{pendant} edge if $\min\{d_{T}(u), d_{T}(v)\} = 1$.
An edge $e \in T$ is \textbf{crucial} if $\sigma(T\setminus e) \le \sigma(T)-1$. 

\begin{figure}[htbp]
\centering
 \begin{minipage}[t]{.45\textwidth}
\centering
\tikzset{every picture/.style={line width=0.75pt}} 
\begin{tikzpicture}[x=0.75pt,y=0.75pt,yscale=-1,xscale=1,scale=0.8]

\draw [line width=1pt,color=sqsqsq]   (200,106) -- (260,106) ;
\draw [line width=1pt,color=red]   (260,106) -- (320,106) ;
\draw  [line width=1pt,color=sqsqsq]  (320,106) -- (380,106) ;
\draw  [line width=1pt,color=sqsqsq]  (174,157) -- (200,106) ;
\draw   [line width=1pt,color=sqsqsq] (192,157) -- (200,106) ;
\draw [line width=1pt,color=sqsqsq]   (208,157) -- (200,106) ;
\draw [line width=1pt,color=sqsqsq]   (223,157) -- (200,106) ;
\draw  [line width=1pt,color=sqsqsq]  (353,157) -- (379,106) ;
\draw  [line width=1pt,color=sqsqsq]  (371,157) -- (379,106) ;
\draw [line width=1pt,color=sqsqsq]   (387,157) -- (379,106) ;
\draw  [line width=1pt,color=sqsqsq] (402,157) -- (379,106) ;

\draw [fill=red] (200,106) circle (2pt);
\draw [fill=uuuuuu]  (260,106) circle (2pt);
\draw [fill=uuuuuu]   (320,106)  circle (2pt);
\draw [fill=red]  (380,106)  circle (2pt);

\draw [fill=uuuuuu]   (174,157) circle (2pt);
\draw [fill=uuuuuu]  (192,157) circle (2pt);
\draw [fill=uuuuuu]  (208,157)  circle (2pt);
\draw [fill=uuuuuu]  (223,157)  circle (2pt);
\draw [fill=uuuuuu]   (353,157) circle (2pt);
\draw [fill=uuuuuu]  (371,157)   circle (2pt);
\draw [fill=uuuuuu]  (387,157) circle (2pt);
\draw [fill=uuuuuu]   (402,157)  circle (2pt);
\end{tikzpicture}
\caption{A tree whose crosscut-pair $(I,R)$ is highlighted in red color.}
\label{fig:Tree-Crosscut}
\end{minipage}%
\hfill
\begin{minipage}[t]{.45\textwidth}
\centering
\tikzset{every picture/.style={line width=0.75pt}} 
\begin{tikzpicture}[x=0.75pt,y=0.75pt,yscale=-1,xscale=1,scale=0.7]

\draw[line width=1pt,color=sqsqsq]   (257.31,70.98) .. controls (270.32,70.88) and (281.09,101.48) .. (281.36,139.32) .. controls (281.63,177.16) and (271.3,207.91) .. (258.29,208.01) .. controls (245.28,208.1) and (234.51,177.5) .. (234.24,139.66) .. controls (233.97,101.82) and (244.3,71.07) .. (257.31,70.98) -- cycle ;
\draw[line width=1pt,color=sqsqsq]   (352.72,44.92) .. controls (369.82,44.8) and (383.98,86.11) .. (384.35,137.19) .. controls (384.71,188.27) and (371.15,229.78) .. (354.04,229.9) .. controls (336.94,230.02) and (322.78,188.71) .. (322.41,137.63) .. controls (322.05,86.55) and (335.62,45.04) .. (352.72,44.92) -- cycle ;
\draw [line width=1pt,color=sqsqsq]   (256,93) -- (350,71) ;
\draw[line width=1pt,color=sqsqsq]    (256,93) -- (350,98) ;
\draw [line width=1pt,color=sqsqsq]   (256,93) -- (350,123) ;
\draw [line width=1pt,color=sqsqsq]   (256,135) -- (350,123) ;
\draw [line width=1pt,color=sqsqsq]   (256,135) -- (350,143) ;
\draw  [line width=1pt,color=sqsqsq]  (256,135) -- (350,167) ;
\draw  [line width=1pt,color=sqsqsq]  (256,160) -- (350,167) ;
\draw[line width=1pt,color=sqsqsq]    (256,160) -- (350,195) ;
\draw [line width=1pt,color=green]   (256,187) -- (350,195) ;

\draw [fill=uuuuuu] (256,93) circle (2pt);
\draw [fill=uuuuuu]  (256,135) circle (2pt);
\draw [fill=uuuuuu]  (256,160) circle (2pt);
\draw [fill=uuuuuu]  (256,187) circle (2pt);

\draw [fill=uuuuuu]  (350,71) circle (2pt);
\draw [fill=uuuuuu] (350,98) circle (2pt);
\draw [fill=uuuuuu]  (350,123) circle (2pt);
\draw [fill=uuuuuu]  (350,143)circle (2pt);
\draw [fill=uuuuuu]  (350,167)  circle (2pt);
\draw [fill=uuuuuu]  (350,195)  circle (2pt);

\draw (253,220) node [anchor=north west][inner sep=0.75pt]   [align=left] {$I$};
\draw (252,193) node [anchor=north west][inner sep=0.75pt]   [align=left] {$v$};
\draw (345,202) node [anchor=north west][inner sep=0.75pt]   [align=left] {$u$};

\end{tikzpicture}
\caption{A strongly edge-critical tree in which $I$ is an independent vertex-cover and $uv$ is a critical edge.}
\label{fig:tree cross-cut}
\end{minipage}%
\end{figure}

In the proof of Theorems~\ref{THM:Hypergraph-Tree-Stability} and~\ref{THM:GenTuran-Tree-Stability} we will use the following structural result on trees. 
\begin{proposition}\label{PROP:tree-decomposition}
    Suppose that $T$ is a tree and $(I, R)$ is a cross-cut pair of $T$. Then one of the following holds. 
    \begin{enumerate}[label=(\roman*)]
        \item There exists a vertex $v\in I$ such that all but one vertex in $N_T(v)$ are leaves in $T$. 
        \item There exists an edge $e\in R$ which is a pendant edge in $T$. 
    \end{enumerate} 
    In particular, if $I$ is maximum, then $(i)$ holds. 
\end{proposition}
Proposition~\ref{PROP:tree-decomposition} can be proved easily by showing that the set system $\mathcal{T} := R \cup \left\{N_T(v) \colon v\in I\right\}$ define on $U:= V(T)\setminus I$ 
is linear and acyclic. 
We refer the interested reader to the Appendix for more detail. 

The following structural result for trees with $\sigma(T) = \tau_{\mathrm{ind}}(T)$ will be crucial for proofs of Theorems~\ref{THM:Hypergraph-Tree-Exact} and~\ref{THM:GenTuran-Tree-Exact}. 
\begin{proposition}\label{PROP:tree-R-empty}
    Let $T$ be a tree with $\sigma(T) = \tau_{\mathrm{ind}}(T)$ and $I \subseteq V(T)$ be a minimum independent vertex cover. 
    Suppose that $T$ has a critical edge. 
    Then there exists a pendant edge $e^{\ast}\in T$ whose leaf endpoint is contained in $I$ such that $\sigma(T\setminus e^{\ast}) \le \sigma(T)-1$. 
\end{proposition}
\begin{proof}
    Let $I \subseteq V(T)$ be a minimum independent vertex cover. 
    It suffices to show that $I$ contains a leaf since the edge containing this leaf will be a witness for $e^{\ast}$. 
    
    Suppose to the contrary that $d_{T}(v) \ge 2$ for all $v\in I$. 
    Fix $e\in T$ such that $\sigma(T\setminus e) \le \sigma(T)-1$.
    Let $T' = T\setminus e$. 
    For every set $S \subseteq I$, we have 
    \begin{align*}
        \sum_{v\in S}d_{T'}(v) \ge \left(\sum_{v\in S}d_{T}(v)\right) - 1 \ge 2|S|-1. 
    \end{align*}
    On the other hand, since the induced graph $T'[S\cup N_{T'}(S)]$ is a forest, we have 
    \begin{align*}
        \sum_{v\in S}d_{T'}(v) = |T'[S\cup N_{T'}(S)]| \le |S|+|N_{T'}(S)|-1. 
    \end{align*}
    Combining these two inequalities, we obtain that $|N_{T'}(S)| \ge |S|$. 
    Therefore, by Hall's theorem~\cite{Hall35}, $T'$ contains a matching which saturates $I$, and 
    it follows from the K\"{o}nig's theorem (see e.g.~\cite{Konig31,BM76}) that $\tau(T') = |I| = \sigma(T)$. 
    However, this implies that $\sigma(T') \ge \tau(T') = \sigma(T)$, contradicting the definition of $e$. 
\end{proof}
Using Proposition~\ref{PROP:tree-decomposition}, we obtain the following embedding result for the expansion of trees. 
\begin{proposition}\label{PROP:embed-T3-two-t-sets}
    Let $T$ be a tree with $k$ vertices, and $t := \sigma(T)-1$. 
    Let $\mathcal{H}$ be a $3$-graph on $V$, $S_1, S_2 \subseteq V$ are two distinct $t$-subsets, $V_1, V_2 \subseteq V\setminus(S_1\cup S_2)$ are two sets with nonempty intersection. 
    Suppose that there are two graphs $G_1$ and $G_2$ on $V_1$ and $V_2$ respectively such that for $i\in \{1,2\}$
    \begin{enumerate}[label=(\roman*)]
        \item $d_{G_i}(v) \ge 3k$ for all $v\in V_i$, and 
        \item $G_{i} \subseteq L_{\mathcal{H}}(v)$ for all $v\in S_i$.
    \end{enumerate}
    Then $T^3 \subseteq \mathcal{H}$. 
\end{proposition}
    \textbf{Remark.}
    We would like to remind the reader that for trees satisfying $\sigma(T) = \tau_{\mathrm{ind}}(T)$, assumption~(i) can be replaced by $d_{G_i}(v) \ge 1$. 
    We state this general form as it might be useful for future research.
%

\begin{figure}[htbp]
\centering
\tikzset{every picture/.style={line width=0.85pt}} 

\begin{tikzpicture}[x=0.75pt,y=0.75pt,yscale=-1,xscale=1,scale=0.9]

\draw   (180.42,72) .. controls (193.43,71.91) and (204.18,98.92) .. (204.41,132.33) .. controls (204.65,165.74) and (194.3,192.91) .. (181.29,193) .. controls (168.28,193.09) and (157.54,166.08) .. (157.3,132.67) .. controls (157.06,99.25) and (167.41,72.09) .. (180.42,72) -- cycle ;
\draw   (264.76,48.49) .. controls (281.87,48.37) and (295.99,83.87) .. (296.3,127.78) .. controls (296.62,171.69) and (283.01,207.38) .. (265.9,207.51) .. controls (248.8,207.63) and (234.68,172.13) .. (234.36,128.22) .. controls (234.05,84.31) and (247.66,48.62) .. (264.76,48.49) -- cycle ;
\draw  [line width=1pt,color=sqsqsq]  (180,96) -- (264,70) ;
\draw  [line width=1pt,color=sqsqsq]  (180,96) -- (264,86) ;
\draw  [line width=1pt,color=sqsqsq]  (180,96) -- (264,100) ;
\draw  [line width=1pt,color=sqsqsq]  (180,136) -- (264,115) ;
\draw  [line width=1pt,color=sqsqsq]  (180,136) -- (264,135) ;
\draw [line width=1pt,color=sqsqsq]   (180,136) -- (264,156) ;
\draw  [line width=1pt,color=sqsqsq]  (180,172) -- (264,172) ;
\draw [line width=1pt,color=sqsqsq]   (180,172) -- (264,156) ;
\draw  [line width=1pt,color=sqsqsq]  (180,172) -- (264,190) ;
\draw [fill=uuuuuu] (180,96) circle (2pt);
\draw [fill=uuuuuu] (180,136) circle (2pt);
\draw [fill=uuuuuu] (180,172) circle (2pt);
\draw [fill=uuuuuu] (264,70) circle (2pt);
\draw [fill=uuuuuu] (264,86) circle (2pt);
\draw [fill=uuuuuu] (264,100) circle (2pt);
\draw [fill=uuuuuu] (264,115) circle (2pt);
\draw [fill=uuuuuu] (264,135) circle (2pt);
\draw [fill=uuuuuu] (264,156) circle (2pt);
\draw [fill=uuuuuu] (264,172) circle (2pt);
\draw [fill=uuuuuu] (264,190) circle (2pt);
\draw   (384.44,136.01) .. controls (392.26,135.95) and (398.7,149.56) .. (398.82,166.41) .. controls (398.94,183.25) and (392.7,196.95) .. (384.87,197.01) .. controls (377.05,197.07) and (370.61,183.46) .. (370.49,166.61) .. controls (370.37,149.76) and (376.61,136.06) .. (384.44,136.01) -- cycle ;
\draw   (383.95,67.99) .. controls (391.78,67.93) and (398.23,83.11) .. (398.36,101.9) .. controls (398.5,120.68) and (392.26,135.95) .. (384.44,136.01) .. controls (376.61,136.06) and (370.16,120.88) .. (370.03,102.1) .. controls (369.89,83.32) and (376.13,68.04) .. (383.95,67.99) -- cycle ;
\draw  [line width=1pt,color=sqsqsq]  (264,100) .. controls (284.33,112) and (281.33,141) .. (264,156) ;
\draw [line width=1pt,color=sqsqsq]   (382,92) -- (467,48) ;
\draw [line width=1pt,color=sqsqsq]   (382,92) -- (467,65) ;
\draw  [line width=1pt,color=sqsqsq]  (382,92) -- (467,80) ;
\draw [line width=1pt,color=sqsqsq]   (382,116) -- (467,99) ;
\draw  [line width=1pt,color=sqsqsq]  (382,116) -- (467,119) ;
\draw [line width=1pt,color=sqsqsq]   (382,116) -- (467,149) ;
\draw [line width=1pt,color=sqsqsq]   (467,80) .. controls (485.33,92) and (487.33,134) .. (467,149) ;
\draw [line width=1pt,color=sqsqsq]   (382,161) -- (467,187) ;
\draw [line width=1pt,color=sqsqsq]   (382,161) -- (467,149) ;
\draw [line width=1pt,color=sqsqsq]   (382,161) -- (467,213) ;
\draw [fill=uuuuuu] (382,92) circle (2pt);
\draw [fill=uuuuuu] (382,116) circle (2pt);
\draw [fill=uuuuuu] (382,161) circle (2pt);
\draw [fill=uuuuuu] (467,48) circle (2pt);
\draw [fill=uuuuuu] (467,65) circle (2pt);
\draw [fill=uuuuuu] (467,80) circle (2pt);
\draw [fill=uuuuuu] (467,99) circle (2pt);
\draw [fill=uuuuuu] (467,119) circle (2pt);
\draw [fill=uuuuuu] (467,149) circle (2pt);
\draw [fill=uuuuuu] (467,187) circle (2pt);
\draw [fill=uuuuuu] (467,213) circle (2pt);
\draw   (455,38.47) .. controls (455,34.34) and (458.34,31) .. (462.47,31) -- (484.87,31) .. controls (488.99,31) and (492.33,34.34) .. (492.33,38.47) -- (492.33,158.53) .. controls (492.33,162.66) and (488.99,166) .. (484.87,166) -- (462.47,166) .. controls (458.34,166) and (455,162.66) .. (455,158.53) -- cycle ;
\draw   (455,140.47) .. controls (455,136.34) and (458.34,133) .. (462.47,133) -- (484.87,133) .. controls (488.99,133) and (492.33,136.34) .. (492.33,140.47) -- (492.33,232.53) .. controls (492.33,236.66) and (488.99,240) .. (484.87,240) -- (462.47,240) .. controls (458.34,240) and (455,236.66) .. (455,232.53) -- cycle ;

\draw (175,205) node [anchor=north west][inner sep=0.75pt]   [align=left] {$I$};
\draw (170,175) node [anchor=north west][inner sep=0.75pt]   [align=left] {$v_{\ast}$};
\draw (272,155) node [anchor=north west][inner sep=0.75pt]   [align=left] {$u_{\ast}$};
\draw (474,143) node [anchor=north west][inner sep=0.75pt]   [align=left] {$u_{\ast}'$};
\draw (377,171) node [anchor=north west][inner sep=0.75pt]   [align=left] {$v_{\ast}'$};
\draw (261,217) node [anchor=north west][inner sep=0.75pt]   [align=left] {$J$};
\draw (350,94) node [anchor=north west][inner sep=0.75pt]   [align=left] {$S_1$};
\draw (350,157) node [anchor=north west][inner sep=0.75pt]   [align=left] {$S_2$};
\draw (506,90) node [anchor=north west][inner sep=0.75pt]   [align=left] {$V_1$};
\draw (506,181) node [anchor=north west][inner sep=0.75pt]   [align=left] {$V_2$};
\end{tikzpicture}
\caption{Embedding a tree.}
\label{fig:PROP-embed-tree}
\end{figure}
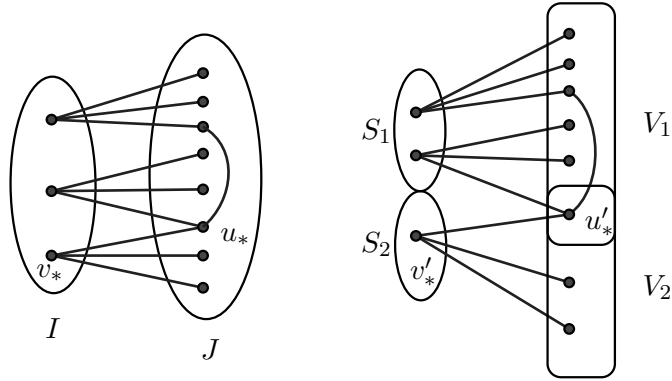

\begin{proof}[Proof of Proposition~\ref{PROP:embed-T3-two-t-sets}]
    Let $(I,R)$ be a crosscut pair of $T$ such that $|I|$ is maximized.
    By Proposition~\ref{PROP:tree-decomposition}, there exists a vertex $v_{\ast} \in I$ such that all but one vertex $u_{\ast} \in N_T(v_{\ast})$ are leaves. 
    Let $N'(v_{\ast}) := N_{T}(v_{\ast}) \setminus\{u_{\ast}\}$, $T' := T-\left(\{v_{\ast}\} \cup N'(v_{\ast})\right)$, $I':= I\setminus \{v_{\ast}\}$, $J:= V(T)\setminus I$, and $J':= J\setminus N_{T}(v_{\ast})$. 
    Observe that $|I'|+|R| = |I|+|R|-1 = t$. 
    
    Fix $v_{\ast}' \in S_2 \setminus S_1$, $u_{\ast}' \in V_1 \cap V_2$, and a $\left(d_{T}(v_{\ast})-1\right)$-set $N' \subseteq V_2 \setminus \{u_{\ast}'\}$. 
    Let $V_1':= V_1\setminus N'$. 
    It follows from the assumptions that the induced bipartite graph of $\partial\mathcal{H}$ on $S_1 \cup V_1$ is complete. 
    Combined with $\delta(G_1[V_1']) \ge \delta(G_1)-d_{T}(v_{\ast})-1 \ge 3k-k\ge 2k$ and a simple greedy argument, it is easy to see that there exists an embedding
    $\phi \colon V(T') \to S_1 \cup V_1'$ such that $\phi(I')  \subseteq S_1$, $\phi\left(T'-I'\right) \subseteq G_1[V_1']$, and moreover, $\phi\left(u_{\ast}\right) = u_{\ast}'$. 
    We can extend $\phi$ by letting $\phi(v_{\ast}) = v_{\ast}'$ and $\phi\left(N'(v_{\ast})\right) = N'$ to obtain an embedding of $T$ into $\partial\mathcal{H}$ (see Figure~\ref{fig:PROP-embed-tree}). 
    Observe that for every edge $e\in R$ we have $\phi(e) \subseteq G_1[V_1']$, and $|S_1| - |\phi(I')| = t-|I'| = |R|$. 
    Thus, assigning each vertex in $S_1\setminus \phi(I')$ to an edge $e\in R$ we obtain an embedding of $R^3$ into $\mathcal{H}$. 
    In addition, notice that every pair in $S_1 \times V_1$ and $S_2 \times V_2$ has codegree at least $\min\{\delta(G_1),\ \delta(G_2)\} \ge 3k$, so it follows from Fact~\ref{FACT:partial-embedding} that $T^3 \subseteq \mathcal{H}$. 
\end{proof}
The following simple and crude upper bound for $\mathrm{ex}(n,K_3,F^{\triangle})$ will be crucial for the proofs of generalzied Tur\'{a}n theorems. 
Its proof is a standard application of the K\"{o}vari--S\'{o}s--Tur\'{a}n Theorem~\cite{KST54}, the Triangle Removal Lemma (see e.g.~\cite{RS78,AFKS00,Fox11}), and Simonovits' stability theorem~\cite{S68}. 
We refer the interested reader to the Appendix for more detail. 
%
\begin{proposition}\label{PROP:shadow-size}
    Let $F$ be a bipartite graph or an odd cycle. 
    For every $\delta>0$ there exists $\varepsilon>0$ such that for sufficiently large $n$, every $n$-vertex $F^{\triangle}$-free graph $G$ with at least $n^2/4 - \varepsilon n^2$ edges can be made bipartite by removing at most $\delta n^2$ edges.
    In particular, every $n$-vertex $F^{\triangle}$-free graph $G$ satisfies  $|G| \le n^2/4 + o(n^2)$. 
\end{proposition}
\subsection{Expansion of trees}
For convenience, we use $\mathcal{T}_k$ to denote the family of all trees with $k$ edges. 
The following lemma can be proved using a simple greedy argument. 
\begin{lemma}[{\cite[Lemma~3.2]{KMV15a}}]\label{LEMMA:KMVa-3k-full-subgraph-tree}
    Let $k \ge 3$ be an integer. 
    Every $3k$-full nonempty $3$-graph $\mathcal{H}$ contains $T_{k}^{3}$ for all $T_k \in \mathcal{T}_k$. 
\end{lemma}
The following bound for the number of edges in an $n$-vertex $T_k^{3}$-free $3$-graph with bounded codegree can be proved easily using the argument for~{\cite[Proposition~3.8]{KMV15a}}.
We refer the interested reader to the Appendix for more detail. 
%
\begin{lemma}\label{LEMMA:KMVa-lambda-sparse-tree}
    Fix integers $k \ge 3$ and $C \ge 1$. 
    Let  $T_k \in \mathcal{T}_k$. 
    If $\mathcal{H}$ is an $n$-vertex $T_k^{3}$-free $3$-graph with $\Delta_2(\mathcal{H}) \le C$, then $|\mathcal{H}| \le 6Ck n= o(n^2)$.
\end{lemma}
Though not specifically stated in~\cite{KMV17b}, the following useful lemma can be proved using the argument for~{\cite[Theorem~1.1]{KMV17b}} (see pages~467 and 468 and, in particular, Claim~4 in ~\cite{KMV17b}).
\begin{lemma}[{\cite{KMV17b}}]\label{LEMMA:l+1-shadows}
        Let $T$ be a tree with $\sigma(T) = \tau_{\mathrm{ind}}(T)$. 
        For every constant $\varepsilon >0$ there exists $n_0$ such that the following holds for all $3$-graphs $\mathcal{H}$ with $n \ge n_0$ vertices. 
      If there exists $E\subseteq \partial \mathcal{H}$ with $|E| \ge \varepsilon n^{2}$ and $d_{\mathcal{H}}(e) \ge \sigma(T)$ for all $e\in E$,  then $T^{3} \subseteq \mathcal{H}$. 
\end{lemma}
The proof for~{\cite[Theorem~1.1]{KMV17b}} implies the following lemma (see {\cite[Section~6]{KMV17b}}). 
\begin{lemma}[{\cite{KMV17b}}]\label{LEMMA:l+1-full-tree}
    Let $T$ be a tree. 
    Every $n$-vertex $\sigma(T)$-full $T^3$-free $3$-graph $\mathcal{H}$ satisfies $|\mathcal{H}| = o(n^2)$. 
\end{lemma}
The following proposition follows easily from 
Lemmas~\ref{LEMMA:KMVa-d-full-subgraph} and~\ref{LEMMA:l+1-full-tree}.  
\begin{proposition}\label{PROP:shadow-bound-tree}
    Let $T$ be a tree and $\varepsilon > 0$ be a constant. 
    For sufficiently large $n$, every $n$-vertex $T_{k}^{3}$-free $3$-graph $\mathcal{H}$ satisfies 
    \begin{align*}
        |\mathcal{H}|
        \le \left(\sigma(T)-1\right) |\partial\mathcal{H}| + \varepsilon n^2. 
    \end{align*}
\end{proposition}
In the proof of~{\cite[Theorem~6.2]{KMV15a}}, Kostochka--Mubayi--Verstra\"{e}te introduced the following process to obtain a large $\left(\floor*{(k-1)/2}, 3k\right)$-superfull subgraph in a $C_{k}^3$-free (or $P_k^3$-free) $3$-graph.
We will use it in proofs of Theorems~\ref{THM:Hypergraph-Tree-Stability},~\ref{THM:GenTuran-Tree-Stability} and~\ref{THM:GenTuran-Cycle-Stability}. 
\begin{algorithm}
\renewcommand{\thealgorithm}{}
\caption{Cleaning Algorithm}\label{ALGO:cleaning}
\begin{algorithmic}

\State \textbf{Input:}  
A triple $(\mathcal{H}, k, t)$, where $\mathcal{H}$ is a $3$-graph and $k \ge t \ge 0$ are integers. 

\State \textbf{Output:} An integer $q$, a sequence of edges $f_1, \ldots, f_q \in \partial\mathcal{H}$, and a sequence of subgraphs $\mathcal{H} \supseteq \mathcal{H}_0 \supseteq \cdots \supseteq \mathcal{H}_q$. 

\State \textbf{Initialization:} 
    Let 
    \begin{align*}
        i=0,\quad
        \mathcal{H}^{\ast} 
            := \left\{e \in \mathcal{H} 
            \colon \Delta_2(e) \le 3k\right\},\quad\text{and}\quad 
            \mathcal{H}_0 := \mathcal{H}\setminus \mathcal{H}^{\ast}
    \end{align*}

\State \textbf{Operation:} 
            For every edge $e \in \partial\mathcal{H}_i$, we say $e$ is of  
            \begin{itemize}
                \item type-1 if $d_{\mathcal{H}_i}(e) \le t-1$,
                \item type-2 if $d_{\mathcal{H}_i}(e) = t$ and there exists $f' \in \partial\mathcal{H}_i$ such that 
                \begin{align*}
                    |f'\cap e| = 1,\quad d_{\mathcal{H}_i}(f') = t,\quad \text{and}\ e \cup f' \in \mathcal{H}_i, 
                \end{align*}
                \item type-3 if $t+1 \le d_{\mathcal{H}_i}(e) \le 3k-1$.
            \end{itemize}
            If $\partial\mathcal{H}_i$ contains no edges of types-$1\sim 3$, then we let $q=i$ and stop. 
            Otherwise, we choose an edge $e\in \partial\mathcal{H}_i$ of minimum type. Let 
            \begin{align*}
                f_{i+1}:= e \quad\text{and}\quad
                \mathcal{H}_{i+1} 
                := \left\{E\in \mathcal{H}_i \colon e\nsubseteq E\right\}, 
            \end{align*}
            and then repeat this process to $\mathcal{H}_{i+1}$. 
\end{algorithmic}
\end{algorithm}

Repeating the proof of~{\cite[Theorem~6.2]{KMV15a}} to a tree with $\sigma(T) = \tau_{\mathrm{ind}}(T)$, it is easy to obtain the following lemma. 
We refer the interested reader to the Appendix for more detail. 
%
\begin{lemma}\label{LEMMA:Cleaning-Algo-tree}
    Let $T$ be a fixed tree with $\sigma(T) = \tau_{\mathrm{ind}}(T) =: t+1$ and $k:= |T|$.
    Fix constant $\varepsilon >0$ and let $n$ be sufficiently large. 
    Let $\mathcal{H}$ be an $n$-vertex $T^3$-free $3$-graph and $\mathcal{H}_q$ be the outputting $3$-graph of Cleaning Algorithm with input $(\mathcal{H}, k, t)$ as defined here.
    If $|\mathcal{H}| \ge t |\partial\mathcal{H}| - \varepsilon n^2$,  
    then $\mathcal{H}_q$ is $(t, 3k)$-superfull, and moreover, 
    \begin{align*}
        q 
        \le 12k \varepsilon n^2, \quad 
        |\mathcal{H}_q| 
        \ge |\mathcal{H}| - 48k^2 \varepsilon n^2, \quad\text{and}\quad
        |\partial\mathcal{H}_q| \ge |\partial\mathcal{H}| - 50k^2\varepsilon n^2. 
    \end{align*}
\end{lemma}
\subsection{Expansion of cycles}
In line with the previous subsection, we adopt the subsequent lemmas on $C_{k}^3$ from~\cite{KMV15a}. 
\begin{lemma}[{\cite[Lemma~3.2]{KMV15a}}]\label{LEMMA:KMVa-3k-full-subgraph-cycle}
    Let $k \ge 3$ be an integer. 
    Every $3k$-full nonempty $3$-graph $\mathcal{H}$ contains $C_{k}^{3}$. 
\end{lemma}
\begin{lemma}[{\cite[Proposition~3.8]{KMV15a}}]\label{LEMMA:KMVa-lambda-sparse-cycle}
   Let $k \ge 3$ and $C \ge 1$  be fix integers. 
   Every $n$-vertex $C_{k}^{3}$-free $3$-graph $\mathcal{H}$ with $\Delta_2(\mathcal{H}) \le C$ satisfies $|\mathcal{H}| = o(n^2)$.
\end{lemma}
%
%
\begin{lemma}[{\cite[Theorem~6.1~(i)]{KMV15a}}]\label{LEMMA:KMVa-ell+1-full}
    Let $k \ge 4$ be a fixed integer. 
    Every $n$-vertex $\sigma(C_k)$-full $C_{k}^3$-free $3$-graph $\mathcal{H}$ satisfies $|\mathcal{H}| = o(n^2)$. 
\end{lemma}
Similar to Proposition~\ref{PROP:shadow-bound-tree}, 
the following proposition follows easily from Lemmas~\ref{LEMMA:KMVa-d-full-subgraph} and~\ref{LEMMA:KMVa-ell+1-full}.
\begin{proposition}\label{APPENDIX:PROP:Turan-cycle-shadow}
     Let $k \ge 4$ be an integer and $\varepsilon > 0$ be a constant. 
    For sufficiently large $n$, every $n$-vertex $C_{k}^{3}$-free $3$-graph $\mathcal{H}$ satisfies 
    \begin{align*}
        |\mathcal{H}|
        \le \left(\sigma(C_k) - 1\right) |\partial\mathcal{H}| + \varepsilon n^2. 
    \end{align*}
\end{proposition}
\begin{lemma}[{\cite[Lemma~5.1]{KMV15a}}]\label{LEMMA:KMVa-many-ell+1-shadow-cycle}
    Let $\varepsilon >0$ and $k \ge 4$ be fixed and $n$ be sufficiently large. 
    Suppose that $\mathcal{H}$ is an $n$-vertex $3$-graph such that there exists $E\subseteq \partial\mathcal{H}$ with $|E| \ge \varepsilon n^2$ so that 
    \begin{enumerate}[label=(\roman*)]
        \item for even $k$, every $uv \in E$ satisfies $d_{\mathcal{H}}(uv) \ge \sigma(C_k)$, 
        \item for odd $k$,  every $uv \in E$ satisfies $d_{\mathcal{H}}(uv) \ge \sigma(C_k)$, and there exists $w\in V(\mathcal{H})$ with $uvw \in \mathcal{H}$ such that 
        \begin{align*}
            \min\left\{d_{\mathcal{H}}(uw),\ d_{\mathcal{H}}(vw)\right\} \ge 2
            \quad\mathrm{and}\quad 
            \max\left\{d_{\mathcal{H}}(uw),\ d_{\mathcal{H}}(vw)\right\} \ge 3k. 
        \end{align*}
    \end{enumerate}
    Then $C_{k}^3 \subseteq \mathcal{H}$. 
\end{lemma}
The proof of~{\cite[Theorem~6.2]{KMV15a}} implies the following lemma. 
We refer the interested reader to the Appendix for more detail. 
%
\begin{lemma}\label{LEMMA:Cleaning-Algo-cycle}
    Let $k \ge 4$ and $t = \floor*{(k-1)/2}$ be integers.
    Fix a constant $\varepsilon >0$ and let $n$ be sufficiently large. 
    Let $\mathcal{H}$ be an $n$-vertex $C_{k}^3$-free $3$-graph and $\mathcal{H}_q$ be the outputting $3$-graph of the Cleaning Algorithm with input $(\mathcal{H}, k, t)$ as defined here.
    If $|\mathcal{H}| \ge t |\partial\mathcal{H}| - \varepsilon n^2$,  
    then $\mathcal{H}_q$ is $(t, 3k)$-superfull, and moreover, 
    \begin{align*}
        q 
        \le 12k \varepsilon n^2, \quad 
        |\mathcal{H}_q| 
        \ge |\mathcal{H}| - 48k^2 \varepsilon n^2, \quad\text{and}\quad
        |\partial\mathcal{H}_q| \ge |\partial\mathcal{H}| - 50k^2\varepsilon n^2. 
    \end{align*}
\end{lemma}
In the proof of Theorem~\ref{THM:GenTuran-Cycle-Stability}, we need the following strengthen of~{\cite[Lemma~3.5]{KMV15a}}. 
Since two proofs are essentially the same, we omit its proof and refer the interested reader to the Appendix for more detail. 
%
\begin{lemma}\label{LEMMA:KMVa-superfull-complete-bipartite-cycle}
    Let $k \ge 4$, $t := \left\lfloor(k-1)/2\right\rfloor$, and $\mathcal{H}$ be a $t$-superfull $3$-graph containing two disjoint sets $W_1, W_2 \subseteq V(\mathcal{H})$ each of size at least $3k$ such that every pair of vertices $(w_1, w_2) \in W_1 \times W_2$ has codegree exactly $t$ in $\mathcal{H}$. 
    If $C_k^3 \not\subseteq \mathcal{H}$, then there exists a $t$-set $L\subseteq V(\mathcal{H}) \setminus (W_1\cup W_2)$ such that $N_{\mathcal{H}}(w_1w_2) = L$ for all pairs $(w_1, w_2) \in W_1 \times W_2$. 
\end{lemma}
\section{Proofs for hypergraph Tur\'{a}n results}\label{SEC:Proof-Hypergraph-Turan}
We prove Theorems~\ref{THM:Hypergraph-Tree-Exact} and~\ref{THM:Hypergraph-Tree-Stability} in this section. 
We will first prove the stability theorem (Theorem~\ref{THM:Hypergraph-Tree-Stability}) and then use it to prove the exact result (Theorem~\ref{THM:Hypergraph-Tree-Exact}). 
\subsection{Proof of Theorem~\ref{THM:Hypergraph-Tree-Stability}}\label{SUBSEC:Proof-Hypergraph-Turan-Stability}
Building upon the foundation laid out in the proof of~{\cite[Theorem~6.2]{KMV15a}}, we begin with a near-extremal $T^{3}$-free $3$-graph $\mathcal{H}$, and apply the Cleaning Algorithm to obtain a $(t,3k)$-superfull subgraph $\mathcal{H}'\subseteq \mathcal{H}$ by removing a negligible number $o(n^2)$ of edges. 
After that, our proof diverges from the prior proof of~{\cite[Theorem~6.2]{KMV15a}} due to the failure of a crucial lemma (Lemma~\ref{LEMMA:KMVa-superfull-complete-bipartite-cycle}) in the context of trees. 
To handle this technical difficulty, we first iteratively remove a 
small amount of edges from $\mathcal{H}'$ to shrink its vertex covering number to at most $\sqrt{n}$ (see Claim~\ref{equ:THM:Hypergraph-Tree-Stability-4}). 
Subsequently, we turn our attention to the common links of $t$-subsets within a covering set $Z$ of the remaining $3$-graph, which still contains most edges in $\mathcal{H}$.
Using properties of trees (Proposition~\ref{PROP:embed-T3-two-t-sets}), we show that only one $t$-set in $Z$ has a large common link, and this particular $t$-set is the object we are looking for. 
\begin{proof}[Proof of Theorem~\ref{THM:Hypergraph-Tree-Stability}]
    Let $T$ be a tree with $\sigma(T) = \tau_{\mathrm{ind}}(T)$.
    Fix $\delta <0$. 
    Let $0 \le \varepsilon \ll \varepsilon_1 \ll \varepsilon_2 \ll \varepsilon_3 \ll \delta$ be sufficiently small constant, and let $n$ be sufficiently large. 
    Let $k := |T|$ and $t := \sigma(T) - 1$. 
    Let $\mathcal{H}$ be a $T^{3}$-free $3$-graph on $n$ vertices with at least $(t/2 -\varepsilon) n^2$ edges.
    For convenience, let $V:= V(\mathcal{H})$. 
    Our aim is to show that $\mathcal{H}$ is $\delta$-close to $\mathcal{S}(n,t)$. 

    First, observe from Proposition~\ref{PROP:shadow-bound-tree} that 
    \begin{align}\label{equ:THM:Hypergraph-Tree-Stability-1}
        |\partial\mathcal{H}|
        \ge \frac{|\mathcal{H}|-\varepsilon n^2}{t}
        \ge \frac{n^2}{2} - 2\varepsilon n^2. 
    \end{align}
    Let $\mathcal{H}_q$ be the output $3$-graph of the Cleaning Algorithm with input $(\mathcal{H}, k, t)$ as defined here. 
    Since $|\mathcal{H}| \ge t n^2/2-\varepsilon n^2 \ge t |\partial\mathcal{H}| - \varepsilon n^2$, 
    it follows from Lemma~\ref{LEMMA:Cleaning-Algo-tree} that $\mathcal{H}_q$ is $(t, 3k)$-superfull and
    \begin{align}\label{equ:THM:Hypergraph-Tree-Stability-2}
        q 
            \le 12k \varepsilon n^2, \quad 
        |\mathcal{H}_q| 
            \ge |\mathcal{H}| - 48k^2 \varepsilon n^2,  \quad\text{and}\quad
        |\partial\mathcal{H}_q| 
            \ge |\partial\mathcal{H}| - 50k^2\varepsilon n^2.
    \end{align}
    Define graphs 
    \begin{align*}
        G' 
            := \left\{e\in \partial\mathcal{H}_q \colon d_{\mathcal{H}_q}(e) = t\right\}, 
            \quad\text{and}\quad
        G''
            := \left\{e\in \partial\mathcal{H}_q \colon d_{\mathcal{H}_q}(e) \ge 3k\right\}. 
    \end{align*}
    It follows from the definition of the Cleaning Algorithm that $G' \cup G'' = \partial\mathcal{H}_q$ and every edge in $\mathcal{H}_q$ contains at least two elements from $G''$. 
    \begin{claim}\label{CLAIM:Turan-tree-l-shadow}
        We have $|G''| \le kn$ and $|G'| \ge \left(1/2 - \varepsilon_1\right)n^2$. 
    \end{claim}
    \begin{proof}
        It follows from Fact~\ref{FACT:partial-embedding} that $G''$ is $T$-free, Therefore, $|G''| \le kn$. 
        In addition, by~\eqref{equ:THM:Hypergraph-Tree-Stability-1} and the third inequality in~\eqref{equ:THM:Hypergraph-Tree-Stability-2}, we obtain 
        \begin{align*}
            |G'| 
            = |\partial\mathcal{H}_q|- |G''| 
            \ge \frac{n^2}{2}-2\varepsilon n^2 -  50k^2\varepsilon n^2 - \varepsilon n^2
            \ge \frac{n^2}{2}-\varepsilon_1 n^2. 
        \end{align*}
        This proves Claim~\ref{CLAIM:Turan-tree-l-shadow}. 
    \end{proof}
    Define a $3$-graph $\mathcal{G} \subseteq \mathcal{H}_q$ as follows:
    \begin{align*}
        \mathcal{G} := \left\{e\in \mathcal{H}_q \colon \left|\binom{e}{2} \cap G'\right| = 1 \right\}. 
    \end{align*}
    Since every pair in $G'$ contributes exactly $t$ triples to $\mathcal{G}$, 
    it follows from Claim~\ref{CLAIM:Turan-tree-l-shadow} that 
    \begin{align}\label{equ:THM:Hypergraph-Tree-Stability-3}
        |\mathcal{G}|
            \ge t |G'|
            \ge \left(\frac{1}{2} - \varepsilon_1\right)tn^2. 
    \end{align}
    For every vertex $v\in V$ define the \textbf{light link} of $v$ as 
    \begin{align*}
        \widehat{L}(v) := \left\{e\in G' \colon e\cup \{v\} \in \mathcal{H}_q\right\}. 
    \end{align*}
    It is easy to see from the definitions that 
    \begin{align}\label{equ:THM:Hypergraph-Tree-Stability-4}
        |\mathcal{G}| = \sum_{v\in V} |\widehat{L}(v)|, \quad\text{and}\quad
        \widehat{L}(v) \subseteq \binom{N_{G''}(v)}{2} \quad\text{for all}\quad v\in V(\mathcal{H}), 
    \end{align}
    where recall that $N_{G''}(v)$ is the set of neighbors of $v$ in $G''$. 
    For every integer $i\ge 0$ define 
    \begin{align*}
        Z_i := \left\{v\in V \colon |\widehat{L}(v)| \ge \frac{n^{2-2^{-i}}}{(\log n)^2}\right\} 
        \quad\text{and}\quad
        U_i:= V\setminus Z_i.  
    \end{align*}
    Additionally, let 
    \begin{align*}
        \mathcal{G}_0 
        \coloneqq \left\{e\in \mathcal{G} \colon |e\cap Z_0|=1 \text{ and } e\setminus Z_{0}\in G^\prime\right\},
        \quad\text{and}\quad
        \mathcal{G}_i \coloneqq \left\{e\in \mathcal{G}_{i-1} \colon |e\cap Z_i|=1 \right\} 
    \end{align*}
    for $i \ge 1$. 
    Observe that for all $i\ge 1$, 
    \begin{align*}
        V \supseteq Z_0 \supseteq Z_1 \supseteq \cdots \supseteq Z_{i}, \quad
        U_0 \subseteq U_1 \subseteq \cdots \subseteq U_i \subseteq V, \quad\text{and}\quad
        \mathcal{G} \supseteq \mathcal{G}_0 \supseteq \mathcal{G}_1 \supseteq \cdots \supseteq \mathcal{G}_i. 
    \end{align*}
    \begin{claim}\label{CLAIM:GenTuran-tree-Zk}
        We have $|Z_i| \le 2kn^{2^{-i-1}} \log n$ and $|\mathcal{G}_i|\ge |\mathcal{G}| - \frac{2(i+1)n^2}{(\log n)^2}$ for all $i \ge 0$. 
        In particular, 
        \begin{align*}
            \binom{|Z_k|}{t} \le \sqrt{n} 
            \quad\text{and}\quad
            |\mathcal{G}_k| 
                \ge \left(\frac{1}{2} - 2\varepsilon_1\right)tn^2. 
        \end{align*}
    \end{claim}
    \begin{proof}
        We prove it by induction on $i$. 
        For the base case $i=0$, 
        observe from the second part of~\eqref{equ:THM:Hypergraph-Tree-Stability-4} that 
        for every $v\in Z_0$, we have $d_{G''}(v) \ge \sqrt{|\widehat{L}(v)|} \ge n^{1/2}/\log n$ (as the worst case is that $\widehat{L}(v)$ is a complete graph on $d_{G''}(v)$ vertices). 
        Therefore, by Claim~\ref{CLAIM:Turan-tree-l-shadow}, 
        \begin{align*}
            2kn
            \ge 2|G''| 
            \ge \sum_{v\in Z_0}d_{G''}(v) 
            \ge \sum_{v\in Z_0}\sqrt{|\widehat{L}(v)|} 
            \ge |Z_0|\frac{n^{1/2}}{\log n}, 
        \end{align*}
        which implies that $|Z_0| \le 2kn^{1/2}\log n$. 
        Consequently, the number of edges in $G'$ that have nonempty intersection with $Z_0$ is at most $|Z_0|n \le 2kn^{3/2}\log n$.
        Observe that every triple in $\mathcal{G}$ with at least two vertices in $Z_0$ must contain a pair (in $G'$) with (at least) one endpoint in $Z_0$. Therefore, the number of edges in $\mathcal{G}$ with at least two vertices in $Z_0$ is at most $2k t n^{3/2}\log n$ (recall that each edge in $G'$ contributes $t$ triples in $\mathcal{G}$). 
        Therefore, 
        \begin{align*}
            |\mathcal{G}_0| 
                 \ge  |\mathcal{G}| - \sum_{v \in V\setminus Z_0}|\widehat{L}(v)| -2k t n^{3/2}\log n 
                 & \ge |\mathcal{G}| - n \times \frac{n}{(\log n)^2} - 2k t n^{3/2}\log n \\
                 & \ge |\mathcal{G}| - \frac{2n^2}{(\log n)^2}. 
        \end{align*}
        Now suppose that $|Z_i| \le 2kn^{2^{-i-1}}\log n$ holds for some $i \ge 0$.
        Repeating the argument above to $Z_{i+1}$, we obtain 
        \begin{align*}
            2kn
            \ge 2|G''| 
            \ge \sum_{v\in Z_{i+1}}d_{G''}(v) 
            \ge \sum_{v\in Z_{i+1}}\sqrt{|\widehat{L}(v)|} 
            \ge |Z_{i+1}|\frac{n^{1-2^{-i-2}}}{\log n}, 
        \end{align*}
        which implies that $|Z_{i+1}| \le 2kn^{2^{-i-2}}\log n$.
        So it follows from the inductive hypothesis that 
        \begin{align*}
            |\mathcal{G}_{i+1}| 
                 \ge  |\mathcal{G}_i| - \sum_{v \in Z_i\setminus Z_{i+1}}|\widehat{L}(v)| -t|Z_{i+1}|n
                 & \ge |\mathcal{G}_i| - |Z_i|\times \frac{n^{2-2^{-i-1}}}{(\log n)^2} -  2ktn^{1+2^{-i-2}}\log n\\
                 & \ge |\mathcal{G}_i| - \frac{2n^2}{(\log n)^2}
                 \ge |\mathcal{G}| - \frac{2(i+2)n^2}{(\log n)^2}. 
        \end{align*}
        Here, $t|Z_{i+1}|n$ is an upper bound for the number of edges in $\mathcal{G}_{i}$ with at least two vertices in $Z_{i+1}$. 
    \end{proof}
    Let $\mathcal{G}_k'$ denote the outputting $3$-graph of the Cleaning Algorithm with input $\left(\mathcal{G}_k, k, t\right)$ as defined here. 
    Since $|\mathcal{G}_k| \ge t n^2/2 - 2\varepsilon_1 tn^2$, it follows from Lemma~\ref{LEMMA:Cleaning-Algo-tree} that 
    \begin{align*}
        |\mathcal{G}_k'|
        \ge \left(\frac{t}{2} - \varepsilon_2\right) n^2. 
    \end{align*}
    Recall from the definition of $\mathcal{G}_k$ that every edge in $\mathcal{G}_k'$ contains exactly one vertex in $Z_{k}$. 
    Similar to the proof of Claim~\ref{CLAIM:Turan-tree-l-shadow}, the number of pairs in $\binom{U_k}{2}$ that have codegree exactly $t$ in $\mathcal{G}_k'$ is at least $(1/2-\varepsilon_3)n^2 - |Z_k|n \ge (1/2-2\varepsilon_3)n^2$. 
    Therefore, the $3$-graph 
    \begin{align*}
        \mathcal{G}_k''
            := \left\{e \in \mathcal{G}_k' \colon  \text{the pair } \binom{e}{2}\cap \binom{U_k}{2} \text{ has codegree exactly $t$ }\right\}
    \end{align*}
    satisfies 
    \begin{align}\label{equ:THM:Hypergraph-Tree-Stability-5}
        |\mathcal{G}_k''|
        \ge \left(\frac{1}{2}-2\varepsilon_3\right) t n^2. 
    \end{align}
    For every set $T \subseteq Z_k$, let $L_{\mathcal{G}_k''}(T) := \bigcap_{v\in T}L_{\mathcal{G}_k''}(v)$ denote the link of $T$. 
    Observe that for distinct $t$-sets $T_1, T_2 \subseteq Z_{k}$, the two link graphs $L_{\mathcal{G}_k''}(T_1)$ and $L_{\mathcal{G}_k''}(T_2)$ are edge disjoint, and moreover, $\bigcup_{T\in \binom{Z_k}{t}} L_{\mathcal{G}_k''}(T)$ is a partition of the induced subgraph of $\partial\mathcal{G}_{k}''$ on $U_k$. 
    Therefore, we have 
    \begin{align}\label{equ:sum-link-Z-t-set}
        \sum_{T\in \binom{Z_k}{t}}|L_{\mathcal{G}_k''}(T)| = {|\mathcal{G}_k''|}/{t}. 
    \end{align}
    Define 
    \begin{align*}
        \mathcal{Z} := \left\{T\in \binom{Z_k}{t} \colon |L_{\mathcal{G}_k''}(T)| \ge 4kn\right\}.  
    \end{align*} 
    For every $T\in \mathcal{Z}$ let $L'(T)$ be a maximum induced subgraph of $L_{\mathcal{G}_k''}(T)$ with minimum degree at least $3k$. 
    By greedily removing vertex with degree less than $3k$ we have  
    \begin{align}\label{equ:Hypergraph-tree-stability-L'T}
        |L'(T)| \ge |L_{\mathcal{G}_k''}(T)| - 3kn >0. 
    \end{align} 
    In addition, it follows from~\eqref{equ:sum-link-Z-t-set} and Claim~\ref{CLAIM:GenTuran-tree-Zk} that 
    \begin{align}\label{equ:sum-link-mathcalZ}
        \sum_{T\in \mathcal{Z}}|L_{\mathcal{G}_k''}(T)| 
        \ge {|\mathcal{G}_k''|}/{t} - 4kn^{3/2}.  
    \end{align}
    Let $\mathrm{Supp}_T$ denote the set of vertices in the graph $L'(T)$ with positive degree. 
    The following claim follows easily from Proposition~\ref{PROP:embed-T3-two-t-sets}. 
    \begin{claim}\label{CLAIM:Turan-disjoint-support}
        We have $\mathrm{Supp}_T \cap \mathrm{Supp}_{T'}= \emptyset$ for all distinct sets $T, T' \in \mathcal{Z}$. 
    \end{claim}
    %
    Since $L'(T) \subseteq \binom{\mathrm{Supp}_T}{2}$ for all $T \in \mathcal{Z}$ and $\sum_{T\in \mathcal{Z}}|L_{\mathcal{G}_k''}(T)| \ge |\mathcal{G}_k''|/t - 4kn^{3/2}$ (see~\eqref{equ:sum-link-mathcalZ}), it follows from~\eqref{equ:THM:Hypergraph-Tree-Stability-5} and~\eqref{equ:Hypergraph-tree-stability-L'T} that 
    \begin{align*}
        \sum_{T\in \mathcal{Z}} \binom{|\mathrm{Supp}_T|}{2}
        \ge \sum_{T\in \mathcal{Z}}|L'(T)| 
        \ge \sum_{T\in \mathcal{Z}}\left(|L_{\mathcal{G}_k''}(T)|-3kn\right)
        & \ge |\mathcal{G}_k''|/t  - 4kn^{3/2}- 3kn|\mathcal{Z}| \\
        & \ge \left(\frac{1}{2}-3\varepsilon_3\right) n^2. 
    \end{align*}
    It follows from Claim~\ref{CLAIM:Turan-disjoint-support} that $\sum_{T\in \mathcal{Z}} |\mathrm{Supp}_T| \le |U_k| \le n$. 
    Combined with the inequality above and using the fact $\sum_{i\ge 1} \binom{x_i}{2} \le \binom{x_1}{2} + \binom{\sum_{i\ge 2}x_i}{2}$, we obtain 
    \begin{align*}
        \max\left\{|\mathrm{Supp}_T| \colon T \in \mathcal{Z}\right\} 
        \ge \left(1-\sqrt{6\varepsilon_3}\right)n. 
    \end{align*}
    Let $T_{\ast} \in \mathcal{Z}$ be the (unique) $t$-set with $|\mathrm{Supp}_{T_{\ast}}|\ge \left(1-\sqrt{6\varepsilon_3}\right)n$. 
    Then, by Claim~\ref{CLAIM:Turan-disjoint-support}, 
    \begin{align*}
        |L_{\mathcal{G}_k''}(T_{\ast})|
        \ge |\mathcal{G}_{k}''|/t - 3kn^{3/2} - \sum_{T\in \mathcal{Z}-T_{\ast}}|L_{\mathcal{G}_k''}(T)|
        & \ge \left(\frac{1}{2}-2\varepsilon_3\right) n^2 -3kn^{3/2} - \binom{\sqrt{6\varepsilon_3}n}{2} \\
        & \ge \left(\frac{1}{2}-\frac{\delta}{2t} \right) n^2. 
    \end{align*}
    Let $L \coloneqq T_{\ast}$. 
    It follows from the definition of 
    $L_{\mathcal{G}_k''}(T_{\ast})$ that for every $v \in L$, 
    \begin{align*}
        d_{\mathcal{H}}(v)
        = |L_{\mathcal{H}}(v)|
        \ge |L_{\mathcal{G}_{k}''}(v)|
        \ge |L_{\mathcal{G}_k''}(T_{\ast})|
        \ge \left(\frac{1}{2}-\frac{\delta}{2t} \right) n^2. 
    \end{align*}
    Additionally, by combining this with Theorem~\ref{THM:KMV-tree} that 
    \begin{align*}
        |\mathcal{H} - L|
        \le |\mathcal{H}| - \sum_{v\in L}d_{\mathcal{H}}(v)  
        \le \left(\frac{t}{2} +o(1)\right)n^2 - t \cdot \left(\frac{1}{2}-\frac{\delta}{2t} \right) n^2 
        \le \delta n^2.
    \end{align*}
    This completes the proof of Theorem~\ref{THM:Hypergraph-Tree-Stability}. 
\end{proof}
%
\subsection{Proof of Theorem~\ref{THM:Hypergraph-Tree-Exact}}\label{SUBSEC:Proof-Hypergraph-Turan-Exact}
In this subsection, we use Theorem~\ref{THM:Hypergraph-Tree-Stability} to prove Theorem~\ref{THM:Hypergraph-Tree-Exact}. 
\begin{proof}[Proof of Theorem~\ref{THM:Hypergraph-Tree-Exact}]
    Let $T$ be a tree with $\sigma(T) = \tau_{\mathrm{ind}}(T)$.
    Let $I\cup J = V(T)$ be a partition such that $I$ is a minimum independent vertex cover of $T$. 
    Let $e_{\ast} := \{u_0, v_0\}$ be a pendant critical edge such that $v_0 \in I$ is a leaf (the existence of such an edge is guaranteed by Proposition~\ref{PROP:tree-R-empty}).  
    Choose $C>0$ sufficiently large such that $\mathrm{ex}(N, T^3) \le C N^2/2$ holds for all integers $N \ge 0$.
    By the theorem of Kostochka--Mubayi--Verstra\"{e}te~\cite{KMV17b} (or Theorem~\ref{THM:Hypergraph-Tree-Stability}), such a constant $C$ exists. 
    Let $0 < \delta \ll \delta_1 \ll C^{-1}$ be sufficiently small, $n \gg C$ be sufficiently large,  $t := \sigma(T) - 1$,  and $q := |\mathcal{S}(n,t)| = \binom{n}{3} - \binom{n-t}{3}$. 
    Let $\mathcal{H}$ be an $n$-vertex $T^3$-free $3$-graph with $q$ edges. 
    Our aim is to show that $\mathcal{H} \cong \mathcal{S}(n,t)$. 

    Let $V:= V(\mathcal{H})$. 
    Since $|\mathcal{H}| = q = (t-o(1))n^2/2$, 
    by Theorem~\ref{THM:Hypergraph-Tree-Stability}, there exists a $t$-set $L:= \{x_1, \ldots, x_{t}\} \subseteq V$ such that 
    \begin{align}\label{equ:Turan-vtx-L}
        d_{\mathcal{H}}(x_i) \ge (1-\delta)\frac{n^2}{2}
        \quad\text{for all}\quad 
        x_i \in L. 
    \end{align}

    Let $V' := V \setminus L$, 
    \begin{align*}
        \tau := \frac{n}{200}, \quad
        D 
        := \left\{y\in V' \colon d_{\mathcal{H}}(y x_i) \ge \tau \text{ for all } x_i \in L \right\}, 
        \quad\text{and}\quad 
        \overline{D} := V' \setminus D. 
    \end{align*}

    Define 
    \begin{align*}
        \mathcal{S} := \left\{e \in \binom{V}{3} \colon |e \cap L| \ge 1\right\}, \quad
        \mathcal{B} := \mathcal{H} - L = \mathcal{H}[V'], 
        \quad\mathrm{and}\quad
        \mathcal{M} := \mathcal{S}\setminus  \mathcal{H}. 
    \end{align*}
    Let $m:= |\mathcal{M}|$ and $b:= |\mathcal{B}|$. 
    Notice that we are done if $b = 0$. So we may assume that $b \ge 1$. 
    Our aim is to show that $b < m$. 
    \begin{claim}\label{CLAIM:Turan-size-D}
        We have $m \ge n |\overline{D}|/6$ and $|\overline{D}| \le 3\delta t n$. 
    \end{claim}
    \begin{proof}
        It follows from the definition of $D$ that for every vertex $u\in \overline{D}$, there exists $x_i \in L$ such that the pair $\{u,x_i\}$ contributes at least $n - 2 -\tau$ members (called missing edges) in $\mathcal{M}$. 
        Therefore, we have 
            \begin{align}\label{equ:Turan-missing}
                m 
                \ge  \frac{1}{3}(n-2- \tau)|\overline{D}| 
                \ge \frac{n |\overline{D}|}{6}. 
            \end{align}
        On the other hand, it follows from~\eqref{equ:Turan-vtx-L} that 
        \begin{align*}
            m 
            \le |\mathcal{S}| - \sum_{x_i \in L}d_{\mathcal{H}}(x_i)
            \le t \left(\binom{n}{2} - (1-\delta)\frac{n^2}{2}\right)
            \le  \frac{\delta t n^2}{2}. 
        \end{align*}
        Combined with~\eqref{equ:Turan-missing}, we obtain 
        \begin{align*}
            |\overline{D}|
            \le \frac{\delta t n^2}{2 \times n/6} 
            = 3\delta t n. 
        \end{align*}
        This proves Claim~\ref{CLAIM:Turan-size-D}. 
    \end{proof}
    
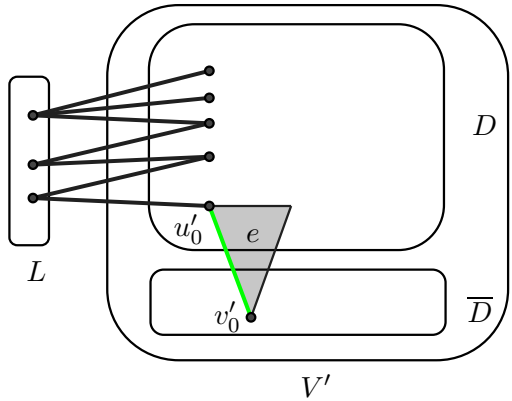
\begin{figure}[htbp]
\centering
\tikzset{every picture/.style={line width=0.85pt}} 
\begin{tikzpicture}[x=0.75pt,y=0.75pt,yscale=-1,xscale=1,scale=0.8]

\draw   (169.06,60.7) .. controls (171.78,60.7) and (173.99,62.91) .. (173.99,65.63) -- (173.99,161.73) .. controls (173.99,164.46) and (171.78,166.67) .. (169.06,166.67) -- (154.26,166.67) .. controls (151.53,166.67) and (149.32,164.46) .. (149.32,161.73) -- (149.32,65.63) .. controls (149.32,62.91) and (151.53,60.7) .. (154.26,60.7) -- cycle ;
\draw   (391.93,27.7) .. controls (407.62,27.7) and (420.33,40.41) .. (420.33,56.1) -- (420.33,141.3) .. controls (420.33,156.98) and (407.62,169.7) .. (391.93,169.7) -- (264.72,169.7) .. controls (249.04,169.7) and (236.32,156.98) .. (236.32,141.3) -- (236.32,56.1) .. controls (236.32,40.41) and (249.04,27.7) .. (264.72,27.7) -- cycle ;
\draw   (415.67,15.67) .. controls (440.34,15.67) and (460.33,35.66) .. (460.33,60.33) -- (460.33,194.33) .. controls (460.33,219) and (440.34,239) .. (415.67,239) -- (255,239) .. controls (230.33,239) and (210.33,219) .. (210.33,194.33) -- (210.33,60.33) .. controls (210.33,35.66) and (230.33,15.67) .. (255,15.67) -- cycle ;
\draw   (413.13,182) .. controls (417.66,182) and (421.33,185.67) .. (421.33,190.2) -- (421.33,214.8) .. controls (421.33,219.33) and (417.66,223) .. (413.13,223) -- (245.52,223) .. controls (241,223) and (237.32,219.33) .. (237.32,214.8) -- (237.32,190.2) .. controls (237.32,185.67) and (241,182) .. (245.52,182) -- cycle ;
\fill[sqsqsq, fill opacity=0.25]  (300,212) -- (274,142)-- (325,142) -- cycle;
\draw[line width=1.5pt,color=green]   (300,212) -- (274,142) ;
\draw[color=sqsqsq]    (274,142) -- (325,142) ;
\draw[color=sqsqsq]   (300,212)  -- (325,142)  ;
\draw[line width=1.5pt,color=sqsqsq]    (164,85) -- (274,57) ;
\draw[line width=1.5pt,color=sqsqsq]    (164,85) -- (274,74) ;
\draw[line width=1.5pt,color=sqsqsq]    (164,85) -- (274,90) ;
\draw[line width=1.5pt,color=sqsqsq]    (164,116) -- (274,90) ;
\draw[line width=1.5pt,color=sqsqsq]    (164,116) -- (274,111) ;
\draw[line width=1.5pt,color=sqsqsq]    (164,137) -- (274,111) ;
\draw[line width=1.5pt,color=sqsqsq]    (164,137) -- (274,142) ;
\draw [fill=uuuuuu] (300,212) circle (2pt);
\draw [fill=uuuuuu] (274,142) circle (2pt);
%
\draw [fill=uuuuuu] (164,85) circle (2pt);
\draw [fill=uuuuuu] (274,57) circle (2pt);
\draw [fill=uuuuuu] (274,74) circle (2pt);
\draw [fill=uuuuuu] (274,90) circle (2pt);
\draw [fill=uuuuuu] (164,116) circle (2pt);
\draw [fill=uuuuuu] (164,137) circle (2pt);
\draw [fill=uuuuuu] (274,111) circle (2pt);
%
\draw (158,175) node [anchor=north west][inner sep=0.75pt]   [align=left] {$L$};
\draw (433,197) node [anchor=north west][inner sep=0.75pt]   [align=left] {$\overline{D}$};
\draw (436,83) node [anchor=north west][inner sep=0.75pt]   [align=left] {$D$};
\draw (329,246) node [anchor=north west][inner sep=0.75pt]   [align=left] {$V'$};
\draw (250,145) node [anchor=north west][inner sep=0.75pt]   [align=left] {$u_0'$};
\draw (295,155) node [anchor=north west][inner sep=0.75pt]   [align=left] {$e$};
\draw (275,200) node [anchor=north west][inner sep=0.75pt]   [align=left] {$v_0'$};
\end{tikzpicture}
\caption{Embedding the tree $T$.}
\label{fig:hypergraph-turan-tree-1}
\end{figure}
    
    \begin{claim}\label{CLAIM:Turan-size-b-1}
        We have $\mathcal{B} \setminus \mathcal{B}\left[\overline{D}\right] = \emptyset$.
    \end{claim}
    \begin{proof}
        Suppose to the contrary that there exists an edge $e \in \mathcal{B} \setminus \mathcal{B}\left[\overline{D}\right]$. 
        Then fix $u_0' \in e\cap D$ and $v_0' \in e\setminus \{u_0'\}$. 
        Let $T' := T-v_0$, $I':= I-v_0$, and $D' := D\setminus (e-u_0')$. 
        Since the induced bipartite subgraph of $\partial\mathcal{H}$ on $L \cup D'$ is complete, 
        there exists an embedding $\psi \colon T' \to \partial\mathcal{H}[L, D']$ such that $\psi(I') = L$, $\psi(J) \subseteq D$, and $\psi(u_0) = u_0'$. 
        Since every pair in $\partial\mathcal{H}[L, D]$ has codegree at least $\tau \ge 3k$, by Fact~\ref{FACT:partial-embedding}, the map $\psi$ can be greedily extended to be an embedding $\hat{\psi} \colon T^3 \to \mathcal{H}$ with $\hat{\psi}(e_{\ast}^3) = e$ (see Figure~\ref{fig:hypergraph-turan-tree-1}), a contradiction.
        Here, recall that $e_{\ast}^3$ is the expansion of $e_{\ast}$. 
    \end{proof}

    Observe that the $3$-graph $\mathcal{B}\left[\overline{D}\right]$ is $T^3$-free, so we have $|\mathcal{B}\left[\overline{D}\right]| \le C |\overline{D}|^2$. 
    It follows from Claim~\ref{CLAIM:Turan-size-b-1} that 
    \begin{align}\label{equ:Turan-bad}
        b 
        = |\mathcal{B}| 
        = |\mathcal{B}\left[\overline{D}\right]| 
        \le C |\overline{D}|^2. 
    \end{align}
    By Claim~\ref{CLAIM:Turan-size-D} and~\eqref{equ:Turan-bad}, we obtain  
    \begin{align*}
        b 
        \le  C |\overline{D}|^2
        \le  C |\overline{D}| \times 3\delta t n
        < \frac{n |\overline{D}|}{6}
        = m. 
    \end{align*}
    Therefore, 
    \begin{align*}
        q
        = |\mathcal{H}|
        = |\mathcal{S}| + |\mathcal{B}| - |\mathcal{M}|
        \le q + b - m
        < q, 
    \end{align*}
    a contradiction. 
    This proves Theorem~\ref{THM:Hypergraph-Tree-Exact}. 
\end{proof}
\section{Proofs for anti-Ramsey results}\label{SEC:Proof-antiRamsey}
%
We prove Theorem~\ref{THM:AntiRamsey-Tree-Exact} in this section. 
Let us first present some useful lemmas. 

Using Lemma~\ref{LEMMA:KMVa-ell+1-full} and the theorem of  Kostochka--Mubayi--Verstra\"{e}te~\cite{KMV17b} on linear cycles (or Theorem~\ref{THM:Hypergraph-Tree-Stability}), it is not hard to obtain the following crude bound for $\mathrm{ex}(n, F^3)$ when $F$ is an augmentation of a tree.  
We omit its proof and refer the interested reader to the Appendix for more detail. 
%
\begin{proposition}\label{PROP:Turan-augmentation}
        Suppose that $F$ is an augmentation of a tree $T$.  
        Then $\mathrm{ex}(n, F^3) \le 3\left(|T|+1\right) n^2$ holds for sufficiently large $n$. 
\end{proposition}
The following lemma is a stability result for augmentations of trees. 
\begin{lemma}\label{LEMMA:anti-Ramsey-stability}
    Let $T$ be a tree with $\sigma(T) = \tau_{\mathrm{ind}}(T)$ and $F$ be an augmentation of $T$. 
    For every $\delta > 0$ there exist $\varepsilon >0$ and $n_0$ such that the following holds for all $n \ge n_0$. 
    Suppose that $\chi \colon K_n^3 \to \mathbb{N}$ is an edge-coloring without rainbow copies of $F^3$ and $\mathcal{H} \subseteq K_n^3$ is a rainbow subgraph with 
    \begin{align*}
        |\mathcal{H}| 
        \ge \left(\frac{\sigma(T) - 1}{2} - \varepsilon\right)n^2. 
    \end{align*}
    Then there exists a $(\sigma(T)-1)$-set $L \subseteq [n]$ such that 
    \begin{align*}
        d_{\mathcal{H}}(v) \ge \left(\frac{1}{2} - \delta\right)n^2
        \quad\text{for all}\quad v\in L. 
    \end{align*}
\end{lemma}
\begin{proof}[Proof of Lemma~\ref{LEMMA:anti-Ramsey-stability}]
    Let $\{e_{\ast}\} = F\setminus T$ and
    let $x,y$ denote the endpoints of $e_{\ast}$ in $F$. Let $\widehat{T}$ be the graph obtained from $F$ by removing the edge $e_{\ast}$ but keeping the endpoints of $e_{\ast}$ (since $e_{\ast}$ might be outside $V(T)$, two graphs $T$ and $\widehat{T}$ can be different).
    Let $\mathcal{C}$ be a maximal collection of edge disjoint copies of $\widehat{T}^{3}$ in $\mathcal{H}$.
    Assume that $\mathcal{C} = \{\widehat{T}^{3}_{1}, \ldots, \widehat{T}^{3}_{m}\}$ for some integer $m\ge 0$. 
    For $i\in [m]$ let $x_i, y_i$ be two distinct vertices in $\widehat{T}^{3}_i$ such that there exists an isomorphism $\psi_{i} \colon \widehat{T}^{3} \to \widehat{T}_i^3$ with $\psi_{i}(x) = x_i$ and $\psi_{i}(y) = y_i$.
    We call $\{x_i, y_i\}$ the \textbf{critical pair} of $\widehat{T}^{3}_i$ for $i\in [m]$. 
    Let $G$ denote the multi-graph on $[n]$ whose edge set is the multi-set $\{x_iy_i \colon i\in [m]\}$.
    Since the coloring $\chi$ does not contain rainbow copies of $F^3$, we obtain the following claim. 
    \begin{claim}\label{CLAIM:Anti-1}
        For every $i\in [m]$ and every vertex $z\in [n]\setminus V(\widehat{T}_i^{3})$ we have $\chi(zx_iy_i) \in \chi\left(\widehat{T}_i^{3}\right)$.
    \end{claim}
    \begin{claim}\label{CLAIM:Anti-2}
        The graph $G$ does not contain multi-edges. 
    \end{claim}
    \begin{proof}
        Suppose to the contrary that $G$ contains multi-edges, and by symmetry, we may assume that $(x_1, y_1) = (x_2, y_2)$. 
        Take any vertex $z$ outside $V(\widehat{T}_1^{3}) \cup V(\widehat{T}_2^{3})$. 
        Then by Claim~\ref{CLAIM:Anti-1}, we have $\chi(x_1y_1z) = \chi(x_2y_2z) \in \chi\left(\widehat{T}_1^{3}\right) \cap \chi\left(\widehat{T}_2^{3}\right) = \emptyset$ (note that $\mathcal{H}$ is rainbow, and $\widehat{T}_1^{3}, \widehat{T}_2^{3} \subset \mathcal{H}$ are edge-disjoint), a contradiction. 
    \end{proof}
    \begin{claim}\label{CLAIM:Anti-3}
        We have $m = |G| \le 2|V(F)|n$.
    \end{claim}
    \begin{proof}
        Assume that the set of neighbors of $x$ in $F$ is $\{w_1, \ldots, w_{\ell}\} = N_{F}(x)$, where $\ell = d_{F}(x)$ and $w_1 = y$.
        Let $\{z_1, \ldots, z_{\ell}\}$ be a set of new vertices not contained in $V(F)$, and define a new graph 
        \begin{align*}
            \widehat{F}
             = \left(F - x\right) \cup \left\{w_iz_i \colon i\in [\ell]\right\},
        \end{align*}
        Observe that $\widehat{F}$ is acyclic (i.e. a forest) and $|V(\widehat{F})| \le 2|V(F)|$.
        
        Suppose to the contrary that $|G| \ge 2|V(F)|n$. 
        Then, by the trivial bound for the Tur\'{a}n number of a forest, there exists an embedding $\phi \colon \widehat{F} \to G$. 
        For every $e\in \phi(\widehat{F})$ we fix a member in $\mathcal{C}$ whose critical pair is $e$, and denote this member by $\widehat{T}_{e}$.
        Let $B_1 := \bigcup_{e\in \phi(\widehat{F})} V(\widehat{T}_{e})$. 
        Choose a vertex set $U_1 := \{w_{e} \colon e\in \phi(F - x)\} \subseteq [n]\setminus B_1$. 
        Fix a vertex $v\in [n]\setminus (B_1 \cup U_1)$. 
        Then, by Claim~\ref{CLAIM:Anti-1}, the $3$-graph 
        \begin{align*}
            F' := \left\{e\cup w_e \colon e\in \phi(F - x)\right\} \cup \left\{v\phi(w_i)\phi(z_i) \colon i\in [\ell]\right\}
        \end{align*}
        is rainbow under $\chi$. 
        Observe that $F'$ is a copy of $F^{3}$ (with vertex $v$ playing the role of $x$), contradicting the rainbow-$F$-freeness of $\chi$.
    \end{proof}
    By Claim~\ref{CLAIM:Anti-3}, the $3$-graph $\mathcal{B} := \bigcup_{i\in [m]} \widehat{T}_{i}^{3}$ has size $m\times |F|\le  2|V(F)||F|n \le \varepsilon n^2$.
    Hence, $\mathcal{H}' := \mathcal{H}\setminus \mathcal{B}$ has size at least $\left(\frac{\sigma(T) - 1}{2} - 2\varepsilon\right)n^2$.  
    In addition, 
    by the maximality of $\mathcal{C}$, the $3$-graph $\mathcal{H}'$ is $\widehat{T}^3$-free, and hence, $T^3$-free. 
    Therefore, it follows from Theorem~\ref{THM:Hypergraph-Tree-Stability} that there exists a $(\sigma(T)-1)$-set $L \subseteq [n]$ such that $d_{\mathcal{H}}(v) \ge \left(1/2 - \delta\right)n^2$ holds for all $v\in L$. 
\end{proof}
Now we are ready to prove Theorem~\ref{THM:AntiRamsey-Tree-Exact}. 
\begin{proof}[Proof of Theorem~\ref{THM:AntiRamsey-Tree-Exact}]
    Let $T$ be a tree satisfying $\sigma(T) = \tau_{\mathrm{ind}}(T)$ and containing a critical edge. Let $F$ be an augmentation of $T$. 
    Let $\{f\}:= F\setminus T$ and assume that $f = \{u_1,v_1\}$. 
    Let $I\cup J = V(T)$ be a partition such that $I$ is a minimum independent vertex cover of $T$. 
    Let $e_{\ast} := \{u_0, v_0\}$ be a pendant critical edge such that $v_0 \in I$ is a leaf. The existence of such an edge is guaranteed by Proposition~\ref{PROP:tree-R-empty}.   

    Choose $C>0$ sufficiently large such that $\mathrm{ex}(N, F^3) \le C N^2$ holds for all integers $N \ge 0$. 
    By Proposition~\ref{PROP:Turan-augmentation}, such a constant $C$ exists and depends only on $F$. 
    Let $t := \sigma(T) - 1$ and $q := |\mathcal{S}(n,t)|+2= \binom{n}{3} - \binom{n-t}{3}+2$. 
    Let $\chi \colon K_n^3 \to [q]$ be a surjective map.
    We aim to show that there exists a rainbow copy of $F^3 \subseteq K_{n}^3$ under the coloring $\chi$. 

    Suppose to the contrary that there is no rainbow copy of $F^3$ under the coloring $\chi$. 
    Let $\mathcal{H} \subseteq K_n^3$ a rainbow subgraph with $q$ edges. 
    Since $|\mathcal{H}| = q =  \left(t - o(1)\right)n^2/2$, 
    it follows from Lemma~\ref{LEMMA:anti-Ramsey-stability} that there exists a $t$-set $L:= \{x_1, \ldots, x_t\} \subseteq V:=[n]$ such that 
    \begin{align}\label{equ:AntiRamsey-deg-L}
        d_{\mathcal{H}}(x_i) \ge \left(\frac{1}{2} - \delta\right)n^2
        \quad\text{for all}\quad x_i \in L. 
    \end{align}
    Let $V' := V \setminus L$, 
    \begin{align*}
        \tau := \frac{n}{200},\quad
        D 
        := \left\{y\in V' \colon d_{\mathcal{H}}(y x_i) \ge \tau \text{ for all } x_i \in L \right\}, 
        \quad\text{and}\quad 
        \overline{D} := V' \setminus D. 
    \end{align*}
    Define 
    \begin{align*}
        \mathcal{S} := \left\{e \in \binom{V}{3} \colon |e \cap L| \ge 1\right\}, \quad
        \mathcal{B} := \mathcal{H} - L = \mathcal{H}[V'], 
        \quad\mathrm{and}\quad
        \mathcal{M} := \mathcal{S}\setminus  \mathcal{H}. 
    \end{align*}
    Let $m:= |\mathcal{M}|$ and $b:= |\mathcal{B}|$. 
   %
    %
    \begin{claim}\label{CLAIM:Ramsey-missing}
        We have $m \ge n|\overline{D}|/6$ and $|\overline{D}| \le  6\delta t n$. 
    \end{claim}
    \begin{proof}
        Same as the proof of Claim~\ref{CLAIM:Turan-size-D}. 
    \end{proof}
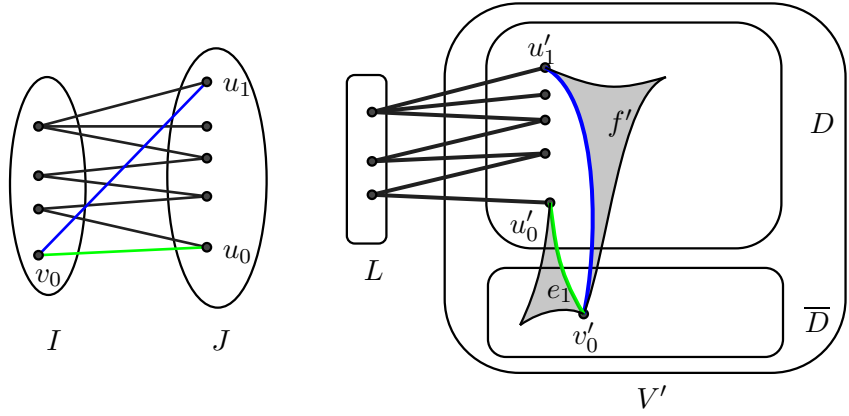
\begin{figure}[htbp]
\centering
\tikzset{every picture/.style={line width=0.85pt}} 
\begin{tikzpicture}[x=0.75pt,y=0.75pt,yscale=-1,xscale=1, scale=0.8]

\draw   (314.06,64.7) .. controls (316.78,64.7) and (318.99,66.91) .. (318.99,69.63) -- (318.99,165.73) .. controls (318.99,168.46) and (316.78,170.67) .. (314.06,170.67) -- (299.26,170.67) .. controls (296.53,170.67) and (294.32,168.46) .. (294.32,165.73) -- (294.32,69.63) .. controls (294.32,66.91) and (296.53,64.7) .. (299.26,64.7) -- cycle ;
\draw   (536.93,31.7) .. controls (552.62,31.7) and (565.33,44.41) .. (565.33,60.1) -- (565.33,145.3) .. controls (565.33,160.98) and (552.62,173.7) .. (536.93,173.7) -- (409.72,173.7) .. controls (394.04,173.7) and (381.32,160.98) .. (381.32,145.3) -- (381.32,60.1) .. controls (381.32,44.41) and (394.04,31.7) .. (409.72,31.7) -- cycle ;
\draw   (558.87,19.67) .. controls (584.53,19.67) and (605.33,40.47) .. (605.33,66.13) -- (605.33,205.53) .. controls (605.33,231.2) and (584.53,252) .. (558.87,252) -- (401.8,252) .. controls (376.14,252) and (355.33,231.2) .. (355.33,205.53) -- (355.33,66.13) .. controls (355.33,40.47) and (376.14,19.67) .. (401.8,19.67) -- cycle ;
\draw   (555.13,186) .. controls (561.32,186) and (566.33,191.01) .. (566.33,197.2) -- (566.33,230.8) .. controls (566.33,236.99) and (561.32,242) .. (555.13,242) -- (393.52,242) .. controls (387.34,242) and (382.32,236.99) .. (382.32,230.8) -- (382.32,197.2) .. controls (382.32,191.01) and (387.34,186) .. (393.52,186) -- cycle ;
\draw [line width=1.5pt,color=sqsqsq]   (310,88) -- (418,60) ;
\draw  [line width=1.5pt,color=sqsqsq]  (310,88) -- (418,77) ;
\draw  [line width=1.5pt,color=sqsqsq]  (310,88) -- (418,93) ;
\draw  [line width=1.5pt,color=sqsqsq]  (310,119) -- (418,93) ;
\draw  [line width=1.5pt,color=sqsqsq]  (310,119) -- (418,114) ;
\draw [line width=1.5pt,color=sqsqsq]   (310,140) -- (418,114) ;
\draw  [line width=1.5pt,color=sqsqsq]  (310,140) -- (421,145) ;
\draw [fill=uuuuuu] (310,88) circle (2pt);
\draw [fill=uuuuuu] (310,119) circle (2pt);
\draw [fill=uuuuuu] (310,140) circle (2pt);
\draw [fill=uuuuuu] (418,60) circle (2pt);
\draw [fill=uuuuuu] (418,77) circle (2pt);
\draw [fill=uuuuuu] (418,93) circle (2pt);
\draw [fill=uuuuuu] (418,114) circle (2pt);
\draw [fill=uuuuuu] (421,145) circle (2pt);
\draw [fill=uuuuuu] (442,215) circle (2pt);
\fill[sqsqsq, fill opacity=0.25] (418,60) .. controls (454.33,80) and (450.33,178) .. (442,215) .. controls (457.33,176) and  (464.33,78) .. (493.33,66) .. controls (464.33,74) and (455.33,76) .. (418,60) -- cycle;
\draw  [line width=1.5pt,color=blue]  (418,60) .. controls (454.33,80) and (450.33,178) .. (442,215) ;
\draw    (418,60) .. controls (455.33,76) and (464.33,74) .. (493.33,66) ;
\draw    (493.33,66) .. controls (464.33,78) and (457.33,176) .. (442,215) ;
\draw  [line width=1.5pt,color=green]  (421,145) .. controls (424.33,168) and (424.33,186) .. (442,215) ;
\draw    (421,145) .. controls (418.33,178) and (413.33,199) .. (402.33,222) ;
\draw    (442,215) .. controls (432.33,210) and (416.33,212) .. (402.33,222) ;
\fill[sqsqsq, fill opacity=0.25] (421,145) .. controls (424.33,168) and (424.33,186) .. (442,215) .. controls (432.33,210) and (416.33,212) .. (402.33,222) .. controls (413.33,199) and (418.33,178) .. (421,145) -- cycle;
\draw   (107.31,65.98) .. controls (120.32,65.88) and (131.09,96.48) .. (131.36,134.32) .. controls (131.63,172.16) and (121.3,202.91) .. (108.29,203.01) .. controls (95.28,203.1) and (84.51,172.5) .. (84.24,134.66) .. controls (83.97,96.82) and (94.3,66.07) .. (107.31,65.98) -- cycle ;
\draw   (212.72,47.92) .. controls (229.82,47.8) and (243.95,84.2) .. (244.27,129.23) .. controls (244.59,174.26) and (230.99,210.87) .. (213.89,210.99) .. controls (196.78,211.11) and (182.66,174.71) .. (182.33,129.68) .. controls (182.01,84.65) and (195.62,48.04) .. (212.72,47.92) -- cycle ;
\draw  [line width=1pt,color=sqsqsq]  (102,97) -- (207,69) ;
\draw  [line width=1pt,color=sqsqsq]  (102,97) -- (207,97) ;
\draw  [line width=1pt,color=sqsqsq]  (102,97) -- (207,117) ;
\draw [line width=1pt,color=sqsqsq]   (102,128) -- (207,117) ;
\draw [line width=1pt,color=sqsqsq]   (102,128) -- (207,141) ;
\draw  [line width=1pt,color=sqsqsq]  (102,149) -- (207,141) ;
\draw  [line width=1pt,color=sqsqsq]  (102,149) -- (207,173) ;
\draw [line width=1pt,color=green]   (102,178) -- (207,173) ;
\draw [line width=1pt,color=blue]   (102,178) -- (207,69) ;
\draw [fill=uuuuuu] (102,97) circle (2pt);
\draw [fill=uuuuuu] (207,69) circle (2pt);
\draw [fill=uuuuuu] (207,97) circle (2pt);
\draw [fill=uuuuuu] (207,117) circle (2pt);
\draw [fill=uuuuuu] (102,128) circle (2pt);
\draw [fill=uuuuuu] (102,149) circle (2pt);
\draw [fill=uuuuuu] (102,178) circle (2pt);
\draw [fill=uuuuuu] (207,173) circle (2pt);
\draw [fill=uuuuuu] (207,141) circle (2pt);
%
\draw (303,179) node [anchor=north west][inner sep=0.75pt]   [align=left] {$L$};
\draw (578,208) node [anchor=north west][inner sep=0.75pt]   [align=left] {$\overline{D}$};
\draw (581,87) node [anchor=north west][inner sep=0.75pt]   [align=left] {$D$};
\draw (473,258) node [anchor=north west][inner sep=0.75pt]   [align=left] {$V'$};
\draw (394,148) node [anchor=north west][inner sep=0.75pt]   [align=left] {$u_0'$};
\draw (433,218) node [anchor=north west][inner sep=0.75pt]   [align=left] {$v_0'$};
\draw (405,37) node [anchor=north west][inner sep=0.75pt]   [align=left] {$u_1'$};
\draw (455,84) node [anchor=north west][inner sep=0.75pt]   [align=left] {$f'$};
\draw (417,195) node [anchor=north west][inner sep=0.75pt]   [align=left] {$e_1$};
\draw (105,223) node [anchor=north west][inner sep=0.75pt]   [align=left] {$I$};
\draw (98,185) node [anchor=north west][inner sep=0.75pt]   [align=left] {$v_0$};
\draw (215,170) node [anchor=north west][inner sep=0.75pt]   [align=left] {$u_0$};
\draw (215,65) node [anchor=north west][inner sep=0.75pt]   [align=left] {$u_1$};
\draw (208,223) node [anchor=north west][inner sep=0.75pt]   [align=left] {$J$};
\end{tikzpicture}
\caption{Embedding $F$.}
\label{fig:antiRamsey-tree-1}
\end{figure}
    \begin{claim}\label{CLAIM:antiRamsey-missing-edge-D-bar}
        We have $\left|\mathcal{B} \setminus \mathcal{B}[\overline{D}] \right| \le 1$. 
    \end{claim}
    \begin{proof}
        Suppose to the contrary that there exist two edges $e_1, e_2 \in \mathcal{B} \setminus \mathcal{B}[\overline{D}]$. 
        We will consider several cases depending on the position of $f$ in $F$ (see Figure~\ref{fig:tree-plus}). The most technical case would be Case 6. 
        
        \textbf{Case 1:} $f\cap V(T) = \emptyset$. 

            Take any triple $f'\subseteq D \setminus (e_1 \cup e_2)$.  By symmetry, we may assume that $\chi(f') \neq \chi(e_1)$. 
            Fix $u_0' \in e_1 \cap D$ and let $e_1' := e_1 \setminus \{u_0'\}$. 
            Fix $v_0' \in e_1'$. 
            Let $D' := D\setminus (f' \cup e_1')$. 
            Let $\mathcal{H}'$ be the $3$-graph be obtained from $\mathcal{H}$ by removing (at most one) edge with color $\chi(f')$. 
            Observe that the induced bipartite subgraph of $\partial\mathcal{H}'$ on $L\cup D'$ is complete. 
            Combined with $\sigma(T\setminus\{e_{\ast}\}) = |L|$, it is easy to see that there exists an embedding $\psi \colon V(T- v_0) \to L \cup D'$ such that $\psi(T\setminus\{e_{\ast}\}) \subseteq \partial\mathcal{H}'[L, D']$,  $\psi(I\setminus\{v_0\}) = L$, and $\psi(u_0) = u_0'$. 
            By definition, each pair in $\partial\mathcal{H}'[L, D']$ has codegree at least $\tau-1 \ge 3k$. 
            Therefore, by Fact~\ref{FACT:partial-embedding}, the map $\psi$ can be extended greedily to be a (rainbow) embedding $\hat{\psi} \colon F^3 \to \mathcal{H}' \cup \{e_1, f'\}$ with $\hat{\psi}(v_0) = v_0'$, $\hat{\psi}(f^3) = f'$ and $\hat{\psi}(e_{\ast}^3) = e_1$. 
        
            \bigskip

        \textbf{Case 2:} $|f\cap V(T)| = 1$. 

            Suppose that $f\cap e_{\ast} = \emptyset$ and $f\cap I \neq\emptyset$.
            Let us assume that $\{u_1\} = f \cap I$. 
            Then take any triple $f'\subseteq L\cup D \setminus (e_1 \cup e_2)$ such that $|f' \cap L| = 1$.
            Let $\{u_1'\} := f' \cap L$. 
            By symmetry, we may assume that $\chi(f') \neq \chi(e_1)$. 
            Fix $u_0' \in e_1 \cap D$.
            Then similar to the proof in Case~1, it is easy to see that there exists a (rainbow) embedding $\hat{\psi} \colon F^3 \to \mathcal{H} \cup \{e_1, f'\}$ with $\hat{\psi}(u_0) = u_0'$, $\hat{\psi}(u_1) = u_1'$,   $\hat{\psi}(f^3) = f'$, and $\hat{\psi}(e_{\ast}^3) = e_1$. 

            Suppose that $f\cap e_{\ast} = \emptyset$ and $f\cap J \neq\emptyset$. 
            Let us assume that $\{u_1\} = f \cap J$. 
            Then take any triple $f'\subseteq  D \setminus (e_1 \cup e_2)$. 
            By symmetry, we may assume that $\chi(f') \neq \chi(e_1)$. 
            Fix $u_1' \in  f' \cap D$ and $u_0' \in e_1 \cap D$.
            Then similar to the proof in Case~1, it is easy to see that there exists a (rainbow) embedding $\hat{\psi} \colon F^3 \to \mathcal{H} \cup \{e_1, f'\}$ with $\hat{\psi}(u_1) = u_1'$, $\hat{\psi}(u_0) = u_0'$,   $\hat{\psi}(f^3) = f'$, and $\hat{\psi}(e_{\ast}^3) = e_1$. 

            Suppose that $f\cap e_{\ast} \neq\emptyset$ and $f \cap e_{\ast} \in J$.
            Let us assume that $\{u_1\} = \{u_0\} = f \cap e_{\ast}$. 
            Take any triple $f'\subseteq  D \setminus (e_1 \cup e_2)$ such that $|f' \cap e_1| = |f' \cap e_2| = 1$. 
            By symmetry, we may assume that $\chi(f') \neq \chi(e_1)$. 
            Let $\{u_1'\} := f' \cap e_1$.
            Then similar to the proof in Case~1, it is easy to see that there exists a (rainbow) embedding $\hat{\psi} \colon F^3 \to \mathcal{H} \cup \{e_1, f'\}$ with  $\hat{\psi}(u_1) = \hat{\psi}(u_0) = u_1'$,  $\hat{\psi}(f^3) = f'$, and $\hat{\psi}(e_{\ast}^3) = e_1$. 
            
            Suppose that $f\cap e_{\ast} \neq\emptyset$ and $f \cap e_{\ast} \in I$. 
            Let us assume that $\{u_1\} = \{v_0\} = f \cap e_{\ast}$. 
            Fix $u_0' \in e_1 \cap D$ and $u_0'' \in e_2 \cap D$. 
            Let $e_1' := e_1 \setminus \{u_0'\}$ and $e_2' := e_2 \setminus \{u_0''\}$. 
            Take any triple $f'\subseteq  V'$ such that $|f' \cap e_1'| = |f' \cap e_2'| = 1$ and $f' \setminus (e_1' \cup e_2') \subseteq D$. 
            By symmetry, we may assume that $\chi(f') \neq \chi(e_1)$. 
            Let $v_0' \in e_1 \cap f'$ and fix $v_1' \in f' \setminus (e_1' \cup e_2')$. 
            Then similar to the proof in Case~1, it is easy to see that there exists a (rainbow) embedding $\hat{\psi} \colon F^3 \to \mathcal{H} \cup \{e_1, f'\}$ with $\hat{\psi}(u_0) = u_0'$, $\hat{\psi}(u_1) = \hat{\psi}(v_0) = v_0'$, $\hat{\psi}(v_1) = v_1'$,  $\hat{\psi}(f^3) = f'$, and $\hat{\psi}(e_{\ast}^3) = e_1$. 
            
            \bigskip 
            
        \textbf{Case 3:} $f\subseteq J$ and $e_{\ast} \cap f = \emptyset$. 

            Take any triple $f'\subseteq D$ such that $|f' \cap e_1| = |f' \cap e_2| = 0$. 
            By symmetry, we may assume that $\chi(f') \neq \chi(e_1)$. 
            Fix $\{u_1', v_1'\} \subseteq f'$ and $u_0' \subseteq e_1 \setminus \overline{D}$. 
            Then similar to the proof in Case~1,  it is easy to see that there exists a (rainbow) embedding $\hat{\psi} \colon F^3 \to \mathcal{H} \cup \{e_1, f'\}$ with  $\hat{\psi}(u_1) = u_1'$, $\hat{\psi}(v_1) = v_1'$, $\hat{\psi}(u_0) = u_0'$, $\hat{\psi}(f^3) = f'$, and $\hat{\psi}(e_{\ast}^3) = e_1$. 

            \bigskip 
            
        \textbf{Case 4:} $f\subseteq J$ and $e_{\ast} \cap f \neq \emptyset$. 

            Take any triple $f'\subseteq D$ such that $|f' \cap e_1| = |f' \cap e_2| = 1$. 
            By symmetry, we may assume that $\chi(f') \neq \chi(e_1)$.
            Let $u_1' := f' \cap e_1$ and fix $v_1' \in f'\setminus e_1$. 
            Then similar to the proof in Case~1,  it is easy to see that there exists a (rainbow) embedding $\hat{\psi} \colon F^3 \to \mathcal{H} \cup \{e_1, f'\}$ with  $\hat{\psi}(u_1) = \hat{\psi}(u_0) = u_1'$, $\hat{\psi}(v_1) = v_1'$, $\hat{\psi}(f^3) = f'$, and $\hat{\psi}(e_{\ast}^3) = e_1$.

            \bigskip

        \textbf{Case 5:} $f\cap I \neq \emptyset$, $f\cap J\neq \emptyset$, and $e_{\ast} \cap f =\emptyset$. 

            Assume that $\{u_1\} = f\cap I$ and $\{v_1\} = f\cap J$. 
            Take any triple $f'\subseteq L \cup D$ such that $|f' \cap L| = 1$ and $|f' \cap e_1| = |f' \cap e_2| = 0$.  
            By symmetry, we may assume that $\chi(f') \neq \chi(e_1)$.
            Let $u_1' := f' \cap L$. 
            Fix $v_1' \in f'\setminus L$ and $u_0' \in e_1 \setminus \overline{D}$. 
            Then similar to the proof in Case~1,  it is easy to see that there exists a (rainbow) embedding $\hat{\psi} \colon F^3 \to \mathcal{H} \cup \{e_1, f'\}$ with  $\hat{\psi}(u_1) = u_1'$, $\hat{\psi}(v_1) = v_1'$,  $\hat{\psi}(u_0) = u_0'$, $\hat{\psi}(f^3) = f'$, and $\hat{\psi}(e_{\ast}^3) = e_1$. 

            \bigskip 
            
        \textbf{Case 6:} $f\cap I \neq \emptyset$, $f\cap J\neq \emptyset$, and $e_{\ast} \cap f \neq\emptyset$.

            Assume that $\{u_1\} = f\cap I$ and $\{v_1\} = f\cap J$.

            Suppose that $f \cap e_{\ast} \in J$. 
            Then take any triple $f'\subseteq L \cup D$ such that $|f' \cap L| = 1$ and $|f' \cap e_1| = |f' \cap e_2| = 1$.  
            By symmetry, we may assume that $\chi(f') \neq \chi(e_1)$.
            Let $u_1' := f' \cap L$ and $v_1' := f' \cap e_{\ast}$. 
            Then similar to the proof in Case~1,  it is easy to see that there exists a (rainbow) embedding $\hat{\psi} \colon F^3 \to \mathcal{H} \cup \{e_1, f'\}$ with  $\hat{\psi}(u_1) = u_1'$, $\hat{\psi}(v_1) = v_1'$,  $\hat{\psi}(f^3) = f'$, and $\hat{\psi}(e_{\ast}^3) = e_1$. 

            Suppose that $f \cap e_{\ast} \in I$.
            Then fix $u_0' \in e_1 \cap (D)$ and $u_0'' \in e_2 \cap (D)$ (it could be true that $u_0' = u_0''$). 
            Let $e_1' := e_1 \setminus \{u_0'\}$ and $e_2' := e_2 \setminus \{u_0''\}$. 
            Choose a triple $f' \subseteq V'$ such that $|f' \cap e_1'| = |f' \cap e_2'| = 1$ and $f' \setminus (e_1 \cup e_2) \in D$. 
            By symmetry, we may assume that $\chi(f') \neq \chi(e_1)$.
            Let $\{v_0'\} := f' \cap e_1$ and $\{u_1'\} := f' \setminus (e_1 \cup e_2)$. 
            Then similar to the proof in Case~1,  it is easy to see that there exists a (rainbow) embedding $\hat{\psi} \colon F^3 \to \mathcal{H} \cup \{e_1, f'\}$ with  $\hat{\psi}(u_1) = u_1'$, $\hat{\psi}(v_1) = \hat{\psi}(v_0) = v_0'$, $\hat{\psi}(u_0) =  u_0'$,  $\hat{\psi}(f^3) = f'$, and $\hat{\psi}(e_{\ast}^3) = e_1$ (see Figure~\ref{fig:antiRamsey-tree-1}). 

            \bigskip 
            
        \textbf{Case 7:} $f \subseteq I$ and $f\cap e_{\ast} = \emptyset$. 

            The proof is similar to Case 1, and the only difference is that in this case we choose the triple $f'$ to satisfy $|f'\cap L| = 2$ and $|f'\cap e_1| = |f'\cap e_2| = 0$. 

            \bigskip 
            
        \textbf{Case 8:} $f \subseteq I$ and $f\cap e_{\ast} \neq \emptyset$.

            First fix a vertex $u_0' \in e_1 \cap D$ and a vertex $u_0''\in e_2 \cap D$. Let $e_1':= e_1 \setminus\{u_0'\}$ and $e_2':= e_2 \setminus\{u_0''\}$. 
            Take a triple $f'$ satisfying $|f'\cap L| = |f'\cap e_1'| = |f'\cap e_2'| = 1$. 
            The rest part is similar to the proof of Case 6, and in this case, we let $u_0'$ play the role of $u_0$, let the vertex in $f'\cap e_1$ play the role of the vertex $f\cap e_{\ast}$, and let the vertex $f'\cap L$ play the role of $f\setminus e_{\ast}$.            
    \end{proof}
    \begin{claim}\label{CLAIM:Ramsey-bad}
        We have $b \le C |\overline{D}|^2+1$. 
    \end{claim}
    \begin{proof}
        For $i\in \{1,2,3\}$ let  
        \begin{align*}
            \mathcal{B}_i := \left\{e\in \mathcal{B} \colon |e\cap \overline{D}| = i\right\}. 
        \end{align*}
        Since $\mathcal{B}_3$ is $F^3$-free, it follows from the definition of $C$ that $|\mathcal{B}_3|\le C |\overline{D}|^2$. 
        By Claim~\ref{CLAIM:antiRamsey-missing-edge-D-bar}, we have $|\mathcal{B}_1| + |\mathcal{B}_2| \le 1$. 
        Therefore, $b =|\mathcal{B}_1| + |\mathcal{B}_2| + |\mathcal{B}_3| \le C |\overline{D}|^2  + 1$. 
    \end{proof}

    Now it follows from Claims~\ref{CLAIM:Ramsey-missing} and~\ref{CLAIM:Ramsey-bad} that 
    \begin{align*}
        |\mathcal{H}| 
        = |\mathcal{S}| + |\mathcal{B}| - |\mathcal{M}| 
         \le q-2 + C |\overline{D}|^2 + 1 - m 
         \le q-1 - \left(\frac{n}{6}-C\times 6\delta t n\right)|\overline{D}|
        < q
    \end{align*}
    a contradiction. 
\end{proof}
\section{Proofs for generalized Tur\'{a}n results}\label{SEC:Proof-GenTuran}
We prove Theorems~\ref{THM:GenTuran-Cycle-Exact} and~\ref{THM:GenTuran-Cycle-Stability} in this section. 
We will prove Theorem~\ref{THM:GenTuran-Cycle-Stability} first, and then use it to prove Theorem~\ref{THM:GenTuran-Cycle-Exact}. 
%
\subsection{Proof of Theorem~\ref{THM:GenTuran-Cycle-Stability}}\label{SEC:Proof-GenTuran-even-cycle-stability}
Now we are ready to prove Theorem~\ref{THM:GenTuran-Cycle-Stability}. 
The proof is an adaptation of the proof for~{\cite[Theorem~6.2]{KMV15a}}. 

\begin{proof}[Proof of Theorem~\ref{THM:GenTuran-Cycle-Stability}]
        Let $k \ge 4$ be a fixed integer and $t := \floor*{(k-1)/2}$. 
        Fix $\delta >0$. 
        Let $0 < \varepsilon \ll \varepsilon_1 \ll \varepsilon_2 \ll \delta$ be sufficiently small constants and $n$ be a sufficiently large integer. 
        Let $G$ be an $n$-vertex $C_{k}^{\triangle}$-free graph with 
        \begin{align}\label{equ:THM:GenTuran-Cycle-Stability-1}
            N(K_{3}, G) 
            \ge  \left(\frac{t}{4} -\varepsilon\right) n^2.
        \end{align}
        Let $\mathcal{H} := \mathcal{K}_{G}$ and recall that 
        \begin{align*}
            \mathcal{K}_{G} 
            := \left\{e \in \binom{V(G)}{3} \colon G[e] \cong K_3\right\}. 
        \end{align*}
        Let $\mathcal{H}_q$ denote the outputting $3$-graph of the Cleaning Algorithm with input $\left(\mathcal{H}, k, t\right)$ as defined here. 
        Recall from Fact~\ref{FACT:shadow-Turan} that $\mathcal{H}$ is $C_{k}^3$-free. 
        In addition, it follows from Proposition~\ref{PROP:shadow-size} and~\eqref{equ:THM:GenTuran-Cycle-Stability-1} that 
        \begin{align*}
            |\mathcal{H}|
            = N(K_{3}, G)
            \ge t \frac{n^2}{4}  -\varepsilon n^2
            \ge t \left(|G|-\varepsilon n^2\right) -\varepsilon n^2
            = t |\partial\mathcal{H}| - \varepsilon_1 n^2, 
        \end{align*}
        where $\varepsilon_1 := (t+1)\varepsilon$. 
        Therefore, by Lemma~\ref{LEMMA:Cleaning-Algo-cycle}, we have 
        \begin{align}\label{equ:THM:GenTuran-Cycle-Stability-2}
            q 
            \le 12k \varepsilon_1 n^2, \quad 
            |\mathcal{H}_q| 
            \ge |\mathcal{H}| - 48k^2 \varepsilon_1 n^2, \quad\text{and}\quad
            |\partial\mathcal{H}_q| \ge |\partial\mathcal{H}| - 50k^2 \varepsilon_1 n^2.
        \end{align}
        Moreover, the $3$-graph $\mathcal{H}_q$ is $(t,3k)$-superfull. 
        %
        
       Using the third inequality in~\eqref{equ:THM:GenTuran-Cycle-Stability-2} and Proposition~\ref{PROP:shadow-size}, we obtain the following claim. 
       \begin{claim}\label{CLAIM:shadow-Hq-lower-bound}
           We have $|\partial\mathcal{H}_q| \ge \left(1/4-\varepsilon_2\right) n^2$. 
       \end{claim}
       Define a graph 
       \begin{align*}
           G' 
           := \left\{uv \in \partial\mathcal{H}_q \colon d_{\mathcal{H}_q}(uv) = t\right\}.
       \end{align*}
       \begin{claim}\label{CLAIM:size-t-shadow-lower-bound}
           We have $|G'| \ge \left(1/4-2\varepsilon_2\right) n^2$. 
       \end{claim}
       \begin{proof}
           Let $G^{\ast} := \partial\mathcal{H}_q \setminus G'$. 
           Since $\mathcal{H}_q$ is $(t,3k)$-superfull, we have $d_{\mathcal{H}_q}(e) \ge 3k$ for all $e \in G^{\ast}$. 
           Therefore, it follows from Lemma~\ref{LEMMA:KMVa-many-ell+1-shadow-cycle} and the $(t,3k)$-superfullness of $\mathcal{H}_q$ that $|G^{\ast}| \le \varepsilon_2 n^2$.
           Combined with Claim~\ref{CLAIM:shadow-Hq-lower-bound}, we obatin $|G'| \ge |\partial\mathcal{H}_q| - |G^{\ast}| \ge \left(1/4-2\varepsilon_2\right) n^2$. 
       \end{proof}
        Let $V_1 \cup V_2 = V(\mathcal{H})$ be a bipartition such that the number of edges (in $G'$) crossing $V_1$ and $V_2$ is maximized. 
        It follows from Claim~\ref{CLAIM:size-t-shadow-lower-bound} and Proposition~\ref{APPENDIX:PROP:Turan-cycle-shadow} that $|G'[V_1, V_2]| \ge |G'| - \varepsilon_2 n^2 \ge \left(1/4-3\varepsilon_2\right) n^2$. 
        Simple calculations show that 
        \begin{align}\label{equ:THM:GenTuran-Cycle-Stability-3}
            \frac{n}{2} - \sqrt{3\varepsilon_2}n 
            \le |V_i|
            \le \frac{n}{2} + \sqrt{3\varepsilon_2}n 
            \quad\text{for}\quad i\in \{1,2\}. 
        \end{align}
        Define
           \begin{align*}
               \mathcal{W}
               := \left\{(S, T) \in \binom{V_1}{3k}\times \binom{V_2}{3k} \colon G'[S,T] \cong K_{3k, 3k}\right\}. 
           \end{align*}
        Since each nonedge in $G'[V_1, V_2]$ contributes at most $\binom{|V_1|-1}{3k-1}\binom{|V_2|-1}{3k-1}$ elements in $\binom{V_1}{3k}\times \binom{V_2}{3k} - \mathcal{W}$, it follows from~\eqref{equ:THM:GenTuran-Cycle-Stability-3} that  
        \begin{align*}
            |\mathcal{W}|
            & \ge \binom{|V_1|}{3k}\binom{|V_2|}{3k}
                - 3\varepsilon_2 n^2 \binom{|V_1|-1}{3k-1}\binom{|V_2|-1}{3k-1} \\
            & = \binom{|V_1|}{3k}\binom{|V_2|}{3k} - 3\varepsilon_2 n^2 \frac{3k}{|V_1|}\frac{3k}{|V_2|} \binom{|V_1|}{3k}\binom{|V_2|}{3k} 
             \ge (1-144\varepsilon_2 k^2)\binom{|V_1|}{3k}\binom{|V_2|}{3k},  
        \end{align*}
        Here we used the identity $\binom{|V_i|-1}{3k-1} = \frac{3k}{|V_i|}\binom{|V_i|}{3k}$.  
        By averaging, there exists an edge $e\in G'[V_1, V_2]$ that is contained in at least 
        \begin{align*}
            \frac{(3k)^2|\mathcal{W}|}{|G'[V_1, V_2]|}
            \ge \frac{(3k)^2|\mathcal{W}|}{|V_1||V_2|}
            \ge (1-144\varepsilon_2 k^2)\binom{|V_1|-1}{3k-1}\binom{|V_2|-1}{3k-1}
        \end{align*}
        members in $\mathcal{W}$. 
        For every edge $f\in G'[V_1, V_2]$ that is disjoint from $e$, there are at most $\binom{|V_1|-2}{3k-2}\binom{|V_2|-2}{3k-2}$ copies of $K_{3k,3k}$ in $G'[V_1, V_2]$ containing both $e$ and $f$. 
        Hence, the set 
        \begin{align*}
            G'':= \left\{f \in G'[V_1, V_2] \colon \mathrm{\ there\ exists\ a \ } W\in \mathcal{W},\ \{e,f\} \subset W\right\}
        \end{align*}
        has size at least 
        \begin{align*}
           (3k-1)^2 \frac{(1-144\varepsilon_2 k^2)\binom{|V_1|-1}{3k-1}\binom{|V_2|-1}{3k-1}}{\binom{|V_1|-2}{3k-2}\binom{|V_2|-2}{3k-2}}
           & = (1-144\varepsilon_2 k^2)\left(|V_1|-1\right)\left(|V_2|-1\right) \\
           & \ge \left(\frac{1}{4}-\frac{\delta}{2t}\right) n^2, 
        \end{align*}
       where the last inequality follows from~\eqref{equ:THM:GenTuran-Cycle-Stability-3} and the assumption that $\varepsilon_2 \ll \delta$. 
       Let $L:= N_{\mathcal{H}_{q}}(e)$. 
       Then $|L| = t$ and 
       it follows from Lemma~\ref{LEMMA:KMVa-superfull-complete-bipartite-cycle} and the definitions of $\mathcal{W}, G''$ that 
       \begin{align*}
           N_{\mathcal{H}_{q}}(f) = N_{\mathcal{H}_{q}}(e) = L
           \quad\text{for all}\quad f\in G''. 
       \end{align*}
       This implies that $|G-L| \ge |G''| \ge \left(\frac{1}{4}-\frac{\delta}{2t}\right) n^2$ and $d_{G}(v) \ge |G''| \ge \left(\frac{1}{4}-\frac{\delta}{2t}\right) n^2$ for every $v\in L$. 
       The latter conclusion implies that the number of triangles in $G$ containing at least one vertex from $L$ is at least $t \left(\frac{1}{4}-\frac{\delta}{2t}\right) n^2$. 
       Combining this with Propositions~\ref{PROP:shadow-size} and~\ref{APPENDIX:PROP:Turan-cycle-shadow}, we obtain 
       \begin{align*}
           N(K_3, G-L)
           \le |\mathcal{H}| - t \left(\frac{1}{4}-\delta\right) n^2
           & \le t \cdot |G| + \varepsilon n^2 - t \left(\frac{1}{4}-\frac{\delta}{2t}\right) n^2 \\
           & \le t \cdot \left(\frac{1}{4}+o(1)\right) n^2 + \varepsilon n^2 - t \left(\frac{1}{4}-\frac{\delta}{2t}\right) n^2
           \le \delta n^2. 
       \end{align*}
       Finally, it follows from Proposition~\ref{PROP:shadow-size} that $G-L$ can be made bipartite by removing at most $\delta n^2$ edges. 
    This completes the proof of Theorem~\ref{THM:GenTuran-Cycle-Stability}.  
\end{proof}
%
\subsection{Proof of Theorem~\ref{THM:GenTuran-Cycle-Exact}}\label{SEC:Proof-GenTuran-even-cycle-exact}
In this subsection, we prove Theorem~\ref{THM:GenTuran-Cycle-Exact} by using Theorem~\ref{THM:GenTuran-Cycle-Stability}. 
There are slight distinctions between the proofs for the even and odd cases, despite their shared framework. 
Since the proof for the odd case is less technical, we omit it and refer the interested reader to the Appendix for more detail. 
%
\begin{proof}[Proof of Theorem~\ref{THM:GenTuran-Cycle-Exact} for even $k$]
    Fix $k = 2t+2 \ge 6$. 
    Let $C>0$ be a constant such that $\mathrm{ex}(N, C_{k}^3) \le C N^2$ holds for all integers $N \ge 0$. 
    The existence of such a constant $C$ is guaranteed by the theorem of Kostochka--Mubayi--Verstra\"{e}te~\cite{KMV15a}. 
    Let $0 < \delta \ll C^{-1}$ be sufficiently small and $n \gg C$ be sufficiently large. 
    Let $G$ be an $n$-vertex $C_{k}^{\triangle}$-free graph with 
    \begin{align*}
        N(K_3, G) = |\mathcal{S}^{+}_{\mathrm{bi}}(n,t)|. 
    \end{align*}
    We may assume that every edge in $G$ is contained in some triangle of $G$, since otherwise we can delete it from $G$ and this does not change the value of $N(K_3, G)$. 
    Our aim is to prove that $G \cong S^{+}_{\mathrm{bi}}(n,t)$. 

    Since $N(K_3, G) = t n^2/4 - o(n^2)$ and $n$ is large, it follows from Theorem~\ref{THM:GenTuran-Cycle-Stability} that there exists a $t$-set $L:= \{x_1, \ldots, x_{t}\} \subseteq V(G)$ such that 
    \begin{enumerate}[label=(\roman*)]
        \item $|G-L| \ge n^2/4 - \delta n^2$, 
        \item $N(K_3, G-L) \le \delta n^2$, 
        \item $G-L$ can be made bipartite by removing at most $\delta n^2$ edges, and
        \item $d_{G}(v) \ge (1-\delta)n$ for all $v\in L$. 
    \end{enumerate}
    Let $V := V(G)$, $V':= V-L$, and $V_1 \cup V_2 = V'$ be a bipartition such that the number of edges (in $G$) crossing $V_1$ and $V_2$ is maximized. 
    Define 
    \begin{align*}
        \mathcal{S} &:= \left\{e \in \binom{V}{3} \colon |e \cap L| \ge 1,\ |e\cap V_1| \le 1,\ |e\cap V_2| \le 1\right\}, \quad\text{and} \\
         \mathcal{S}' &:= \left\{e \in \binom{V}{3} \colon |e \cap L| = |e\cap V_1| = |e\cap V_2| = 1\right\} \subseteq \mathcal{S}.
    \end{align*}
    and let $\mathcal{H} := \mathcal{K}_{G}$, $\mathcal{B} := \mathcal{H}\setminus \mathcal{S}$, $\mathcal{M} := \mathcal{S}'\setminus  \mathcal{H} \subseteq \mathcal{S}\setminus  \mathcal{H}$. 
    Let $m:= |\mathcal{M}|$ and $b:= |\mathcal{B}|$. 
    It follows from Statements~(i) and~(iii) above that 
    \begin{align}\label{equ:GenTuran-cycle-G[V1,V2]}
        |G[V_1, V_2]| \ge \frac{n^2}{4} - 2\delta n^2. 
    \end{align}
    Combined with Statement~(iv), for every $x_i \in L$ the intersection of links $L_{\mathcal{H}}(x_i)$ and $L_{\mathcal{S}'}(x_i)$  satisfies 
    \begin{align*}
        |L_{\mathcal{H}}(x_i) \cap L_{\mathcal{S}'}(x_i)|
        = |L_{\mathcal{H}}(x_i) \cap G[V_1, V_2]|
        \ge |G[V_1, V_2]| - \delta n \times n
        \ge \frac{n^2}{4} - 3\delta n^2. 
    \end{align*}
    Therefore, 
    \begin{align}\label{equ:GenTuran-cycle-m-upper}
        |\mathcal{M}|
       &  = \sum_{x_i \in L}\left(|L_{\mathcal{S}'}(x_i)| - |L_{\mathcal{H}}(x_i) \cap L_{\mathcal{S}'}(x_i)|\right)  \notag \\
        & \le t \left(|V_1||V_2| - \left(\frac{n^2}{4} - 3\delta n^2\right)\right)
        \le 3\delta t n^2. 
    \end{align}
    Inequality~\eqref{equ:GenTuran-cycle-G[V1,V2]} with some simple calculations also imply that 
    \begin{align}\label{equ:GenTuran-cycle-Vi-size}
        \left(\frac{1}{2} - \sqrt{2\delta}\right)n 
            \le |V_i| 
            \le \left(\frac{1}{2} + \sqrt{2\delta}\right)n 
            \quad\text{for}\quad i\in \{1,2\}. 
        \end{align}
    Let 
    \begin{gather*}
        \tau := \frac{n}{200}, \quad
        D  := \left\{y\in V' \colon d_{\mathcal{H}}(y x_i) \ge \tau \text{ for all } x_i \in L \right\}, \quad \overline{D} := V'\setminus D.
    \end{gather*}
    Let $D_i  := D \cap V_i$, $\overline{D}_i := V_i \setminus D_i$ for $i\in \{1,2\}$. 
    We also divide $\mathcal{M}$ further by letting 
    \begin{align*}
        \mathcal{M}_1 := \left\{e\in \mathcal{M} \colon e\cap \overline{D} \neq\emptyset\right\}, \quad\text{and}\quad
        \mathcal{M}_2 := \mathcal{M}\setminus \mathcal{M}_1. 
    \end{align*}
    Let $m_1 := |\mathcal{M}_1|$ and $m_2:= |\mathcal{M}_2|$. 
    \begin{claim}\label{CLAIM:GenTuran-cycle-D-bar}
        We have $m_1 \ge 49 n|\overline{D}|/100$ and  $|\overline{D}| \le 8 \delta t n$. 
    \end{claim}
    \begin{proof}
        By the definition of $D$, 
        for every $i\in \{1,2\}$ and for every vertex $v\in \overline{D}_i$ there exists a vertex $x\in L$ (depending on $v$) such that 
        \begin{align*}
            |N_{\mathcal{H}}(vx) \cap V_{3-i}| \le \tau. 
        \end{align*}
        Hence, the pair $\{v,x\}$ contributes at least $|V_{3-i}| - \tau$ elements to $\mathcal{M}$. 
        Therefore, it follows from~\eqref{equ:GenTuran-cycle-Vi-size} that 
        \begin{align*}
            m_1 
             \ge \left(\min\{|V_1|, |V_2|\} - \tau \right) \left(|\overline{D}_1| + |\overline{D}_2|\right) 
             \ge \left(\frac{n}{2}-\sqrt{2\delta}n - \tau\right)|\overline{D}|
             \ge \frac{49}{100} n |\overline{D}|. 
        \end{align*}
        Combined with~\eqref{equ:GenTuran-cycle-m-upper}, we obtain $|\overline{D}|\le \frac{3\delta t n^2}{49 n/100}\le 8 \delta t n$.
    \end{proof}
    %
    
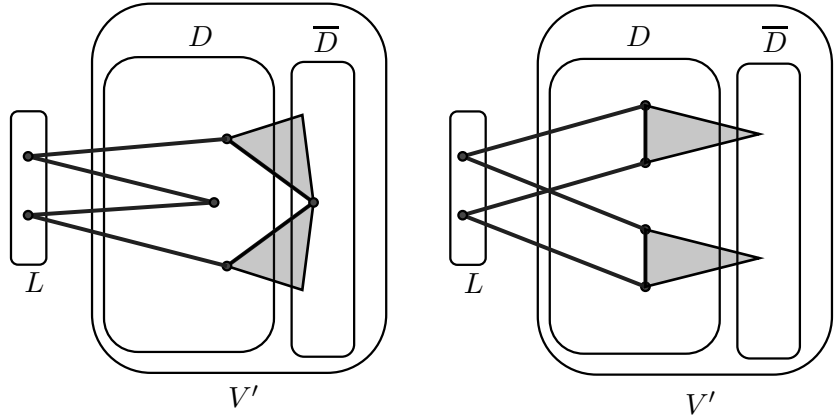
\begin{figure}[htbp]
\centering
\tikzset{every picture/.style={line width=0.85pt}} 

\begin{tikzpicture}[x=0.75pt,y=0.75pt,yscale=-1,xscale=1, scale=0.8]

\draw   (85.96,87.66) .. controls (88.4,87.66) and (90.37,89.63) .. (90.37,92.07) -- (90.37,179.71) .. controls (90.37,182.14) and (88.4,184.12) .. (85.96,184.12) -- (72.73,184.12) .. controls (70.3,184.12) and (68.32,182.14) .. (68.32,179.71) -- (68.32,92.07) .. controls (68.32,89.63) and (70.3,87.66) .. (72.73,87.66) -- cycle ;
\draw   (211.13,53.62) .. controls (222.84,53.62) and (232.33,63.11) .. (232.33,74.82) -- (232.33,217.8) .. controls (232.33,229.51) and (222.84,239) .. (211.13,239) -- (147.53,239) .. controls (135.82,239) and (126.33,229.51) .. (126.33,217.8) -- (126.33,74.82) .. controls (126.33,63.11) and (135.82,53.62) .. (147.53,53.62) -- cycle ;
\draw   (265.64,20) .. controls (285.9,20) and (302.33,36.43) .. (302.33,56.69) -- (302.33,215.46) .. controls (302.33,235.72) and (285.9,252.15) .. (265.64,252.15) -- (155.55,252.15) .. controls (135.29,252.15) and (118.86,235.72) .. (118.86,215.46) -- (118.86,56.69) .. controls (118.86,36.43) and (135.29,20) .. (155.55,20) -- cycle ;
\draw   (274.53,56.62) .. controls (278.84,56.62) and (282.33,60.11) .. (282.33,64.42) -- (282.33,234.2) .. controls (282.33,238.51) and (278.84,242) .. (274.53,242) -- (251.13,242) .. controls (246.83,242) and (243.33,238.51) .. (243.33,234.2) -- (243.33,64.42) .. controls (243.33,60.11) and (246.83,56.62) .. (251.13,56.62) -- cycle ;
\draw   (359.96,87.76) .. controls (362.4,87.76) and (364.37,89.73) .. (364.37,92.17) -- (364.37,179.81) .. controls (364.37,182.25) and (362.4,184.22) .. (359.96,184.22) -- (346.73,184.22) .. controls (344.3,184.22) and (342.32,182.25) .. (342.32,179.81) -- (342.32,92.17) .. controls (342.32,89.73) and (344.3,87.76) .. (346.73,87.76) -- cycle ;
\draw   (489.13,54.72) .. controls (500.84,54.72) and (510.33,64.21) .. (510.33,75.92) -- (510.33,218.9) .. controls (510.33,230.61) and (500.84,240.1) .. (489.13,240.1) -- (425.53,240.1) .. controls (413.82,240.1) and (404.33,230.61) .. (404.33,218.9) -- (404.33,75.92) .. controls (404.33,64.21) and (413.82,54.72) .. (425.53,54.72) -- cycle ;
\draw   (543.64,21.1) .. controls (563.9,21.1) and (580.33,37.53) .. (580.33,57.8) -- (580.33,216.56) .. controls (580.33,236.83) and (563.9,253.26) .. (543.64,253.26) -- (433.55,253.26) .. controls (413.29,253.26) and (396.86,236.83) .. (396.86,216.56) -- (396.86,57.8) .. controls (396.86,37.53) and (413.29,21.1) .. (433.55,21.1) -- cycle ;
\draw   (552.53,57.72) .. controls (556.84,57.72) and (560.33,61.21) .. (560.33,65.52) -- (560.33,235.3) .. controls (560.33,239.61) and (556.84,243.1) .. (552.53,243.1) -- (529.13,243.1) .. controls (524.83,243.1) and (521.33,239.61) .. (521.33,235.3) -- (521.33,65.52) .. controls (521.33,61.21) and (524.83,57.72) .. (529.13,57.72) -- cycle ;
\draw [line width=1.5pt,color=sqsqsq]   (203,105) -- (257,145) ;
\draw [line width=1.5pt,color=sqsqsq]   (203,185) -- (257,145) ;
\draw[line width=1pt, fill=sqsqsq,fill opacity=0.25]  (203,105) -- (257,145) -- (250,90) --cycle;
\draw[line width=1pt, fill=sqsqsq,fill opacity=0.25]  (203,185) -- (257,145) -- (250,200) --cycle;
\draw [fill=uuuuuu] (203,105) circle (2pt);
\draw [fill=uuuuuu] (257,145) circle (2pt);
\draw [fill=uuuuuu] (203,185) circle (2pt);
\draw [fill=uuuuuu] (464,84) circle (2pt);
\draw [fill=uuuuuu] (464,198) circle (2pt);
\draw [fill=uuuuuu] (464,120) circle (2pt);
\draw [fill=uuuuuu] (464,162) circle (2pt);
\draw  [line width=1.5pt,color=sqsqsq]  (79,116) -- (203,105) ;
\draw  [line width=1.5pt,color=sqsqsq]  (79,116) -- (195,145) ;
\draw  [line width=1.5pt,color=sqsqsq]  (79,153) -- (195,145) ;
%
\draw  [line width=1.5pt,color=sqsqsq]  (79,153) -- (203,185) ;
\draw [fill=uuuuuu] (79,116) circle (2pt);
\draw [fill=uuuuuu] (195,145) circle (2pt);
\draw [fill=uuuuuu] (79,153) circle (2pt);
%
\draw  [line width=1.5pt,color=sqsqsq]  (464,84) -- (464,120) ;
\draw  [line width=1.5pt,color=sqsqsq]  (464,198) -- (464,162) ;
%
\draw [line width=1.5pt,color=sqsqsq]   (350,116) -- (464,84) ;
\draw [line width=1.5pt,color=sqsqsq]   (350,116) -- (464,162) ;
\draw [line width=1.5pt,color=sqsqsq]   (350,153) --  (464,120) ;
\draw [line width=1.5pt,color=sqsqsq]   (350,153) -- (464,198) ;
\draw [fill=uuuuuu] (350,116) circle (2pt);
\draw [fill=uuuuuu] (350,153) circle (2pt);
\draw[line width=1pt, fill=sqsqsq,fill opacity=0.25]  (464,162) -- (535,180) -- (464,198) --cycle;
\draw[line width=1pt, fill=sqsqsq,fill opacity=0.25]  (464,84) -- (535,102) -- (464,120) --cycle;
\draw (74.73,187.12) node [anchor=north west][inner sep=0.75pt]   [align=left] {$L$};
\draw (255,32) node [anchor=north west][inner sep=0.75pt]   [align=left] {$\overline{D}$};
\draw (201.77,257.9) node [anchor=north west][inner sep=0.75pt]   [align=left] {$V'$};
\draw (176.95,32) node [anchor=north west][inner sep=0.75pt]   [align=left] {$D$};
\draw (348.73,187.22) node [anchor=north west][inner sep=0.75pt]   [align=left] {$L$};
\draw (535,32) node [anchor=north west][inner sep=0.75pt]   [align=left] {$\overline{D}$};
\draw (486.77,262) node [anchor=north west][inner sep=0.75pt]   [align=left] {$V'$};
\draw (450,32) node [anchor=north west][inner sep=0.75pt]   [align=left] {$D$};
\end{tikzpicture}
\caption{Embedding $C_6$ in two different cases.}
\label{fig:genturan-cycle-1}
\end{figure}
    
    \begin{claim}\label{CLAIM:GenTuran-cycle-P_2}
        The $3$-graph $\mathcal{B}[V']$ does not contain two edges $e_1$ and $e_2$ such that 
        \begin{enumerate}[label=(\roman*)]
            \item\label{equ:CLAIM:GenTuran-cycle-P_2} 
            $|e_1 \cap e_2| = 1$, $(e_1 \setminus e_2) \cap D \neq \emptyset$, and $(e_2 \setminus e_1) \cap D \neq \emptyset$, or 
            \item\label{equ:CLAIM:GenTuran-cycle-B1-two-edges}
                $\min\left\{|e_1 \cap D|,\ |e_2 \cap D|\right\} \ge 2$ and $e_1 \cap e_2 = \emptyset$.
        \end{enumerate}
    \end{claim}
    \begin{proof}
        Suppose to the contrary that there exist two edges $e_1, e_2 \in \mathcal{B}[V']$ such that~\ref{equ:CLAIM:GenTuran-cycle-P_2} holds. 
        Let $\{v_0\} := e_1 \cap e_2$ and 
        fix $v_i \in \left(e_i \cap D\right) \setminus \{v_0\}$ for $i\in \{1,2\}$. 
        Let $D' := D\setminus\left(e_1 \cup e_2\right)$. 
        Choose any $(t-1)$-set $\{u_1, \ldots, u_{t-1}\} \subseteq D'$. 
        It follows from the definition of $D$ that the induced bipartite graph of $\partial\mathcal{H}$ on $L \cup D$ is complete. 
        Therefore, $F := v_1 x_1 u_1 x_2 u_2 \cdots x_{t-1} u_{t-1} x_t v_2$ is copy of $P_{k-2}$ in the bipartite graph $\partial\mathcal{H}[L, D']$. 
        This $P_{k-2}$ together with $v_1 v_0 v_2$ (a copy of $P_2$) form a copy of $C_k$ in $\partial\mathcal{H}$. 
        Since all edges in $F$ have codegree at least $\tau \ge 3k$ in $\mathcal{H}$, it follows from Fact~\ref{FACT:partial-embedding} that $C_{k}^3 \subseteq \mathcal{H}$ (see Figure~\ref{fig:genturan-cycle-1}), a contradiction. 

        Suppose to the contrary that there exist two disjoint edges $e_1, e_2 \in \mathcal{B}[V']$ such that~\ref{equ:CLAIM:GenTuran-cycle-B1-two-edges} holds. 
        Fix a $2$-set $\{v_i, v_i'\} \subseteq e_i \cap D$ for $i\in \{1,2\}$. 
        Choose a $(t-2)$-set $\{u_1, \ldots, u_{t-2}\} \subseteq D\setminus (e_1 \cup e_2)$. 
        Similar to the proof above, the graph $F := v_1 x_1 u_1 x_2  \cdots u_{t-2} x_{t-1} v_2 v_2' x_t v_1' v_1$ is a copy of $C_{k}$ in $\partial\mathcal{H}$. 
        Note that all edges but $v_1 v_1', v_2 v_2'$ in $F$ have codegree at least $\tau \ge 3k$ in $\mathcal{H}$. 
        So it follows from Fact~\ref{FACT:partial-embedding} that $C_{k}^{3} \subseteq \mathcal{H}$ (see Figure~\ref{fig:genturan-cycle-1}),  a contradiction. 
    \end{proof}
    \begin{claim}\label{CLAIM:GenTuran-cycle-Maxdeg}
        For every $v\in V' = D \cup \overline{D}$ we have 
        \begin{align*}
            \min\left\{|N_{G}(v) \cap D_1|, |N_{G}(v) \cap D_2|\right\} 
            \le \sqrt{3\delta}n. 
    \end{align*}
    \end{claim}
    \begin{proof}
        Suppose to the contrary that there exists a vertex $v\in V'$ such that the set $N_i := N_{G}(v) \cap D_i$ has size at least $\sqrt{3\delta}n$ for both $i \in \{1,2\}$. 
        Observe from Claim~\ref{CLAIM:GenTuran-cycle-P_2}~\ref{equ:CLAIM:GenTuran-cycle-P_2} that the induced bipartite graph $G[N_1, N_2]$ does not contain two disjoint edges. 
        Hence, $|G[N_1, N_2]| \le n$ and consequently, 
        \begin{align*}
            |G[V_1, V_2]| 
            \le |V_1||V_2| - \left(|N_1||N_2| - |G[N_1, N_2]|\right)
            \le \frac{n^2}{4} - (3\delta n^2 - n)
            < \frac{n^2}{4} - 2\delta n^2, 
        \end{align*}
        contradicting~\eqref{equ:GenTuran-cycle-G[V1,V2]}. 
    \end{proof}
    \begin{claim}\label{CLAIM:GenTuran-bad-deg-max}
        For every $i\in \{1,2\}$ and for every $v\in D_i \cup \overline{D}_i$ we have $|N_{G}(v) \cap D_i| \le 9\sqrt{\delta}tn$. 
    \end{claim}
    \begin{proof}
        Suppose to the contrary that there exist $i\in \{1,2\}$ and $v\in D_i \cup \overline{D}_i$ such that $|N_{G}(v) \cap D_i| > 9\sqrt{\delta}tn$. 
        Then by the maximality of $G[V_1, V_2]$, we have $|N_{G}(v) \cap V_{3-i}| \ge |N_{G}(v) \cap V_i| \ge |N_{G}(v) \cap D_i| > 9\sqrt{\delta}tn$, since otherwise we can move $v$ from $V_i$ to $V_{3-i}$ and this will result in a large bipartite subgraph of $G$. 
        By Claim~\ref{CLAIM:GenTuran-cycle-D-bar}, this implies that $|N_{G}(v) \cap D_{3-i}| \ge |N_{G}(v) \cap V_{3-i}|- |\overline{D}| > 3\sqrt{\delta} n$, 
         contradicting Claim~\ref{CLAIM:GenTuran-cycle-Maxdeg}. 
    \end{proof}
    \begin{claim}\label{CLAIM:GenTuran-cycle-2-intersecting}
        We have $|\mathcal{B}[D]| \le n/2 + 11\sqrt{\delta}tn$. 
    \end{claim}
    \begin{proof}
        First, observe from Claim~\ref{CLAIM:GenTuran-cycle-P_2} that the $3$-graph $\mathcal{B}[D]$ is $2$-intersecting.
        Suppose to the contrary that $|\mathcal{B}[D]| > n/2 + 11\sqrt{\delta}tn$. 
        Then, by Fact~\ref{FACT:2-intersecting}, there exists a pair $\{v_1, v_2\} \subseteq D$ such that all edges in $\mathcal{B}[D]$ containing $\{v_1, v_2\}$. 
        Then it follows from Claims~\ref{CLAIM:GenTuran-cycle-Maxdeg},~\ref{CLAIM:GenTuran-bad-deg-max}, and~\eqref{equ:GenTuran-cycle-Vi-size} that 
        \begin{align*}
            |\mathcal{B}[D]|
            \le \max\{d_{G[D]}(v_1),\ d_{G[D]}(v_2)\}
            \le \frac{n}{2}+\sqrt{2\delta}n + 9\sqrt{\delta}tn
            \le \frac{n}{2}+11\sqrt{\delta}tn. 
        \end{align*}
        This proves Claim~\ref{CLAIM:GenTuran-cycle-2-intersecting}. 
    \end{proof}
    For $i\in \{0, 1,2,3\}$ let 
    \begin{align*}
        \mathcal{B}_i 
        := \left\{e\in \mathcal{B} \colon |e \cap \overline{D}| = i\right\}. 
    \end{align*}
    Since $\mathcal{B}_3$ is $C_{k}^3$-free, it follows from the definition of $C$ that 
    \begin{align}\label{equ:GenTuran-cycle-B3}
        |\mathcal{B}_3| \le C |\overline{D}|^2. 
    \end{align}
    \begin{claim}\label{CLAIM:GenTuran-cycle-B_2'}
        We have $ |\mathcal{B}_2| \le {n |\overline{D}|}/{3}$. 
    \end{claim}
    \begin{proof}
        Let us partition $\mathcal{B}_2$ into two sets $\mathcal{B}_2'$ and $\mathcal{B}_2''$, where
        \begin{align*}
            \mathcal{B}_2' 
            := \left\{e\in \mathcal{B}_2 \colon e \cap D \neq \emptyset\right\}
            \quad\text{and}\quad 
            \mathcal{B}_2'' := \mathcal{B}_2\setminus \mathcal{B}_2'. 
        \end{align*}
        Since every edge in $\mathcal{B}_2''$ contains exactly one vertex in $L$, we have 
        \begin{align}\label{equ:GenTuran-cycle-B2''}
            |\mathcal{B}_2''|
            \le |L| \binom{|\overline{D}|}{2}
            \le t|\overline{D}|^2. 
        \end{align}
        Let 
        \begin{align*}
            F := \binom{\overline{D}}{2} \cap \partial\mathcal{B}_2', \quad
            F_1 
            := \left\{e\in F \colon d_{\mathcal{B}_2'}(e) = 1\right\}, 
            \quad\text{and}\quad
            F_2 := F\setminus F_1. 
        \end{align*}
        It is clear that $F_1$ contributes at most $\binom{|\overline{D}|}{2}$ edges to $\mathcal{B}_2'$. 
        On the other hand, 
        by Claim~\ref{CLAIM:GenTuran-cycle-P_2}~\ref{equ:CLAIM:GenTuran-cycle-P_2}, the graph $F_2$ is a matching. 
        Combined with Claim~\ref{CLAIM:GenTuran-cycle-Maxdeg}, we obtain 
        \begin{align*}
            |\mathcal{B}_2'|
            \le \binom{|\overline{D}|}{2} + \left(\max\{|D_1|,\ |D_2|\}+\sqrt{3\delta}n\right)|F_2|
            & \le |\overline{D}|^2+ \left(\frac{n}{2}+5\sqrt{\delta}n\right)\frac{|\overline{D}|}{2}. 
        \end{align*}
        Therefore, by Claim~\ref{CLAIM:GenTuran-cycle-D-bar}, 
        \begin{align*}
            |\mathcal{B}_2|
            = |\mathcal{B}_2''|+|\mathcal{B}_2'|
            & \le t|\overline{D}|^2+ |\overline{D}|^2+ \left(\frac{n}{2}+5\sqrt{\delta}n\right)\frac{|\overline{D}|}{2} \\
            & \le \left(t\times 8\delta t n + 8\delta t n + \frac{n}{4}+\frac{5}{2}\sqrt{\delta}n\right) |\overline{D}| 
             \le \frac{n |\overline{D}|}{3}.
        \end{align*}
        This proves Claim~\ref{CLAIM:GenTuran-cycle-B_2'}. 
     \end{proof}
    \begin{claim}\label{CLAIM:GenTuran-cycle-B1}
        We have 
        \begin{align*}
            |\mathcal{B}_1| 
            \le 10\sqrt{\delta} t^2 n  |\overline{D}|+
                \begin{cases}
                    0, & \quad\text{if}\quad \mathcal{B}[D] \neq \emptyset, \\
                 \frac{n}{2}+4\sqrt{\delta}n, & \quad\text{if}\quad \mathcal{B}[D] = \emptyset. 
                \end{cases}
        \end{align*}
        In particular, (by Claim~\ref{CLAIM:GenTuran-cycle-2-intersecting})
        \begin{align*}
            |\mathcal{B}_1| + |\mathcal{B}[D]|
            \le 10\sqrt{\delta} t^2 n  |\overline{D}| + \frac{n}{2}+11\sqrt{\delta}tn. 
        \end{align*}
    \end{claim}
    \begin{proof}
        Let us partition $\mathcal{B}_1$ into two subsets $\mathcal{B}_1'$ and $\mathcal{B}_1''$, where
        \begin{align*}
            \mathcal{B}_1'
            := \left\{e\in \mathcal{B}_1 \colon e\cap L\neq \emptyset\right\}
            \quad\text{and}\quad
            \mathcal{B}_1'':= \mathcal{B}_1\setminus \mathcal{B}_1'. 
        \end{align*}
        It follows from Claim~\ref{CLAIM:GenTuran-bad-deg-max} that 
        \begin{align*}
            |\mathcal{B}_1'| 
            \le |\overline{D}| \times 9\sqrt{\delta}tn \times t 
            = 9\sqrt{\delta} t^2 n |\overline{D}|. 
        \end{align*}
        Now we consider $\mathcal{B}_1''$. 
        First, suppose that there exists an edge $e_1 \in \mathcal{B}[D]$. 
        If, in addition, there exists a vertex $v\in \overline{D}$ having degree at least four in $\mathcal{B}_1''$, then we can choose a pair $e_2' \in L_{\mathcal{H}}(v) \cap \binom{D}{2}$ such that $|e_2' \cap e_1| \le 1$. 
        However, edges $e_1$ and $e_2 := e_2' \cup \{v\} \in \mathcal{B}_1''$ contradict Claim~\ref{CLAIM:GenTuran-cycle-P_2}~\ref{equ:CLAIM:GenTuran-cycle-P_2} (if $|e_2' \cap e_1| = 1$) or~\ref{equ:CLAIM:GenTuran-cycle-B1-two-edges} (if $|e_2' \cap e_1| = 0$). 
        Therefore, every vertex in $\overline{D}$ has degree at most three in $\mathcal{B}_1''$. 
        Hence, $|\mathcal{B}_1''| \le 3 |\overline{D}|$, and 
        \begin{align*}
            |\mathcal{B}|
            = |\mathcal{B}_1'| + |\mathcal{B}_1''|
            \le 9\sqrt{\delta} t^2 n |\overline{D}| + 3|\overline{D}|
            \le 10\sqrt{\delta} t^2 n |\overline{D}|. 
        \end{align*}
        Here we used the fact that $\sqrt{\delta} t^2 n \gg 1$. 

        Next, suppose that $\mathcal{B}[D] = \emptyset$.
        Let $v\in \overline{D}$. 
        Observe that the link graph $L_{\mathcal{B}_1''}(v)$ is a subgraph of $G[D] \subseteq G[V']$. 
        Claim~\ref{CLAIM:GenTuran-cycle-P_2}~\ref{equ:CLAIM:GenTuran-cycle-P_2}  implies that the matching number of $L_{\mathcal{B}_1''}(v)$ is at most one. 
        By Fact~\ref{FACT:matching-1}, either $|L_{\mathcal{B}_1''}(v)|\le 3$ or there exists a vertex $v_1 \in D$ such that all edges in $L_{\mathcal{B}_1''}(v)$ containing $v_1$. 
        In the latter case, it follows from Claim~\ref{CLAIM:GenTuran-cycle-Maxdeg}
        and~\eqref{equ:GenTuran-cycle-Vi-size}
        that 
        \begin{align*}
            |L_{\mathcal{B}_1''}(v)|
            \le d_{G[D]}(v_1)
            \le \frac{n}{2}+\sqrt{2\delta}n + \sqrt{3\delta}n
            \le \frac{n}{2}+4\sqrt{\delta}n. 
        \end{align*}
        
        Similar to the proof above, 
        it follows from Claims~\ref{CLAIM:GenTuran-cycle-P_2}~\ref{equ:CLAIM:GenTuran-cycle-B1-two-edges} that the number of vertices $v\in \overline{D}$ with $|L_{\mathcal{B}_1''}(v)| \ge 4$ is at most one. 
        Therefore, 
        \begin{align*}
            |\mathcal{B}_1''|
            \le 3|\overline{D}| + \frac{n}{2}+4\sqrt{\delta}n. 
        \end{align*}
        Therefore, 
        \begin{align*}
            |\mathcal{B}_1|
            = |\mathcal{B}_1'| + |\mathcal{B}_1''| 
             \le 9\sqrt{\delta} t^2 n |\overline{D}|
                + 3|\overline{D}| + \frac{n}{2}+4\sqrt{\delta}n 
             \le 10\sqrt{\delta} t^2 n |\overline{D}|
                 + \left(\frac{n}{2}+4\sqrt{\delta}n\right). 
        \end{align*}
        This proves Claim~\ref{CLAIM:GenTuran-cycle-B1}. 
    \end{proof}

    To bound the size of $\mathcal{B}_0$ we divide each $D_i$ into two further smaller subsets. 
    More specifically, let 
        \begin{gather*}
            \tau' := \frac{n}{3}, 
            \quad\text{and}\quad
            D_i' := \left\{v\in D_i \colon |N_{G}(v)\cap V_{3-i}| \ge \tau'\right\}
            \quad\text{for}\quad i\in \{1,2\}. 
        \end{gather*}
        Let $D':= D_1'\cup D_2'$, $\overline{D'} := D\setminus D'$, and $\overline{D'}_i := D_i\setminus D_i'$ for $i\in \{1,2\}$.
        \begin{claim}\label{CLAIM:GenTuran-cycle-D'-in-D}
            We have $|\overline{D'}_i| \le 14\delta n$ for $i\in \{1,2\}$, and hence, 
            $|\overline{D'}| \le 28\delta n$.   
        \end{claim}
        \begin{proof}
            Let $i\in \{1,2\}$. 
            If $|\overline{D'}_i| > 14\delta n$, then, by~\eqref{equ:GenTuran-cycle-Vi-size}, the number of edges in $G[V_1, V_2]$ would satisfy 
            \begin{align*}
                |G[V_1, V_2]|
                \le |V_1||V_2| - 14\delta n \times \left(\frac{n}{2}- \sqrt{2\delta}n - \frac{n}{3}\right)
                < \frac{n^2}{4} - 2\delta n^2, 
            \end{align*}
            which contradicts~\eqref{equ:GenTuran-cycle-G[V1,V2]}. 
        \end{proof}
        \begin{claim}\label{CLAIM:GenTuran-cycle-GD1'D2'}
            We have $\left|G[D_1'] \cup G[D_2']\right| \le 1$. 
        \end{claim}
        \begin{proof}
            Suppose that $\{u,v\}$ is an edge in $G[D_i']$. 
            Then it follows from the definition of $D_i'$ that $\min\left\{|N_{G}(u)\cap V_{3-i}|, |N_{G}(v)\cap V_{3-i}|\right\} \ge n/3$. 
            Due to~\eqref{equ:GenTuran-cycle-Vi-size} and the Inclusion-exclusion principle, we have 
            \begin{align*}
                |V_{3-i} \cap N_{G}(u) \cap N_{G}(v)| 
                \ge 2\times \frac{n}{3} - \left(\frac{n}{2}+\sqrt{2\delta}n\right)
                \ge \frac{n}{100}. 
            \end{align*}
            This implies that the codegree of $uv$ in $\mathcal{B}[V']$ is at least $n/100$, since every vertex in $V_{3-i} \cap N_{G}(u) \cap N_{G}(v)$ forms a copy of $K_3$ with $\{u,v\}$.  
            
            Suppose to the contrary that there exist two distinct edges $e_1', e_2' \in G[D_1'] \cup G[D_2']$. Then the argument above shows that there exist two edges $e_1, e_2 \in \mathcal{B}[V']$ such that $e_i' \subseteq e_i$ for $i \in \{1,2\}$ and $e_1 \cap e_2 = e_1' \cap e_2'$. 
            If $e_1' \cap e_2' \neq \emptyset$, then $e_1$ and $e_2$ would contradict Claim~\ref{CLAIM:GenTuran-cycle-P_2}~\ref{equ:CLAIM:GenTuran-cycle-P_2}. 
            If $e_1' \cap e_2' = \emptyset$, then $e_1$ and $e_2$ would contradict Claim~\ref{CLAIM:GenTuran-cycle-P_2}~\ref{equ:CLAIM:GenTuran-cycle-B1-two-edges}. 
        \end{proof}
        Recall that every triple in $\mathcal{M}_2$ has empty intersection with $\overline{D}$. 
        \begin{claim}\label{CLAIM:GenTuran-even-cycle-M2}
            We have $m_2 \ge tn|\overline{D'}|/7$. 
        \end{claim}
        \begin{proof}
            It follows from the definition of $\overline{D'}$ that for $i\in \{1,2\}$ every vertex $v\in \overline{D'}_i$ has at most $\tau'$ neighbors in $D_{3-i}$.
            Since every triple containing $v$, $x_i$, and a vertex in $D_{3-i}\setminus N_{G}(v)$ is a member in $\mathcal{M}_2$, 
            the vertex $v$  contributes at least $t\left(|V_{3-i}| - |\overline{D}| -\tau'\right)$ triples to $\mathcal{M}_2$. 
            So,  it follows from~\eqref{equ:GenTuran-cycle-Vi-size} and Claim~\ref{CLAIM:GenTuran-cycle-D-bar} that 
            \begin{align*}
                m_2 
                 \ge \sum_{i\in\{1,2\}} t\left(|V_{3-i}| - |\overline{D}| -\tau'\right)|\overline{D'}_i| 
                 \ge t\left(\frac{n}{2}-\sqrt{2\delta}n - 8\delta tn -\frac{n}{3}\right)|\overline{D'}| 
                 \ge \frac{tn|\overline{D'}|}{7}. 
            \end{align*}
            This proves Claim~\ref{CLAIM:GenTuran-even-cycle-M2}. 
        \end{proof}
        \begin{claim}\label{CLAIM:GenTuran-even-cycle-B0}
            We have $|\mathcal{B}_0\setminus\mathcal{B}[D]| \le 9\sqrt{\delta}t^2 n |\overline{D'}| + t$. 
        \end{claim}
        \begin{proof}
            It follows from the definition that every triple in $\mathcal{B}_0\setminus\mathcal{B}[D]$ contains one vertex from $L$ and two vertices from $D_i$ for some $i\in \{1,2\}$.
            By Claim~\ref{CLAIM:GenTuran-cycle-GD1'D2'}, the number of triples in $\mathcal{B}_0\setminus\mathcal{B}[D]$ that have empty intersection with $\overline{D'}$ is at most $t$. 
            On the other hand, it follows from Claim~\ref{CLAIM:GenTuran-bad-deg-max} that the number of triples in $\mathcal{B}_0\setminus\mathcal{B}[D]$ that have nonempty intersection with $\overline{D'}$ is at most $t\times 9\sqrt{\delta}tn \left(|\overline{D'}_1| + |\overline{D'}_2|\right) = 9\sqrt{\delta}t^2 n |\overline{D'}|$. 
        \end{proof}
    If $|\overline{D}| + |\overline{D'}| =0$, then $D' = D$ and every triple in $\mathcal{B}$ contains at least two vertices in $D_1'$ or $D_2'$. 
    Combined with Claim~\ref{CLAIM:GenTuran-cycle-GD1'D2'}, we see that $|\mathcal{B}| \le \max\{|V_1|,\ |V_2|\} + t$. 
    Then it follows from some simple calculations that 
    \begin{align*}
        |\mathcal{H}| 
         \le |\mathcal{S}| + |\mathcal{B}| 
          \le \binom{t}{3} + \binom{t}{2}(n-t) + t |V_1||V_2| + \max\{|V_1|,\ |V_2|\} + t
            \le |\mathcal{S}_{\mathrm{bi}}^{+}(n,t)|, 
    \end{align*}
    and equality holds iff $G \cong S^{+}(n,t)$. 
    So we may assume that $|\overline{D}| + |\overline{D'}| \ge 1$. 
    Then 
    it follows from~\eqref{equ:GenTuran-cycle-B3}, Claims~\ref{CLAIM:GenTuran-cycle-D-bar},~\ref{CLAIM:GenTuran-cycle-B_2'},~\ref{CLAIM:GenTuran-cycle-B1},~\ref{CLAIM:GenTuran-even-cycle-M2}~\ref{CLAIM:GenTuran-even-cycle-B0} that 
        \begin{align*}
            |\mathcal{H}|
            & \le |\mathcal{S}| 
                     + |\mathcal{B}_0\setminus\mathcal{B}[D]|+ |\mathcal{B}[D]|
                    + |\mathcal{B}_1| + |\mathcal{B}_2| + |\mathcal{B}_3| -|\mathcal{M}_1| - |\mathcal{M}_2| \\
            & \le |\mathcal{S}_{\mathrm{bi}}\left(n, t\right)| 
                    + 9\sqrt{\delta}t^2 n |\overline{D'}| + t
                    + 10\sqrt{\delta} t^2 n |\overline{D}| +\frac{n}{2}+11\sqrt{\delta}tn
                    + \frac{n |\overline{D}|}{3}
                    + C|\overline{D}|^2  \\
            & \quad  - \frac{49}{100} n |\overline{D}|
                    - \frac{tn|\overline{D'}|}{7}\\
            & \le |\mathcal{S}_{\mathrm{bi}}\left(n, t\right)| +   \left\lceil\frac{n-t}{2}\right\rceil 
            + 2t + 11\sqrt{\delta} t n - \left(\frac{49}{100} n -10\sqrt{\delta} t^2 n - \frac{n}{3} - C \times 10 \delta t n\right)|\overline{D}| \\
            & \quad - \left(\frac{tn}{7} - 9\sqrt{\delta}t^2 n\right)|\overline{D'}| \\
            & \le |\mathcal{S}^{+}_{\mathrm{bi}}\left(n, t\right)|
                + t + 11\sqrt{\delta} t n - \frac{n}{10}\left(|\overline{D}| + |\overline{D'}|\right)
            < |\mathcal{S}^{+}_{\mathrm{bi}}\left(n, t\right)|
        \end{align*}
        a contradiction.
        This completes the proof of Theorem~\ref{THM:GenTuran-Cycle-Exact} for even $k$. 
\end{proof}
\section{Concluding remarks}
$\bullet$ It is a natural open question to extend theorems concerning trees in the present paper to other classes of trees. In particular, for trees $T$ satisfying $\tau(T) = \sigma(T)$ and containing critical edges. 
The main (and probably the only) barrier would be to extend stability theorems (Theorem~\ref{THM:Hypergraph-Tree-Stability} and~\ref{THM:GenTuran-Tree-Stability}) to these trees. 

$\bullet$ Recall that one of the key ingredients in proofs of generalized Tur\'{a}n theorems is Kruskal--Katona type theorems (Propositions~\ref{PROP:shadow-bound-tree},~\ref{APPENDIX:PROP:Turan-cycle-shadow}) for $F$-free hypergraphs. It would be interesting to find applications of other Kruskal--Katona type theorems (see e.g.~\cite{Kr63,Ka68,MV15Survey,FJKMV19a,FJKMV20,LM21a,LM21b,LM23,FK23}) in generalized Tur\'{a}n problems. 

$\bullet$ By utilizing a similar framework as presented in this paper, along with methods and results of F\"{u}redi~\cite{F14tree}, F\"{u}redi--Jiang~\cite{FJ15cycle,FJ15}, and F\"{u}redi--Jiang--Seiver~\cite{FJS14path}, theorems in Sections~\ref{SUBSEC:Intro:antiRamsey} and~\ref{SUBSEC:Intro:GenTuran} can be extended to all values of $r \ge 4$. 
We plan to address this in future work. 

$\bullet$ Given an $r$-graph and an integer $i\in [r-2]$, the $i$-th shadow $\partial_i\mathcal{H}$ of $\mathcal{H}$ is 
\begin{align*}
    \partial_i\mathcal{H}
    := \left\{A\in \binom{V(\mathcal{H})}{r-i} \colon \text{there exists $E\in \mathcal{H}$ such that $A\subseteq E$}\right\}. 
\end{align*}
In general, one could consider the following generalized Tur\'{a}n problem.

Fix integers $r \ge 3$, $i \in [r-2]$, and an $r$-graph $F$. 
\begin{gather*}
    \text{What is the maximum number of edges in an $n$-vertex $r$-graph $\mathcal{H}$ with $\partial_i F \not\subseteq \partial_i\mathcal{H}$\ ?}
\end{gather*}
Note that results in Section~\ref{SUBSEC:Intro:GenTuran} answered this question when $(r,i) = (3,1)$ and $F$ is the expansion of a cycle or certain trees. 
We hope that the framework used in these proofs could find more applications in this direction.


$\bullet$ By applying the following result, which is a slight modification of Lemma~5.3 in~{\cite{KMV17b}} (and the proof is the same), it becomes straightforward to extend the theorems concerning trees presented in this paper to encompass (some special class of) forests as well. One of the simplest examples would be a matching (see e.g. \cite{Erdos65,BKE76,HLS12,Frankl13,OY13,FK19,LTYZ23,GLP23,TLG23,ZCGZ23}). 

\begin{proposition}
    Let $F$ be a $k$-vertex forest and $T_1, \ldots, T_{\ell}$ are connected components of $F$. Suppose that $(I_i, R_i)$ is a crosscut pair of $T_i$ for $i\in [\ell]$. 
    There exists a $k$-vertex tree $T$ such that $F \subseteq T$, $\sigma(T) = \sigma(F)$, and $\left(\bigcup_{i\in[\ell]}I_i, \bigcup_{i\in[\ell]}R_i\right)$ is a crosscut pair of $T$.  Moreover, if $e\in T_i$ is a critical edge in $T_i$ for some $i\in [\ell]$, then $e$ is also a critical edge in $T$. 
    In particular, 
    \begin{itemize}
        \item if $T_i$ satisfies $\sigma(T_i) = \tau_{\mathrm{ind}}(T_i)$ for all $i\in [\ell]$, and 
        \item $T_j$ contains a critical edge for some $j\in [\ell]$, 
    \end{itemize}
    then $T$ satisfies $\sigma(T) = \tau_{\mathrm{ind}}(T)$ and contains a critical edge. 
\end{proposition}
Using arguments in~\cite{GLS13,BK14}, it seems not hard to extend theorems in the present paper further to a special class of graphs whose connected components are trees and cycles. We leave the details to interested readers. 

\section*{Acknowledgment}
We are very grateful to the referees for their careful reading of the manuscript and for providing many valuable comments.
%
\bibliographystyle{abbrv}
\bibliography{GeneralizedTuran}
\begin{appendix}
\section{Proof of Proposition~\ref{PROP:tree-decomposition}}\label{APPENDIX:SEC:PROP:tree-decomposition}
Recall the following statement of Proposition~\ref{PROP:tree-decomposition}. 
\begin{proposition}\label{APPENDIX:PROP:tree-decomposition}
    Suppose that $T$ is a tree and $(I, R)$ is a cross-cut pair of $T$. Then one of the following holds. 
    \begin{enumerate}[label=(\roman*)]
        \item There exists a vertex $v\in I$ such that all but one vertex in $N_T(v)$ are leaves in $T$. 
        \item There exists an edge $e\in R$ which is a pendant edge in $T$. 
    \end{enumerate} 
    In particular, if $I$ is maximum, then $(i)$ holds. 
\end{proposition}
\begin{proof}
    Define a set system $\mathcal{T}$ on $U:= V\setminus I$ by letting 
    \begin{align*}
        \mathcal{T} := R \cup \left\{N_T(v) \colon v\in I\right\}. 
    \end{align*}
    First we claim that $\mathcal{T}$ is linear, i.e. $|E_1 \cap E_2| \le 1$ for all distinct pairs $E_1, E_2 \in \mathcal{T}$.
    Indeed, suppose to the contrary that there exists a distinct pair $E_1, E_2 \in \mathcal{T}$ with $|E_1 \cap E_2| \ge 2$. Then from the definition of $R$, $E_1$ and $E_2$ cannot be both contained in $R$. So, by symmetry, we may assume that $E_1 = N_{T}(v)$ for some $v\in I$. If $E_2\in R$, then $|E_2|=2$, which implies that the induced subgraph of $T$ on $\{v\}\cup E_2$ is a copy of $K_3$, a contradiction. If $E_2 = N_{T}(v')$ for some $v'\in I\setminus \{v\}$, then the induced subgraph of $T$ on $\{v,v',u,u'\}$, where $\{u,u'\}\subseteq N_{T}(v) \cap N_{T}(v')$ are arbitrary two distinct vertices, is a copy of $C_4$, a contradiction. 
    This proves that $\mathcal{T}$ is linear. 

    Next, we claim that $\mathcal{T}$ is acyclic. 
    Suppose to the contrary that there exist distinct vertices $u_1, \ldots, u_{\ell} \subseteq U$ and distinct edges $E_1, \ldots, E_{\ell} \in \mathcal{T}$ such that $\{u_i, u_{i+1} \} \subseteq E_i$ for $i \in [\ell]$ (the indices are taken module $\ell$). 
    Define 
    \begin{align*}
        I_1 := \left\{i\in [\ell] \colon E_i \in \left\{N_T(v) \colon v\in I\right\}\right\}, 
        \quad\text{and}\quad
        I_2:= [\ell]\setminus I_1. 
    \end{align*}
    For every $i\in I_1$ we assume that $v_i$ is the vertex in $I$ such that $E_i = N_T(v_i)$. 
    Then it is easy to see that the induced subgraph of $T$ on $\{v_i \colon i\in I_1\} \cup \bigcup_{j\in I_2}E_j$ contains a cycle of length $2|I_1|+|I_2|$, a contradiction. This proves that $\mathcal{T}$ is cyclic. 

    Let $\mathcal{P} \subseteq \mathcal{T}$ be a maximal path.  Let us assume that $\mathcal{P} = E_1 \cdots E_{m}$ for some $m\ge 1$. Since $\mathcal{T}$ is linear and acyclic, the maximality of $\mathcal{P}$ implies that all but at most one vertex in $E_m$ have degree one in $\mathcal{T}$, meaning that either $(i)$ or $(ii)$ hold. This proves Proposition~\ref{PROP:tree-decomposition}.
\end{proof}
\section{Proof of Proposition~\ref{PROP:shadow-size}}\label{APPENDIX:SEC:PROP:shadow-size}
Recall the following statement of Proposition~\ref{PROP:shadow-size}. 
\begin{proposition}\label{APPENDIX:PROP:shadow-size}
    Let $F$ be a bipartite graph or an odd cycle. 
    For every $\delta>0$ there exists $\varepsilon>0$ such that for sufficiently large $n$, every $n$-vertex $F^{\triangle}$-free graph $G$ with at least $n^2/4 - \varepsilon n^2$ edges can be made bipartite by removing at most $\delta n^2$ edges.
    In particular, every $n$-vertex $F^{\triangle}$-free graph $G$ satisfies  
    \begin{align*}
        |G| \le \frac{n^2}{4} + o(n^2). 
    \end{align*} 
\end{proposition}
\begin{proof}
    Fix a constant $\delta>0$. 
    Let $\varepsilon>0$ be sufficiently small and $n$ be sufficiently large. 
    We may assume that $\varepsilon \le \delta/4$.
    Let $G$ be an $n$-vertex $F^{\triangle}$-free graph with at least $n^2/4 - \varepsilon n^2$ edges. 
    Since $F$ is a bipartite graph or an odd cycle and the associated $3$-graph $\mathcal{K}_{G}$ is $F^{3}$-free, 
    it follows from the well-known K\"{o}vari--S\'{o}s--Tur\'{a}n Theorem~\cite{KST54} the well-known edge-critical theorem of Simonovits~\cite{S68} and~{\cite[Proposition~1.1]{KMV15c}} that $|\mathcal{K}_{G}| \le n \cdot \mathrm{ex}(n,F) + \left(|F|+v(F)\right)n^2 \le n^{3-\alpha}$, where $\alpha >0$ is some constant depending only on $F$.
    This implies that $N(K_3, G) = |\mathcal{K}_{G}| \le n^{3-\alpha}$. 
    Therefore, by the Triangle Removal Lemma (see e.g.~\cite{RS78,AFKS00,Fox11}), $G$ contains a $K_3$-free subgraph $G'$ with $|G'| \ge |G| - \varepsilon n^2 \ge n^2/4 - 2\varepsilon n^2$. 
    Since $\varepsilon$ is sufficiently small, it follows from the Stability theorem of Simonovits~\cite{S68} that $G'$ contains a bipartite subgraph $G''$ with $|G''| \ge |G'| - \delta n^2/4 \ge |G| - \delta n^2$. 
    This completes the proof of Proposition~\ref{PROP:shadow-size}. 
\end{proof}
\section{Proof of Lemma~\ref{LEMMA:KMVa-lambda-sparse-tree}}\label{APPENDIX:Proof-LEMMA:KMVa-lambda-sparse-tree}
%
%
%
\begin{proposition}\label{APPENDIX:PROP:large-linear-subgraph}
    Let $r \ge 2$ and $i \in [r]$ be integers. 
    Every $r$-graph $\mathcal{H}$ contains a subgraph $\mathcal{H}'$ satisfying 
    \begin{align}\label{equ:APPENDIX:PROP:large-linear-subgraph}
        \Delta_i(\mathcal{H}') \le 1 
        \quad\text{and}\quad
        |\mathcal{H}'| \ge \frac{|\mathcal{H}|}{\binom{r}{i} \Delta_i(\mathcal{H})}. 
    \end{align}
\end{proposition}
\begin{proof}
    Define an auxiliary graph $G$ whose vertex set is $\mathcal{H}$ and a pair $\{e, e'\} \subseteq \mathcal{H}$ is an edge in $G$ iff $|e \cap e'| = i$. 
    Since every edge in $\mathcal{H}$ has exactly $\binom{r}{i}$ subsets of size $i$ and each $i$-subset is contained in at most $\Delta_i(\mathcal{H})$ edges in $\mathcal{H}$, the maximum degree of $G$ satisfies $\Delta(G) \le \binom{r}{i} \Delta_i(\mathcal{H})$. 
    A simply greedy argument shows that $G$ contains an independent set of size at least $\frac{|V(G)|}{\Delta(G)} \ge \frac{|\mathcal{H}|}{\binom{r}{i} \Delta_i(\mathcal{H})}$. 
    Since every independent set in $G$ corresponds to a subgraph of $\mathcal{H}$ whose maximum $i$-degree is at most one, we can choose the largest independent set $\mathcal{H}'$, and it satisfies~\eqref{equ:APPENDIX:PROP:large-linear-subgraph}. 
\end{proof}

Recall the following statement of Lemma~\ref{LEMMA:KMVa-lambda-sparse-tree}. 
\begin{lemma}\label{APPENDIX:LEMMA:KMVa-lambda-sparse-tree}
    Fix integers $k \ge 3$ and $C \ge 1$. 
    Let  $T_k \in \mathcal{T}_k$. 
    If $\mathcal{H}$ is an $n$-vertex $T_k^{3}$-free $3$-graph with $\Delta_2(\mathcal{H}) \le C$, then $|\mathcal{H}| \le 6 k C n = o(n^2)$.
\end{lemma}
\begin{proof}
    Suppose to the contrary that there exists an $n$-vertex $T_k^{3}$-free $3$-graph $\mathcal{H}$ with $\Delta_2(\mathcal{H}) \le C$ and $|\mathcal{H}| \ge 6 k C n$. 
    By Proposition~\ref{APPENDIX:PROP:large-linear-subgraph}, 
    there exist  a linear subgraph $\mathcal{H}' \subset \mathcal{H}$ of size $\frac{|\mathcal{H}|}{3C} \ge 2kn$. 
    By removing vertices one by one, it is easy to see that $\mathcal{H}'$ contains an induced subgraph $\mathcal{H}''$ with minimum degree at least $2k$. 
    A simple greedy argument shows that $T_{k}^{3} \subseteq \mathcal{H}''$, a contradiction. 
\end{proof}
\section{Proof of Lemma~\ref{LEMMA:Cleaning-Algo-tree}}\label{APPENDIX:SEC:Cleaning-Algo-tree}
Recall the following statement of Lemma~\ref{LEMMA:Cleaning-Algo-tree}.
\begin{lemma}\label{APPENDIX:LEMMA:Cleaning-Algo-tree}
    Let $T$ be a fixed tree with $\sigma(T) = \tau_{\mathrm{ind}}(T) =: t+1$ and $k:= |T|$.
    Fix constant $\varepsilon >0$ and let $n$ be sufficiently large. 
    Let $\mathcal{H}$ be an $n$-vertex $T^3$-free $3$-graph and $\mathcal{H}_q$ be the outputting $3$-graph of Cleaning Algorithm with input $(\mathcal{H}, k, t)$ as defined here.
    If $|\mathcal{H}| \ge t |\partial\mathcal{H}| - \varepsilon n^2$,  
    then $\mathcal{H}_q$ is $(t, 3k)$-superfull, and moreover, 
    \begin{align*}
        q 
        \le 12k \varepsilon n^2, \quad 
        |\mathcal{H}_q| 
        \ge |\mathcal{H}| - 48k^2 \varepsilon n^2, \quad\text{and}\quad
        |\partial\mathcal{H}_q| \ge |\partial\mathcal{H}| - 50k^2\varepsilon n^2. 
    \end{align*}
\end{lemma}
\begin{proof}
    We keep using notations in the Cleaning Algorithm. 
    First, it follows from Lemma~\ref{LEMMA:KMVa-lambda-sparse-tree} that $|\mathcal{H}^{\ast}| \le \varepsilon n^2$.  
    Hence we have 
    \begin{align*}
        |\mathcal{H}_0| 
            \ge t|\partial\mathcal{H}| - 2\varepsilon n^2. 
    \end{align*}
    Note from the definition of the Cleaning Algorithm that $|\partial\mathcal{H}_q| \le |\partial\mathcal{H}|-q$. 
    Combined with the inequality above, we obtain 
    \begin{align}\label{equ:APPENDIX:LEMMA:Cleaning-Algo-tree-2}
        |\mathcal{H}_0| 
        = |\mathcal{H}| - |\mathcal{H}^{\ast}|
            \ge t \left(|\partial\mathcal{H}_q|+q\right) - 2\varepsilon n^2.  
    \end{align}
    For $i\in \{1,2,3\}$ define
    \begin{align*}
        I_i := \left\{j \in [q] \colon \text{$f_j$ is of type-i in $\mathcal{H}_{j-1}$}\right\}. 
    \end{align*}
    It follows from the definition of type-3 and Lemma~\ref{LEMMA:l+1-shadows} that $|I_3| \le \varepsilon n^2$. 
           
    Observe from the definition of types-1 and 2 that if $j \in I_2$, then $j+1 \in I_1$. Therefore, $|I_2| \le |I_1|$, and combined with~\eqref{equ:APPENDIX:LEMMA:Cleaning-Algo-tree-2}, we obtain  
           \begin{align*}
               |\mathcal{H}_q|
               \ge |\mathcal{H}_0| - (t-1)|I_1| -t|I_2| - 3k |I_3|
               & \ge |\mathcal{H}_0| - \left(t-\frac{1}{2}\right)\left(|I_1|+|I_2|\right) - 3k |I_3| \\
               & \ge t \left(|\partial\mathcal{H}_q|+q\right) - 2\varepsilon n^2 - \left(t-\frac{1}{2}\right)q - 3k |I_3| \\
               & = t |\partial\mathcal{H}_q| + \frac{q}{2} - 5k\varepsilon n^2. 
           \end{align*}
    Combined with Proposition~\ref{PROP:shadow-bound-tree}, we obtain $q/2 - 5k\varepsilon n^2 \le \varepsilon' n^2$, which implies that $q \le 12k\varepsilon n^2$. 
    Hence, it follows from the definition of the Cleaning Algorithm that  
    \begin{align*}
        |\mathcal{H}_q| 
            \ge |\mathcal{H}_0| - 3kq 
            \ge |\mathcal{H}| - 48k^2 \varepsilon n^2. 
    \end{align*}
    Combined with $t|\partial\mathcal{H}_q| \ge |\mathcal{H}_q| - \varepsilon n^2$ (by Proposition~\ref{PROP:shadow-bound-tree}) and $|\mathcal{H}| \ge t|\partial\mathcal{H}| - \varepsilon n^2$ (by assumption), we obtain that $|\partial\mathcal{H}_q| \ge |\partial\mathcal{H}| - 50k^2\varepsilon n^2$.
\end{proof}
\section{Proof of Lemma~\ref{LEMMA:Cleaning-Algo-cycle}}\label{APPENDIX:SEC:Cleaning-Algo-cycle}
Recall the following statement of Lemma~\ref{LEMMA:Cleaning-Algo-cycle}. 
\begin{lemma}\label{APPENDIX:LEMMA:Cleaning-Algo-cycle}
    Let $k \ge 4$ and $t = \floor*{(k-1)/2}$ be integers.
    Fix a constant $\varepsilon >0$ and let $n$ be sufficiently large. 
    Let $\mathcal{H}$ be an $n$-vertex $C_{k}^3$-free $3$-graph and $\mathcal{H}_q$ be the outputting $3$-graph of the Cleaning Algorithm with input $(\mathcal{H}, k, t)$.
    If $|\mathcal{H}| \ge t |\partial\mathcal{H}| - \varepsilon n^2$,  
    then $\mathcal{H}_q$ is $(t, 3k)$-superfull, and moreover, 
    \begin{align*}
        q 
        \le 12k \varepsilon n^2, \quad 
        |\mathcal{H}_q| 
        \ge |\mathcal{H}| - 48k^2 \varepsilon n^2, \quad\text{and}\quad
        |\partial\mathcal{H}_q| \ge |\partial\mathcal{H}| - 50k^2\varepsilon n^2. 
    \end{align*}
\end{lemma}
\begin{proof}[Proof of Lemma~\ref{APPENDIX:LEMMA:Cleaning-Algo-cycle}]
    We keep using notations in the Cleaning Algorithm. 
    First, it follows from Lemma~\ref{LEMMA:KMVa-lambda-sparse-cycle} that $|\mathcal{H}^{\ast}| \le \varepsilon n^2$.  
    Hence we have 
    \begin{align*}
        |\mathcal{H}_0| 
            \ge t|\partial\mathcal{H}| - 2\varepsilon n^2. 
    \end{align*}
    Note from the definition of the Cleaning Algorithm that $|\partial\mathcal{H}_q| \le |\partial\mathcal{H}|-q$. 
    Combined with the inequality above, we obtain 
    \begin{align}\label{equ:APPENDIX:LEMMA:Cleaning-Algo-cycle-2}
        |\mathcal{H}_0| 
            \ge t \left(|\partial\mathcal{H}_q|+q\right) - 2\varepsilon n^2.  
    \end{align}
    For $i\in \{1,2,3\}$ define
    \begin{align*}
        I_i := \left\{j \in [q] \colon \text{$f_j$ is of type-i in $\mathcal{H}_{j-1}$}\right\}. 
    \end{align*}
    If $k$ is even,  then it follows from the definition of type-3 and Lemma~\ref{LEMMA:KMVa-many-ell+1-shadow-cycle}~(i) that $|I_3| \le \varepsilon n^2$. 
    Suppose that $k$ is odd. 
    Note that for every $i\in I_3$, the pair $e_i$ is contained in an edge $\widehat{e} \in \mathcal{H}_i$ with the property that every pair of vertices in $\widehat{e}$ has codegree more than $t \ge 2$.  
    Since $d_{\mathcal{H_0}}(e_i) \ge d_{\mathcal{H}_i}(e_i) \ge t+1$ and every edge in $\mathcal{H_0}$ contains a pair of vertices with codegree at least $3k+1$ in $\mathcal{H}$, it follows from Lemma~\ref{LEMMA:KMVa-many-ell+1-shadow-cycle}~(ii) that $|I_3| \le \varepsilon n^2$.
           
    Observe from the definition of types-1 and 2 that if $j \in I_2$, then $j+1 \in I_1$. Therefore, $|I_2| \le |I_1|$, and combined with~\eqref{equ:APPENDIX:LEMMA:Cleaning-Algo-cycle-2}, we obtain  
           \begin{align*}
               |\mathcal{H}_q|
               \ge |\mathcal{H}_0| - (t-1)|I_1| -t|I_2| - 3k |I_3|
               & \ge |\mathcal{H}_0| - \left(t-\frac{1}{2}\right)\left(|I_1|+|I_2|\right) - 3k |I_3| \\
               & \ge t \left(|\partial\mathcal{H}_q|+q\right) - 2\varepsilon n^2 - \left(t-\frac{1}{2}\right)q - 3k |I_3| \\
               & = t |\partial\mathcal{H}_q| + \frac{q}{2} - 5k\varepsilon n^2. 
           \end{align*}
    Combined with Proposition~\ref{APPENDIX:PROP:Turan-cycle-shadow}, we obtain $q/2 - 5k\varepsilon n^2 \le \varepsilon' n^2$, which implies that $q \le 12k\varepsilon n^2$. 
    Hence, it follows from the definition of the Cleaning Algorithm that  
    \begin{align*}
        |\mathcal{H}_q| 
            \ge |\mathcal{H}_0| - 3kq 
            \ge |\mathcal{H}| - 48k^2 \varepsilon n^2. 
    \end{align*}
    Combined with $t|\partial\mathcal{H}_q| \ge |\mathcal{H}_q| - \varepsilon n^2$ (by Proposition~\ref{APPENDIX:PROP:Turan-cycle-shadow}) and $|\mathcal{H}| \ge t|\partial\mathcal{H}| - \varepsilon n^2$ (by assumption), we obtain that $|\partial\mathcal{H}_q| \ge |\partial\mathcal{H}| - 50k^2\varepsilon n^2$.  
\end{proof}
\section{Proof of Lemma~\ref{LEMMA:KMVa-superfull-complete-bipartite-cycle}}\label{APPENDIX:SEC:LEMMA:KMVa-superfull-complete-bipartite-cycle}
\begin{proposition}\label{PROP:vario-cycle-length}
    Let $\mathcal{H}$ be a $3$-graph and $t \ge 1$ is an integer. 
    \begin{enumerate}[label=(\roman*)]
        \item It there exists a cycle $w_1 \cdots w_t \subseteq \partial\mathcal{H}$ of length $t$ and distinct vertices $v_1, \ldots, v_t \in V(\mathcal{H})\setminus\{w_1, \ldots, w_t\}$ such that 
        \begin{align*}
            \{w_{i}, v_{i+1}, w_{i+1}\} \in \mathcal{H},
            \quad\text{and}\quad 
            \min\left\{d_{\mathcal{H}}(w_i, v_{i+1}),\ d_{\mathcal{H}}(v_{i+1}, w_{i+1})\right\} \ge 6t+6
        \end{align*}
        for $i\in [t]$, where the indices are taken modulo $t$.
        Then $C_{\ell}^3 \subseteq\mathcal{H}$ for all integers $\ell\in [t,2t]$. 
        \item It there exists a path $w_1 \cdots w_{t+1} \subseteq \partial\mathcal{H}$ of length $t$ and distinct vertices $v_1, \ldots, v_t \in V(\mathcal{H})\setminus\{w_1, \ldots, w_t\}$ such that 
        \begin{align*}
            \{w_{i}, v_{i+1}, w_{i+1}\} \in \mathcal{H},
            \quad\text{and}\quad 
            \min\left\{d_{\mathcal{H}}(w_i, v_{i+1}),\ d_{\mathcal{H}}(v_{i+1}, w_{i+1})\right\} \ge 6t+6
        \end{align*}
        for $i\in [t]$.
        Then $P_{\ell}^3 \subseteq\mathcal{H}$ for all integers $\ell\in [t,2t]$. 
    \end{enumerate}
\end{proposition}
\textbf{Remark.} In fact, $6t+6$ in the assumptions above can be replaced by $4t$. 
%
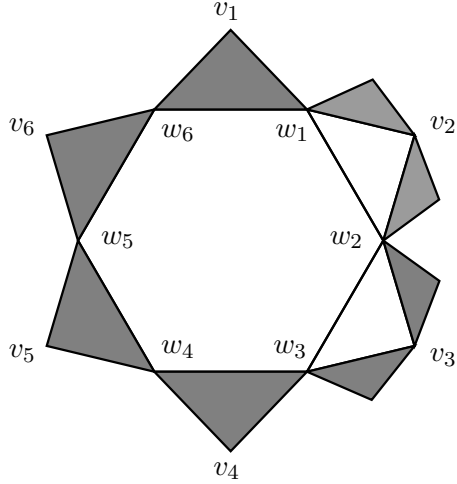
\begin{figure}[htbp]
\centering
\tikzset{every picture/.style={line width=0.85pt}} 

\begin{tikzpicture}[x=0.75pt,y=0.75pt,yscale=-1,xscale=1]
\draw   (398.33,141.17) -- (360.25,207.13) -- (284.08,207.13) -- (246,141.17) -- (284.08,75.2) -- (360.25,75.2) -- cycle ;
\draw  [fill={rgb, 255:red, 128; green, 128; blue, 128 }  ,fill opacity=1 ] (230.4,88.19) -- (284.08,75.2) -- (246,141.17) -- cycle ;
\draw  [fill={rgb, 255:red, 128; green, 128; blue, 128 }  ,fill opacity=1 ] (322.17,35.2) -- (360.25,75.2) -- (284.08,75.2) -- cycle ;
\draw   (413.93,88.19) -- (398.33,141.17) -- (360.25,75.2) -- cycle ;
\draw   (413.93,194.15) -- (360.25,207.13) -- (398.33,141.17) -- cycle ;
\draw  [fill={rgb, 255:red, 128; green, 128; blue, 128 }  ,fill opacity=1 ] (322.17,247.13) -- (284.08,207.13) -- (360.25,207.13) -- cycle ;
\draw  [fill={rgb, 255:red, 128; green, 128; blue, 128 }  ,fill opacity=1 ] (230.4,194.15) -- (246,141.17) -- (284.08,207.13) -- cycle ;
\draw  [fill={rgb, 255:red, 155; green, 155; blue, 155 }  ,fill opacity=1 ] (393.02,60.17) -- (413.93,88.19) -- (360.25,75.2) -- cycle ;
\draw  [fill={rgb, 255:red, 155; green, 155; blue, 155 }  ,fill opacity=1 ] (426.18,120.52) -- (398.33,141.17) -- (413.93,88.19) -- cycle ;
\draw  [fill={rgb, 255:red, 128; green, 128; blue, 128 }  ,fill opacity=1 ] (426.32,161.47) -- (413.93,194.15) -- (398.33,141.17) -- cycle ;
\draw  [fill={rgb, 255:red, 128; green, 128; blue, 128 }  ,fill opacity=1 ] (392.54,221.3) -- (360.25,207.13) -- (413.93,194.15) -- cycle ;

\draw (343,81) node [anchor=north west][inner sep=0.75pt]   [align=left] {$w_1$};
\draw (370,135) node [anchor=north west][inner sep=0.75pt]   [align=left] {$w_2$};
\draw (342,190) node [anchor=north west][inner sep=0.75pt]   [align=left] {$w_3$};
\draw (286,190) node [anchor=north west][inner sep=0.75pt]   [align=left] {$w_4$};
\draw (256,135) node [anchor=north west][inner sep=0.75pt]   [align=left] {$w_5$};
\draw (286,81) node [anchor=north west][inner sep=0.75pt]   [align=left] {$w_6$};
\draw (420,77) node [anchor=north west][inner sep=0.75pt]   [align=left] {$v_2$};
\draw (420,194) node [anchor=north west][inner sep=0.75pt]   [align=left] {$v_3$};
\draw (312,251) node [anchor=north west][inner sep=0.75pt]   [align=left] {$v_4$};
\draw (210,192) node [anchor=north west][inner sep=0.75pt]   [align=left] {$v_5$};
\draw (210,78) node [anchor=north west][inner sep=0.75pt]   [align=left] {$v_6$};
\draw (312,20) node [anchor=north west][inner sep=0.75pt]   [align=left] {$v_1$};
\end{tikzpicture}
\caption{Varying the length of a cycle.}
\label{fig:cycle-vari}
\end{figure}
\begin{proof}[Proof of Proposition~\ref{PROP:vario-cycle-length}]
    The proofs for both statements are essentially the same, so we will prove Statement~(i) only. 

    Let $w_1, \ldots, w_t, v_1, \ldots, v_t$ be vertices in $V(\mathcal{H})$ satisfying the assumptions in Statement~(i). 
    Fix $\ell \in [t,2t]$ and let $k:= \ell-t$. 
    Observe that $w_1v_2w_2 \cdots w_{k}v_{k+1}w_{k+1}w_{k+2}\cdots w_{t}w_1$ is a cycle of length $\ell$ in $\partial\mathcal{H}$ and it satisfies the assumptions in Fact~\ref{FACT:partial-embedding}. 
    Therefore, $C_{\ell}^3 \subseteq \mathcal{H}$ (see Figure~\ref{fig:cycle-vari}). 
\end{proof}

Recall the following statement of Lemma~\ref{LEMMA:KMVa-superfull-complete-bipartite-cycle}. 
\begin{lemma}\label{APPENDIX:LEMMA:KMVa-superfull-complete-bipartite-cycle}
    Let $k \ge 4$, $t = \left\lfloor(k-1)/2\right\rfloor$, and $\mathcal{H}$ be a $(t,3k)$-superfull $3$-graph containing two disjoint sets $W_1, W_2 \subseteq V(\mathcal{H})$ of size at least $3k$ such that every pair of vertices $(w_1, w_2) \in W_1 \times W_2$ has codegree exactly $t$ in $\mathcal{H}$. 
    If $C_k^3 \not\subseteq \mathcal{H}$, then there exists a $t$-set $L\subseteq V(\mathcal{H}) \setminus (W_1\cup W_2)$ such that $N_{\mathcal{H}}(w_1w_2) = L$ for all $(w_1, w_2) \in W_1 \times W_2$. 
\end{lemma}
\begin{proof}
    Let $V':= V(\mathcal{H})\setminus(W_1\cup W_2)$. 
    Let $(w_1, w_2) \in W_1 \times W_2$. 
    First, we prove that $N_{\mathcal{H}}(w_1w_2) \subseteq V(\mathcal{H}) \setminus W$. 
    Suppose to the contrary that there exists $w\in W_1 \cup W_2$ such that $w_1w_2w\in \mathcal{H}$. 
    By symmetry, we may assume that $w\in W_2$. 
    Since $\mathcal{H}$ is $t$-superfull and $d_{\mathcal{H}}(w_1w_2) = t$, it follows from the definition that $d_{\mathcal{H}}(w_1w) \ge 3k$, which contradicts the assumption that every pair of vertices $(w_1', w_2') \in W_1 \times W_2$ has codegree exactly $t$ in $\mathcal{H}$. Therefore, $N_{\mathcal{H}}(w_1w_2) \subseteq V(\mathcal{H}) \setminus W$. 

    Next we prove that there exists a $t$-set $L\subseteq V(\mathcal{H}) \setminus (W_1\cup W_2)$ such that $N_{\mathcal{H}}(w_1w_2) = L$ for all $(w_1, w_2) \in W_1 \times W_2$. 
    Suppose to the contrary that there exist two distinct pairs $(w_1, w_2), (w_1', w_2') \in W_1 \times W_2$ such that $N_{\mathcal{H}}(w_1w_2) \neq N_{\mathcal{H}}(w_1'w_2')$. 
    Since $\min\{|W_1|, |W_2|\} \ge 3k$, there exists $e_1\in W_1 \times W_2$ sharing exactly one vertex with $(w_1, w_2)$ and $(w_1', w_2')$ respectively, and moreover, $N_{\mathcal{H}}(e_1) \neq N_{\mathcal{H}}(w_1w_2)$ or $N_{\mathcal{H}}(e_1) \neq N_{\mathcal{H}}(w_1'w_2')$. 
    By symmetry, we may assume that $N_{\mathcal{H}}(e_1) \neq N_{\mathcal{H}}(w_1w_2)$. Let $e_2:= \{w_1, w_2\}$. 
 
    Suppose that $t \ge 2$ and $t$ is odd.
    Assume that $e_1 = \{w_1, w_{t+1}\}$. 
    Choose distinct vertices $w_{2i+1} \in W_1\setminus \{w_1\}$ for $1 \le i \le (t-1)/2$ and choose distinct vertices $w_{2i+2} \in W_2\setminus\{w_2,w_{t+1}\}$ for $1 \le i \le (t-3)/2$.
    Let $e_i := \{w_{i-1}, w_{i}\}$ for $3 \le i \le t+1$.
    It is clear that $e_1e_2 \ldots e_{t+1}$ is a cycle of length $t+1$. 
    Define an auxiliary bipartite graph $A[V_1, V_2]$, where $V_1:= \{e_1, \ldots, e_{t+1}\}$ and $V_2:= \bigcup_{i\in [t]} N_{\mathcal{H}}(e_i)$ and a pair $(e_i, v) \in V_1 \times V_2$ is an edge (in $A$) iff $e_i \cup \{v\} \in \mathcal{H}$. 
    Since $|N_{\mathcal{H}}(e_i)| = t$ and $\left|\bigcup_{i\in [t]} N_{\mathcal{H}}(e_i)\right| \ge t+1$ (as $N_{\mathcal{H}}(e_1) \neq N_{\mathcal{H}}(e_2)$), Hall's condition is satisfied. 
    Hence, $A$ contains a matching of size $t+1$, which implies that there exist  $t+1$ distinct vertices $v_1, \ldots, v_{t+1} \in V_2 \subseteq V'$ such that $e_i \cup \{v_i\} \in \mathcal{H}$ for all $i\in [t+1]$. 
    Since $\mathcal{H}$ is $(t,3k)$-superfull, $\min\left\{d_{\mathcal{H}}(w_{i-1} v_i),\ d_{\mathcal{H}}(w_{i} v_i)\right\} \ge 3k$ for all $i \in [t+1]$. 
    Therefore, by Proposition~\ref{PROP:vario-cycle-length}~(i), $C_{2t+1}^{3} \subseteq \mathcal{H}$ and $C_{2t+2}^{3} \subseteq \mathcal{H}$. 
    
    Suppose that $t \ge 2$ and $t$ is even. 
    Assume that $e_1 = \{w_1, w_{t+2}\}$. 
    Choose distinct vertices $w_{2i+1} \in W_1\setminus \{w_1\}$ for $1 \le i \le t/2$ and choose distinct vertices $w_{2i+2} \in W_2\setminus\{w_2,w_{t+2}\}$ for $1 \le i \le (t-2)/2$.
    Let $e_i := \{w_{i-1}, w_{i}\}$ for $3 \le i \le t+2$.
    It is clear that $e_1e_2 \ldots e_{t+2}$ is a cycle of length $t+2$.
    Consider the first $t+1$ edges. Similar to the argument above, it follows from Hall's theorem that there exist distinct vertices $v_1, \ldots, v_{t+1} \subseteq V'$ such that $e_i \cup \{v_i\} \in \mathcal{H}$ for all $i\in [t+1]$. 
    If there exists a vertex $v_{t+2} \in V' \setminus \{v_1, \ldots, v_{t+1}\}$ such that $e_{t+2} \cup \{v_{t+2}\} \in \mathcal{H}$, then it follows from Proposition~\ref{PROP:vario-cycle-length}~(i) that $C_{2t+1}^{3} \subseteq \mathcal{H}$ and $C_{2t+2}^{3} \subseteq \mathcal{H}$.
    So we may assume that $N_{\mathcal{H}}(e_{t+2}) \subseteq \{v_1, \ldots, v_{t+1}\}$. 
    Since $|N_{\mathcal{H}}(e_{t+2})| = t$, we have $N_{\mathcal{H}}(e_{t+2}) \cap \{v_1, v_{t+1}\} \neq\emptyset$. By symmetry, we may assume that $v_{t+1} \in N_{\mathcal{H}}(e_{t+2})$. 
    Observe that $w_1v_2w_2\cdots w_{t}v_{t+1}w_{t+1}w_{t+2}w_1$ is a copy of $C_{2t+2}$ in $\partial\mathcal{H}$. 
    Similarly, it follows from Proposition~\ref{PROP:vario-cycle-length}~(i) that $C_{2t+1}^{3} \subseteq \mathcal{H}$ and $C_{2t+2}^{3} \subseteq \mathcal{H}$.

    Suppose that $t=1$ and $k = 4$. 
    Assume that $e_1 = \{w_1, w_4\}$.
    Choose a vertex $w_3 \in W_1\setminus\{w_1\}$. 
    Let $e_3 = \{w_2, w_3\}$ and $e_4 = \{w_3, w_4\}$. 
    Since $N_{\mathcal{H}}(e_1) \neq N_{\mathcal{H}}(e_2)$, 
    there exist distinct vertices $v_1, v_2 \in V'$ such that $e_1 \cup \{v_1\}, e_2 \cup \{v_2\} \in \mathcal{H}$.
    If there exist distinct vertices $v_3, v_4 \in V\setminus\{v_1, v_2\}$ such that $e_3 \cup \{v_3\}, e_4 \cup \{v_4\} \in \mathcal{H}$, then we are done. 
    If $\{w_2, w_3, v_1\} \in \mathcal{H}$ (or $\{w_3, w_4, v_2\} \in \mathcal{H}$), then $v_1w_1v_2w_2v_1$ (or $v_1w_1v_2w_4v_1$) is a copy of $C_4$ and it follows from Fact~\ref{FACT:partial-embedding} that this $C_4$ can be extended to be a copy of $C_4^3$ in $\mathcal{H}$.
    If $\{w_2, w_3, v_2\} \in \mathcal{H}$ and there exists a vertex $v_4\in V'\setminus\{v_1, v_2\}$ such that $\{w_3, w_4, v_4\} \in \mathcal{H}$, then $w_1v_2w_3w_4w_1$ is a copy of $C_4$ and it follows from Fact~\ref{FACT:partial-embedding} that this $C_4$ can be extended to be a copy of $C_4^3$ in $\mathcal{H}$.
    By symmetry, the case $\{w_3, w_4, v_1\} \in \mathcal{H}$ and there exists a vertex $v_3\in V'\setminus\{v_1, v_2\}$ such that $\{w_2, w_3, v_3\} \in \mathcal{H}$ is also not possible. 
    If $\{w_2, w_3, v_2\} \in \mathcal{H}$ and $\{w_3, w_4, v_1\} \in \mathcal{H}$, then $v_1w_1v_2w_3v_1$ is a copy of $C_4$ and it follows from Fact~\ref{FACT:partial-embedding} that this $C_4$ can be extended to be a copy of $C_4^3$ in $\mathcal{H}$.
    This completes the proof of Lemma~\ref{LEMMA:KMVa-superfull-complete-bipartite-cycle}. 
\end{proof}
\section{Proof of Proposition~\ref{PROP:Turan-augmentation}}\label{APPENDIX:SEC:HypergraphTuran-Tree-Aug}
Recall the following statement of Proposition~\ref{PROP:Turan-augmentation}. 
\begin{proposition}\label{APPENDIX:PROP:Turan-augmentation}
        Suppose that $T$ is a tree and $F$ is obtained from $T$ by adding one edge. 
        Then $\mathrm{ex}(n, F^3) \le 3\left(|T|+1\right) n^2$ holds for sufficiently large $n$. 
\end{proposition}
\begin{proof}
    Let $T$ be a tree and $F$ be a graph obtained from $T$ by adding one edge. 
    Let $k:= |F| = |T|+1$. 
    Let $n$ be a sufficiently large integer and $\mathcal{H}$ be an $n$-vertex $3$-graph with $3k n^2$ edges. 
    Let $V:= V(\mathcal{H})$ and $\mathcal{H}' \subseteq \mathcal{H}$ be a maximum $3k$-full subgraph. 
    It follows from Lemma~\ref{LEMMA:KMVa-ell+1-full} that $|\mathcal{H}'| \ge |\mathcal{H}| - 3k \binom{n}{2} \ge 3k n^2/2$.
    We will show that $F^3 \subseteq \mathcal{H}'$. 

    If $F$ is a forest, then it follows from the theorem of  Kostochka--Mubayi--Verstra\"{e}te~\cite{KMV17b} (or Theorem~\ref{THM:Hypergraph-Tree-Stability}) that $F^3 \subseteq \mathcal{H}$. 
    So we may assume that $F$ is not a forest. 
    Then $F$ contains a unique cycle of length $\ell$ for some integer $\ell \le k$. 
    Let $p := k-\ell$. 
    By sequentially removing pendant edges we obtain a sequence of subgraphs 
    \begin{align*}
        F =: F_0 \supseteq F_1 \cdots \supseteq F_{p} \cong C_{\ell}, 
    \end{align*}
    where $|F_{i}| = |F| - i$ for $i \in [0,p]$ and $F_p$ is the unique cycle in $F$.
    In other words, for $i \in [0,p-1]$ we choose a pendant edge  $e_i \in F_i$ and let $F_{i+1} := F_i - e_i$. 
    We will prove by a backward induction on $i$ such that $F_i^{3} \subseteq \mathcal{H}'$. 
 
    The base case $i=p$ follows from the theorem of  Kostochka--Mubayi--Verstra\"{e}te~\cite{KMV15a} on linear cycles. 
    So we may assume that $i\in [p-1]$. 
    Suppose that there exists an embedding $\psi \colon F_i^3 \to \mathcal{H}'$.
    We want to extend $\psi$ to be an embedding of $F_{i-1}^3$ to $\mathcal{H}'$. 
    To achieve this, let $e_{i-1} := \{u,v\}$ denote the edge in $F_{i-1}\setminus F_i$ and assume that $v$ is a leaf in $F_{i-1}$. 
    Observe from the definition that each $F_i$ is connected, so there exists $w \in V(F_i)$ such that $uw$ is an edge in $F_i$. 
    This implies that $\psi(uw) \in \partial\mathcal{H}'$. 
    Since $\mathcal{H}'$ is $3k$-full, there exists a vertex $v' \in V \setminus \psi(V(F_i^3))$  such that $\{\psi(u), \psi(w), v'\} \in \mathcal{H}'$. 
    In particular, $\{\psi(u), v'\} \in \partial\mathcal{H}'$. 
    Using the $3k$-fullness again, there exists a vertex $v'' \in V \setminus \psi(V(F_i^3))$ such that $\{\psi(u), v', v''\} \in \mathcal{H}'$. 
    Observe that $\psi(F_{i}^3) \cup \{\{\psi(u), v', v''\}\}$ is a copy of $F_{i-1}^3$ in $\mathcal{H}'$.  
    This completes the proof for the inductive step, and hence, completes the proof of Proposition~\ref{PROP:Turan-augmentation}. 
\end{proof}

\section{Proof of Theorem~\ref{THM:GenTuran-Tree-Stability}}\label{APPENDIX:SEC:GenTuran-tree-stability}
Recall the following statement of Theorem~\ref{THM:GenTuran-Tree-Stability}. 
\begin{theorem}\label{APPENDIX:THM:GenTuran-Tree-Stability}
    Let $T$ be a tree with $\sigma(T) = \tau_{\mathrm{ind}}(T)$. 
    For every $\delta>0$ there exist $\varepsilon>0$ and $n_0$ such that the following holds for all $n\ge n_0$. 
    Suppose that $G$ is an $n$-vertex $T^{\triangle}$-free graph with
    \begin{align*}
        N(K_{3}, G)
        \ge \left(\frac{\sigma(T)-1}{4}-\varepsilon\right)n^2.   
    \end{align*}
    Then $G$ is $\delta$-close to $S(n,\sigma(T)-1)$. 
\end{theorem}
\begin{proof}[Proof of Theorem~\ref{APPENDIX:THM:GenTuran-Tree-Stability}]
    Let $T$ be a tree with $\sigma(T) = \tau_{\mathrm{ind}}(T)$.
    Fix $\delta>0$.  
    Let $\varepsilon>0$ be sufficiently small and $n$ be sufficiently large. 
    Let $k := |T|$ and $t := \sigma(T) - 1$. 
    Let $G$ be an $n$-vertex $T^{\triangle}$-free graph with $N(K_3, G) \ge \left(t/4-\varepsilon\right)n^2$. 
    We may assume that every edge in $G$ is contained in some triangle in $G$, since otherwise we can remove it from $G$ and this does not affect the value of $N(K_3, G)$. 
        Let $\mathcal{H} := \mathcal{K}_{G}$ and recall that 
        \begin{align*}
            \mathcal{K}_{G} 
            := \left\{e \in \binom{V(G)}{3} \colon G[e] \cong K_3\right\}. 
        \end{align*}
        Let $\mathcal{H}_q$ denote the outputting $3$-graph of the Cleaning Algorithm with input $\left(\mathcal{H}, k, t\right)$ as defined here. 
        Recall from Fact~\ref{FACT:shadow-Turan} that $\mathcal{H}$ is $T^3$-free. 
        In addition, it follows from Proposition~\ref{PROP:shadow-size} that 
        \begin{align}\label{equ:APPENDIX:THM:GenTuran-Tree-Stability-0}
            |\mathcal{H}|
            = N(K_{3}, G)
            \ge t \frac{n^2}{4}  -\varepsilon n^2
            \ge t \left(|G|-\varepsilon n^2\right) -\varepsilon n^2
            = t |\partial\mathcal{H}| - \varepsilon_1 n^2, 
        \end{align}
        where $\varepsilon_1 := (t+1)\varepsilon$. 
        Therefore, by Lemma~\ref{LEMMA:Cleaning-Algo-tree}, 
        \begin{align}\label{equ:APPENDIX:THM:GenTuran-Tree-Stability-1}
            q 
            \le 12k \delta_1 n^2, \quad 
            |\mathcal{H}_q| 
            \ge |\mathcal{H}| - 48k^2 \delta_1 n^2, \quad\text{and}\quad
            |\partial\mathcal{H}_q| \ge |\partial\mathcal{H}| - 50k^2 \delta_1 n^2.
        \end{align}
    First, observe from Proposition~\ref{PROP:shadow-bound-tree} that 
    \begin{align}\label{equ:APPENDIX:THM:GenTuran-Tree-Stability-2}
        |\partial\mathcal{H}|
        \ge \frac{|\mathcal{H}|-\varepsilon n^2}{t}
        \ge \frac{n^2}{4} - 2\varepsilon n^2. 
    \end{align}
    Let $\mathcal{H}_q$ be the output $3$-graph of the Cleaning Algorithm with input $(\mathcal{H}, k, t)$ as defined here. 
    It follows from~\eqref{equ:APPENDIX:THM:GenTuran-Tree-Stability-0} and Lemma~\ref{LEMMA:Cleaning-Algo-tree} that $\mathcal{H}_q$ is $(t, 3k)$-superfull and
    \begin{align}\label{equ:APPENDIX:THM:GenTuran-Tree-Stability-3}
        q 
            \le 12k \varepsilon_1 n^2, \quad 
        |\mathcal{H}_q| 
            \ge |\mathcal{H}| - 48k^2 \varepsilon_1 n^2,  \quad\text{and}\quad
        |\partial\mathcal{H}_q| 
            \ge |\partial\mathcal{H}| - 50k^2\varepsilon_1 n^2.
    \end{align}
    Define graphs 
    \begin{align*}
        G' 
            := \left\{e\in \partial\mathcal{H}_q \colon d_{\mathcal{H}_q}(e) = t\right\}, 
            \quad\text{and}\quad
        G''
            := \left\{e\in \partial\mathcal{H}_q \colon d_{\mathcal{H}_q}(e) \ge 3k\right\}. 
    \end{align*}
    Since $\mathcal{H}_q$ is $(t, 3k)$-superfull, we have $G' \cup G'' = \partial\mathcal{H}_q$ and every edge in $\mathcal{H}_q$ contains at least two elements from $G''$.
    Further, let $V_1 \cup V_2 = V(G)$ be a bipartition such that the number of edges in $G'$ crossing $V_1$ and $V_2$ is maximized. 
    \begin{claim}\label{APPENDIX:CLAIM:GenTuran-tree-l-shadow}
        We have $|G''| \le kn$ and $|G'[V_1, V_2]| \ge \left(1/4 - \varepsilon_2\right)n^2$. 
    \end{claim}
    \begin{proof} 
        It follows from Fact~\ref{FACT:partial-embedding} that $G''$ is $T$-free, Therefore, $|G''| \le kn$. 
        In addition, by~\eqref{equ:APPENDIX:THM:GenTuran-Tree-Stability-2} and the third inequality in~\eqref{equ:APPENDIX:THM:GenTuran-Tree-Stability-3}, we obtain 
        \begin{align*}
            |G'| 
            = |\partial\mathcal{H}_q|- |G''| 
            \ge \frac{n^2}{4}-2\varepsilon n^2 -  50k^2\varepsilon_1 n^2 - \varepsilon n^2
            \ge \frac{n^2}{4}-\frac{\varepsilon_2}{2} n^2. 
        \end{align*}
        By Proposition~\ref{PROP:shadow-size}, $G$ contains a bipartite graph $G_{\mathrm{bi}}$ such that $|G\setminus G_{\mathrm{bi}}| \le \varepsilon n^2$. 
        Therefore, $G'$ contains a bipartite graph $G'_{\mathrm{bi}}$ with 
        \begin{align*}
            |G'_{\mathrm{bi}}|
            = |G'| - |G'\setminus G'_{\mathrm{bi}}|
            \ge |G'| - |G\setminus G_{\mathrm{bi}}|
            \ge \frac{n^2}{4}-\frac{\varepsilon_2}{2} n^2 - \varepsilon n^2
            \ge \frac{n^2}{4}-\varepsilon_2 n^2. 
        \end{align*}
        This proves Claim~\ref{APPENDIX:CLAIM:GenTuran-tree-l-shadow}. 
    \end{proof}
    For convenience, let $G'_{\mathrm{bi}} := G'[V_1, V_2]$. 
    Define a $3$-graph $\mathcal{G} \subseteq \mathcal{H}_q$ as follows:
    \begin{align*}
        \mathcal{G} := \left\{e\in \mathcal{H}_q \colon \left|\binom{e}{2} \cap G'_{\mathrm{bi}}\right| = 1 \right\}. 
    \end{align*}
    Since every pair in $G'_{\mathrm{bi}}$ contributes exactly $t$ triples to $\mathcal{G}$, 
    it follows from the proof of Claim~\ref{APPENDIX:CLAIM:GenTuran-tree-l-shadow} that 
    \begin{align}\label{equ:APPENDIX:THM:GenTuran-Tree-Stability-4}
        |\mathcal{G}|
            \ge t |G'_{\mathrm{bi}}|
            \ge \left(\frac{1}{4} - \varepsilon_2\right)tn^2. 
    \end{align}
    For every vertex $v\in V$ define the \textbf{light link} of $v$ as 
    \begin{align*}
        \widehat{L}(v) := \left\{e\in G'_{\mathrm{bi}} \colon e\cup \{v\} \in \mathcal{H}_q\right\}. 
    \end{align*}
    It is clear from the definitions that 
    \begin{align}\label{equ:APPENDIX:THM:GenTuran-Tree-Stability-5}
        |\mathcal{G}| = \sum_{v\in V} |\widehat{L}(v)|, \quad\text{and}\quad
        \widehat{L}(v) \subseteq \binom{N_{G''}(v)}{2} \quad\text{for all}\quad v\in V(G). 
    \end{align}
    Let $V:= V(\mathcal{H})$, and for every integer $i\ge 0$ define 
    \begin{align*}
        Z_i := \left\{v\in V \colon |\widehat{L}(v)| \ge \frac{n^{2-2^{-i}}}{(\log n)^2}\right\}, 
        \quad U_i:= V\setminus Z_i, 
        \quad\text{and}\quad 
        \mathcal{G}_i := \left\{e\in \mathcal{G} \colon |e\cap Z_i|=1\right\}. 
    \end{align*}
    Observe that for all $i\ge 1$, 
    \begin{align*}
        Z_0 \supseteq Z_1 \supseteq \cdots \supseteq Z_{i}, \quad
        U_0 \subseteq U_1 \subseteq \cdots \subseteq U_i, \quad\text{and}\quad
        \mathcal{G}_0 \supseteq \mathcal{G}_1 \supseteq \cdots \supseteq \mathcal{G}_i. 
    \end{align*}
    \begin{claim}\label{APPENDIX:CLAIM:GenTuran-tree-Z_k}
        We have $|Z_i| \le 2kn^{2^{-i-1}} \log n$ and $|\mathcal{G}_i|\ge |\mathcal{G}| - \frac{2(i+1)n^2}{(\log n)^2}$ for all $i \ge 0$. 
        In particular, 
        \begin{align*}
            \binom{|Z_k|}{t} \le \sqrt{n} 
            \quad\text{and}\quad
            |\mathcal{G}_k| 
                \ge \left(\frac{1}{4} - 2\varepsilon_2\right)tn^2. 
        \end{align*}
    \end{claim}
    \begin{proof}
        We prove it by induction on $i$. 
        For the base case $i=0$, 
        observe from~\eqref{equ:APPENDIX:THM:GenTuran-Tree-Stability-5} that 
        for every $v\in Z_0$ we have $d_{G''}(v) \ge \sqrt{|\widehat{L}(v)|} \ge n^{1/2}/\log n$ (note that the worse case is that $\widehat{L}(v)$ is a complete graph on $d_{G''}(v)$ vertices). 
        Therefore, by Claim~\ref{APPENDIX:CLAIM:GenTuran-tree-l-shadow}, 
        \begin{align*}
            2kn
            \ge 2|G''| 
            \ge \sum_{v\in Z_0}d_{G''}(v) 
            \ge \sum_{v\in Z_0}\sqrt{|\widehat{L}(v)|} 
            \ge |Z_0|\frac{n^{1/2}}{\log n}, 
        \end{align*}
        which implies that $|Z_0| \le 2kn^{1/2}\log n$. 
        Consequently, the number of edges in $G'$ that have nonempty intersection with $Z_0$ is at most $|Z_0|n \le 2kn^{3/2}\log n$.
        Observe that every edge in $\mathcal{G}$ with at least two vertices in $Z_0$ must contain an edge in $G'$ with (at least) one endpoint in $Z_0$. Therefore, the number of edges in $\mathcal{G}$ with at least two vertices in $Z_0$ is at most $2k t n^{3/2}\log n$ (recall that each edge in $G'$ contributes $t$ triples in $\mathcal{G}$). 
        Therefore, 
        \begin{align*}
            |\mathcal{G}_0| 
                 \ge  |\mathcal{G}| - \sum_{v \in V\setminus Z_0}|\widehat{L}(v)| -2k t n^{3/2}\log n 
                 & \ge |\mathcal{G}| - n \times \frac{n}{(\log n)^2} - 2k t n^{3/2}\log n \\
                 & \ge |\mathcal{G}| - \frac{2n^2}{(\log n)^2}. 
        \end{align*}
        Now suppose that $|Z_i| \le n^{2^{-i-1}}\log n$ holds for some $i \ge 0$.
        Repeating the argument above, we obtain 
        \begin{align*}
            2kn
            \ge 2|G''| 
            \ge \sum_{v\in Z_{i+1}}d_{G''}(v) 
            \ge \sum_{v\in Z_{i+1}}\sqrt{|\widehat{L}(v)|} 
            \ge |Z_{i+1}|\frac{n^{1-2^{-i-2}}}{\log n}, 
        \end{align*}
        which implies that $|Z_{i+1}| \le 2kn^{2^{-i-2}}\log n$.
        So it follows from the inductive hypothesis that 
        \begin{align*}
            |\mathcal{G}_{i+1}| 
                 \ge  |\mathcal{G}_i| - \sum_{v \in Z_i\setminus Z_{i+1}}|\widehat{L}(v)| -t|Z_{i+1}|n
                 & \ge |\mathcal{G}_i| - |Z_i|\times \frac{n^{2-2^{-i-1}}}{(\log n)^2} -  2ktn^{1+2^{-i-2}}\log n\\
                 & \ge |\mathcal{G}_i| - \frac{2n^2}{(\log n)^2}
                 \ge |\mathcal{G}| - \frac{2(i+2)n^2}{(\log n)^2}. 
        \end{align*}
        Here, $t|Z_{i+1}|n$ is an upper bound for the number of edges in $\mathcal{G}_{i}$ with at least two vertices in $Z_{i+1}$. 
    \end{proof}
    %
    Let $\mathcal{G}_k'$ denote the output $3$-graph of the Cleaning Algorithm with input $\left(\mathcal{G}_k, k, t\right)$ as defined here. 
    Since $|\mathcal{G}_k| \ge t n^2/4 - 2\varepsilon_2 tn^2 \ge  t|G| - 3\varepsilon_2 tn^2 \ge t|\partial\mathcal{G}_k| - 3\varepsilon_2 tn^2$, it follows from Lemma~\ref{LEMMA:Cleaning-Algo-tree} that 
    \begin{align*}
        |\mathcal{G}_k'|
        \ge \left(\frac{t}{4} - \varepsilon_3\right) n^2. 
    \end{align*}
    Recall from the definition that every edge in $\mathcal{G}_k'$ contains exactly one vertex in $Z_{k}$. 
    Similar to the proof of Claim~\ref{APPENDIX:CLAIM:GenTuran-tree-l-shadow}, the number of pairs in $\binom{U_k}{2}$ that have codegree exactly $t$ in $\mathcal{G}_k'$ is at least $(1/4-\varepsilon_3)n^2 - |Z_k|n \ge (1/4-2\varepsilon_3)n^2$. 
    Therefore, the $3$-graph 
    \begin{align*}
        \mathcal{G}_k''
            := \left\{e \in \mathcal{G}_k' \colon \binom{e}{2} \cap \binom{U_k}{2} \text{ has codegree exactly $t$ }\right\}
    \end{align*}
    satisfies 
    \begin{align}\label{Appendix:equ:THM:Hypergraph-Tree-Stability-5}
        |\mathcal{G}_k''|
        \ge \left(\frac{1}{4}-2\varepsilon_3\right) t n^2. 
    \end{align}
    Define 
    \begin{align*}
        \mathcal{Z} := \left\{T\in \binom{Z_k}{t} \colon |L_{\mathcal{G}_k''}(T)| \ge 4kn\right\},  
    \end{align*}
    where $L_{\mathcal{G}_k''}(T) := \bigcap_{v\in T}L_{\mathcal{G}_k''}(v)$. 
    For every $T\in \mathcal{Z}$ let $L'(T)$ be a maximum induced subgraph of $L_{\mathcal{G}_k''}(T)$ with minimum degree at least $3k$. 
    By greedily removing vertex with degree less than $3k$ we have  $|L'(T)| \ge |L_{\mathcal{G}_k''}(T)| - 3kn >0$.

    %
    Let $\mathrm{Supp}_T$ denote the set of vertices in the graph $L'(T)$ with positive degree.
    Using Proposition~\ref{PROP:embed-T3-two-t-sets}, we obtain the following claim. 
    \begin{claim}\label{Appendix:CLAIM:Turan-disjoint-support}
        We have $\mathrm{Supp}_T \cap \mathrm{Supp}_{T'}= \emptyset$ for all distinct sets $T, T' \in \mathcal{Z}$. 
    \end{claim}

    Since $L'(T)$ is a bipartite graph on $\mathrm{Supp}_T$ for all $T \in \mathcal{Z}$ and $\sum_{T\in \mathcal{Z}}|L_{\mathcal{G}_k''}(T)| = |\mathcal{G}_k''|/t$, it follows from~\eqref{Appendix:equ:THM:Hypergraph-Tree-Stability-5} that 
    \begin{align*}
        \sum_{T\in \mathcal{Z}} \frac{|\mathrm{Supp}_T|^2}{4}
        \ge \sum_{T\in \mathcal{Z}}|L'(T)| 
        \ge \sum_{T\in \mathcal{Z}}\left(|L_{\mathcal{G}_k''}(T)|-3kn\right)
        & \ge \left(\frac{1}{4}-2\varepsilon_3\right) n^2 - 3kn\binom{|Z_{k}|}{t} \\
        & \ge \left(\frac{1}{4}-3\varepsilon_3\right) n^2. 
    \end{align*}
    It follows from Claim~\ref{CLAIM:Turan-disjoint-support} that $\sum_{T\in \mathcal{Z}} |\mathrm{Supp}_T| \le |U_k| \le n$. 
    Therefore, the inequality above and some simple calculations show that 
    \begin{align*}
        \max\left\{|\mathrm{Supp}_T| \colon T \in \mathcal{Z}\right\} 
        \ge \left(1-\sqrt{6\varepsilon_3}\right)n. 
    \end{align*}
    Let $T_{\ast} \in \mathcal{Z}$ be the unique $t$-set with $|\mathrm{Supp}_{T_{\ast}}|\ge \left(1-\sqrt{6\varepsilon_3}\right)n$. 
    Then, by Claim~\ref{Appendix:CLAIM:Turan-disjoint-support}, 
    \begin{align*}
        |L_{\mathcal{G}_k''}(T_{\ast})|
        = |\mathcal{G}_{k}''|/t - \sum_{T\in \mathcal{Z}-T_{\ast}}|L_{\mathcal{G}_k''}(T)|
        & \ge \left(\frac{1}{4}-2\varepsilon_3\right) n^2 - \binom{\sqrt{6\varepsilon_3}n}{2} \\
        & \ge \left(\frac{1}{4}-\delta \right) n^2. 
    \end{align*}
    Here we used the inequality that $\sum_{i}x_i^2 \le \left(\sum_{i}x_i\right)^2$. 
    In other words, the number of edges in $\mathcal{G}_k'' \subseteq \mathcal{H}$ with exactly one vertex in $T_{\ast}$ is at least $\left(1/4-\delta \right) n^2$. 
    This completes the proof of  Theorem~\ref{THM:GenTuran-Tree-Stability}. 
\end{proof}
\section{Proofs of Theorems~\ref{THM:GenTuran-Tree-Exact} and~\ref{THM:GenTuran-Tree-Asymp}}\label{APPENDIX:SEC:GenTuran-tree-exact}
Recall the following statement of Theorem~\ref{THM:GenTuran-Tree-Asymp}.
\begin{theorem}\label{APPENDIX:THM:GenTuran-Tree-Asymp}
    For every tree $T$ we have 
    \begin{align*}
        \mathrm{ex}(n, K_{3}, T^{\triangle})
        \le  \left|\mathcal{S}_{\mathrm{bi}}\left(n,\sigma(T)-1\right)\right| + o(n^2). 
    \end{align*}
\end{theorem}
\begin{proof}[Proof of Theorem~\ref{APPENDIX:THM:GenTuran-Tree-Asymp}]
    It follows immediately from Propositions~\ref{PROP:shadow-size} and~\ref{PROP:shadow-bound-tree}.  
\end{proof}
Recall the following statement of Theorem~\ref{THM:GenTuran-Tree-Exact}.
\begin{theorem}\label{APPENDIX:THM:GenTuran-Tree-Exact}
    Suppose that $T$ is a strongly edge-critical tree. Then  for sufficiently large $n$, 
    \begin{align*}
        \mathrm{ex}(n, K_{3}, T^{\triangle})
        = \left|\mathcal{S}_{\mathrm{bi}}\left(n,\sigma(T)-1\right)\right|. 
    \end{align*}
\end{theorem}
\begin{proof}[Proof of Theorem~\ref{APPENDIX:THM:GenTuran-Tree-Exact}]
    Let $T$ be a strongly edge-critical tree.  
    Let $I\cup J = V(T)$ be a partition such that $I$ is a minimum independent vertex cover of $T$. 
    Let $e_{\ast} := \{u_0, v_0\}$ be a pendant critical edge such that $v_0 \in I$ is a leaf.  
    Let $\delta >0$ be sufficiently small and $n$ be sufficiently large. 
    Let $t := \sigma(T) - 1$, $q := |\mathcal{S}_{\text{bi}}(n,t)|$, and $G$ be an $n$-vertex $T^{\triangle}$-free graph with $N(K_3, G) = q$. 
    Note that we may assume that every edge in $G$ is contained in some triangle of $G$, since otherwise we may delete it from $G$ and this does not change the value of $N(K_3, G)$. 
    Recall that our aim is to prove that $G \cong S(n,t)$. 

    Since $N(K_3, G) = (t/4-o(1))n^2$, it follows from Theorem~\ref{THM:GenTuran-Tree-Stability} that there exists a $t$-set $L:= \{x_1, \ldots, x_{t}\} \subseteq V(G)$ such that 
    \begin{enumerate}[label=(\roman*)]
        \item $|G-L| \ge \left(1/4-\delta\right)n^2$,  
        \item $N(K_3, G-L) \le \delta n^2$, 
        \item $G-L$ can be made bipartite by removing at most $\delta n^2$ edges, and
        \item $d_{G}(v) \ge (1-\delta)n$ for all $v\in L$. 
    \end{enumerate}
    Let $V := V(G)$, $V':= V-L$, and $V_1 \cup V_2 = V'$ be a bipartition of $V\setminus L$ such that the number of crossing edges between $V_1$ and $V_2$ is maximized. 
    Define 
    \begin{align*}
        \mathcal{S} &:= \left\{E \in \binom{V}{3} \colon |E \cap L| \ge 1,\ |E\cap V_1| \le 1,\ |E\cap V_2| \le 1\right\}, \\
        \mathcal{H} & := \mathcal{K}_{G} = \left\{E \in \binom{V}{3} \colon G[E] \cong K_3\right\}, \quad
        \mathcal{B} := \mathcal{H} \setminus  \mathcal{S}, \quad\mathrm{and}\quad
        \mathcal{M} := \mathcal{S}\setminus  \mathcal{H}. 
    \end{align*}
    %
    It follows from Statements~(i) and~(iii) above that 
    \begin{align}\label{equ:APPENDIX:THM:GenTuran-Tree-Exact-1}
        |G[V_1, V_2]| \ge \frac{n^2}{4} - 2\delta n^2. 
    \end{align}
    Combined with Statement~(iv), for every $x_i \in L$ the intersection of links $L_{\mathcal{H}}(x_i)$ and $L_{\mathcal{S}'}(x_i)$  satisfies 
    \begin{align*}
        |L_{\mathcal{H}}(x_i) \cap L_{\mathcal{S}'}(x_i)|
        = |L_{\mathcal{H}}(x_i) \cap G[V_1, V_2]|
        \ge |G[V_1, V_2]| - \delta n \times n
        \ge \frac{n^2}{4} - 3\delta n^2. 
    \end{align*}
    Therefore, 
    \begin{align}\label{APPENDIX:THM:GenTuran-Tree-Exact-2}
        |\mathcal{M}|
        = \sum_{x_i \in L}\left(|L_{\mathcal{S}'}(x_i)| - |L_{\mathcal{H}}(x_i) \cap L_{\mathcal{S}'}(x_i)|\right)
        \le t \left(|V_1||V_2| - \left(\frac{n^2}{4} - 3\delta n^2\right)\right)
        \le 3\delta t n^2. 
    \end{align}
    Inequality~\eqref{equ:APPENDIX:THM:GenTuran-Tree-Exact-1} with some simple calculations also imply that 
    \begin{align}\label{APPENDIX:THM:GenTuran-Tree-Exact-3}
        \left(\frac{1}{2} - \sqrt{2\delta}\right)n 
            \le |V_i| 
            \le \left(\frac{1}{2} + \sqrt{2\delta}\right)n 
            \quad\text{for}\quad i\in \{1,2\}. 
        \end{align}
    Let 
    \begin{gather*}
        \tau := 3k,\quad 
        D  := \left\{y\in V' \colon d_{\mathcal{H}}(y x_i) \ge \tau \text{ for all } x_i \in L \right\}, \quad \overline{D} := V'\setminus D,  \\
        D_i  := D \cap V_i \quad\text{and}\quad
        \overline{D}_i := V_i \setminus D_i \quad\text{for}\quad i\in \{1,2\}.  \quad
    \end{gather*}
    \begin{claim}\label{APPENDIX:CLAIM:GenTuran-tree-D-bar}
        We have $m \ge 49 n|\overline{D}|/100$ and  $|\overline{D}| \le 8 \delta t n$. 
    \end{claim}
    \begin{proof}
        By the definition of $D$, 
        for every $i\in \{1,2\}$ and for every vertex $v\in \overline{D}_i$ there exists a vertex $x\in L$ (depending on $v$) such that 
        \begin{align*}
            |N_{\mathcal{H}}(vx) \cap V_{3-i}| \le \tau. 
        \end{align*}
        Note that this pair $\{v,x\}$ contributes at least $|V_{2-i}| - \tau$ elements to $\mathcal{M}$. 
        Therefore, it follows from~\eqref{APPENDIX:THM:GenTuran-Tree-Exact-3} that 
        \begin{align*}
            m 
             \ge \left(\min\{|V_1|, |V_2|\} - \tau \right) \left(|\overline{D}_1| + |\overline{D}_2|\right) 
             \ge \left(\frac{n}{2}-\sqrt{2\delta}n - \tau\right)|\overline{D}|
             \ge \frac{49}{100} n |\overline{D}|. 
        \end{align*}
        In addition, we obtain 
        \begin{align*}
            |\overline{D}|
            \le \frac{3\delta t n^2}{49 n/100}
            \le 8 \delta t n. 
        \end{align*}
        This proves Claim~\ref{APPENDIX:CLAIM:GenTuran-tree-D-bar}. 
    \end{proof}
    \begin{claim}\label{APPENDIX:CLAIM:GenTuran-tree-bad}
        We have $\mathcal{B}[V'] \setminus \mathcal{B}\left[\overline{D}\right] = \emptyset$. 
    \end{claim}
    \begin{proof}
        The proof is the same as the proof of Claim~\ref{CLAIM:Turan-size-b-1}. 
    \end{proof}
    \begin{claim}\label{APPENDIX:CLAIM:GenTuran-cycle-Maxdeg}
        For every $v\in V'$ we have 
        \begin{align*}
            \min\left\{|N_{G}(v) \cap D_1|, |N_{G}(v) \cap D_2|\right\} 
            \le \sqrt{3\delta}n. 
    \end{align*}
    \end{claim}
    \begin{proof}
        The proof is the same as the proof of Claim~\ref{CLAIM:GenTuran-cycle-Maxdeg}. 
    \end{proof}
    \begin{claim}\label{APPENDIX:CLAIM:GenTuran-bad-deg-max}
        For every $i\in \{1,2\}$ and for every $v\in V_i$ we have $|N_{G}(v) \cap D_i| \le 9\sqrt{\delta}n$. 
    \end{claim}
    \begin{proof}
        The proof is the same as the proof of Claim~\ref{CLAIM:GenTuran-bad-deg-max}. 
    \end{proof}
    For $i\in \{1,2,3\}$ let 
    \begin{align*}
        \mathcal{B}_i 
        := \left\{E\in \mathcal{B} \colon |E \cap \overline{D}| = i\right\}. 
    \end{align*}
    Since $\mathcal{B}_3$ is $T^3$-free, it follows from the definition of $C$ that 
    \begin{align}\label{APPENDIX:THM:GenTuran-Tree-Exact-4}
        |\mathcal{B}_3| \le C |\overline{D}|^2. 
    \end{align}
    By Claim~\ref{APPENDIX:CLAIM:GenTuran-tree-bad}, every edge in $\mathcal{B}_2$ must contain one vertex in $L$ and two vertices in $\overline{D}$. 
    Therefore, 
    \begin{align}\label{APPENDIX:THM:GenTuran-Tree-Exact-5}
        |\mathcal{B}_2| \le t |\overline{D}|^2. 
    \end{align}
    By Claim~\ref{APPENDIX:CLAIM:GenTuran-tree-bad} again, every edge in $\mathcal{B}_1$ must contain one vertex in $L$, one vertex in $D_i$, and one vertex in $\overline{D}_{i}$ for some $i\in \{1,2\}$.  
    So by the caculation of $|{\cal B}'|$ in  the proof of Claim~\ref{CLAIM:GenTuran-cycle-2-intersecting}  
    \begin{align}\label{APPENDIX:THM:GenTuran-Tree-Exact-6}
        |\mathcal{B}_1| 
        \le t \times 9\sqrt{\delta}n \left(|\overline{D}_1|+|\overline{D}_2|\right)
        = 9\sqrt{\delta} t^2n |\overline{D}|. 
    \end{align}
    By Claim~\ref{APPENDIX:CLAIM:GenTuran-tree-D-bar},~\eqref{APPENDIX:THM:GenTuran-Tree-Exact-4},~\eqref{APPENDIX:THM:GenTuran-Tree-Exact-5}, and~\eqref{APPENDIX:THM:GenTuran-Tree-Exact-6}, we obtain 
    \begin{align*}
        |\mathcal{L}|
        & = |\mathcal{S}| + |\mathcal{B}_1| + |\mathcal{B}_2|+ |\mathcal{B}_3|+ |\mathcal{M}|  \\
        & \le |\mathcal{S}_{\mathrm{bi}}(n,t)| + C |\overline{D}|^2 + t |\overline{D}|^2+  9\sqrt{\delta} t^2n |\overline{D}| - \frac{49}{100}n|\overline{D}| \\
        & = |\mathcal{S}_{\mathrm{bi}}(n,t)| - \left(\frac{49}{100}n - C\times 8\delta n - t \times 8\delta n - 9\sqrt{\delta} t^2n\right) |\overline{D}| \\
        & \le |\mathcal{S}_{\mathrm{bi}}(n,t)|.  
    \end{align*}
    Note that equality holds iff $\overline{D} = \emptyset$, which means that $G \cong S(n,t)$. 
    This completes the proof of Theorem~\ref{APPENDIX:THM:GenTuran-Tree-Exact}. 
\end{proof}

\section{Proof of Theorem~\ref{THM:GenTuran-Even-Path-Exact}}\label{APPENDIX:SEC:GenTuran-even-path-exact}
Recall the following statement of Theorem~\ref{THM:GenTuran-Even-Path-Exact}. 
\begin{theorem}\label{APPENDIX:THM:GenTuran-Even-Path-Exact}
    Suppose that $t \ge 1$ is a fixed integer and $n$ is sufficiently large. 
    Then 
    \begin{align*}
        \mathrm{ex}(n, K_{3}, P_{2t+2}^{\triangle}) 
        = \left|\mathcal{S}^{+}_{\mathrm{bi}}\left(n,t\right)\right|. 
    \end{align*}
\end{theorem}
    The proof for Theorem~\ref{APPENDIX:THM:GenTuran-Even-Path-Exact} is almost identical to the proof of Theorem~\ref{THM:GenTuran-Cycle-Exact} for $k = 2t+2\ge 6$.
    The only difference is the proof for the following claim. 
    \begin{claim}\label{APPENDIX:THM:GenTuran-Even-Path-Exact-claim-1}
         The $3$-graph $\mathcal{B}[V']$ does not contain two edges $e_1$ and $e_2$ such that 
        \begin{enumerate}[label=(\roman*)]
            \item\label{APPENDIX:equ:CLAIM:GenTuran-path-P_2} 
            $|e_1 \cap e_2| = 1$, $(e_1 \setminus e_2) \cap D \neq \emptyset$, and $(e_2 \setminus e_1) \cap D \neq \emptyset$, or 
            \item\label{APPENDIX:equ:CLAIM:GenTuran-path-B1-two-edges}
                $\min\left\{|e_1 \cap D|,\ |e_2 \cap D|\right\} \ge 2$.  
        \end{enumerate}
    \end{claim}
    \begin{proof}
        Suppose to the contrary that there exist two edges $e_1, e_2 \in \mathcal{B}[V']$ such that~\ref{APPENDIX:equ:CLAIM:GenTuran-path-P_2} holds. 
        Let $\{v_0\} := e_1 \cap e_2$ and 
        fix $v_i \in \left(e_i \cap D\right) \setminus \{v_0\}$ for $i\in \{1,2\}$. 
        Let $D' := D\setminus\left(e_1 \cup e_2\right)$. 
        Choose any $t$-set $\{u_1, \ldots, u_{t}\} \subseteq D'$. 
        It follows from the definition of $D$ that the induced bipartite graph of $\partial\mathcal{H}$ on $L \cup D$ is complete. 
        Therefore, $F := v_2 x_1 u_1 x_2 u_2 \cdots x_{t} u_{t}$ is copy of $P_{k-2}$ in the bipartite graph $\partial\mathcal{H}[L, D']$. 
        This $P_{k-2}$ together with $v_1 v_0 v_2$ (a copy of $P_2$) form a copy of $P_k$ in $\partial\mathcal{H}$. 
        Since all edges in $F$ have codegree at least $\tau \ge 3k$ in $\mathcal{H}$, it follows from Fact~\ref{FACT:partial-embedding} that $P_{k}^3 \subseteq \mathcal{H}$, a contradiction. 

        Suppose to the contrary that there exist two disjoint edges $e_1, e_2 \in \mathcal{B}[V']$ such that~\ref{APPENDIX:equ:CLAIM:GenTuran-path-B1-two-edges} holds. 
        Fix a $2$-set $\{v_i, v_i'\} \subseteq e_i \cap D$ for $i\in \{1,2\}$. 
        Choose a $(t-1)$-set $\{u_1, \ldots, u_{t-1}\} \subseteq D\setminus (e_1 \cup e_2)$. 
        Similar to the proof above, the graph $F := v_1' v_1 x_1 u_1 x_2  \cdots u_{t-1} x_{t} v_2 v_2'$ is a copy of $P_{k}$ in $\partial\mathcal{H}$. 
        Note that all edges but $v_1 v_1', v_2 v_2'$ in $F$ have codegree at least $\tau \ge 3k$ in $\mathcal{H}$. 
        So it follows from Fact~\ref{FACT:partial-embedding} that $P_{k}^{3} \subseteq \mathcal{H}$, a contradiction.
    \end{proof}
\section{Proof of Theorem~\ref{THM:GenTuran-Cycle-Exact} for the odd case}\label{APPENDIX:SEC:GenTuran-odd-cycle-exact}
Recall the following statement of Theorem~\ref{THM:GenTuran-Cycle-Exact} for the odd case. 
\begin{theorem}\label{APPENDIX:THM:GenTuran-odd-Cycle-Exact}
    Suppose that $t \ge 2$ is a fixed integer and $n$ is sufficiently large. 
    Then 
    \begin{align*}
        \mathrm{ex}(n, K_{3}, C_{2t+1}^{\triangle})
         = \left|\mathcal{S}_{\mathrm{bi}}\left(n,t\right)\right|. 
     \end{align*}
\end{theorem}
\begin{proof}[Proof of Theorem~\ref{APPENDIX:THM:GenTuran-odd-Cycle-Exact}]
    Fix $k = 2t+1 \ge 5$. 
    Let $C>0$ be a constant such that $\mathrm{ex}(N, C_{k}^3) \le C N^2$ holds for all integers $N \ge 0$. 
    The existence of such a constant $C$ is guaranteed by the theorem of Kostochka--Mubayi--Verstra\"{e}te~\cite{KMV15a}. 
    Let $0 < \delta \ll C^{-1}$ be sufficiently small and $n \gg C$ be sufficiently large. 
    Let $G$ be an $n$-vertex $C_{k}^{\triangle}$-free graph with 
    \begin{align*}
        N(K_3, G) = |\mathcal{S}_{\mathrm{bi}}(n,t)|. 
    \end{align*}
    We may assume that every edge in $G$ is contained in some triangle of $G$, since otherwise we can delete it from $G$ and this does not change the value of $N(K_3, G)$. 
    We aim to prove that $G \cong S_{\mathrm{bi}}(n,t)$.
    
    Since $N(K_3, G) = t n^2/4 - o(n^2)$ and $n$ is large, it follows from Theorem~\ref{THM:GenTuran-Cycle-Stability} that there exists a $t$-set $L:= \{x_1, \ldots, x_{t}\} \subseteq V(G)$ such that 
    \begin{enumerate}[label=(\roman*)]
        \item $|G-L| \ge n^2/4 - \delta n^2$, 
        \item $N(K_3, G-L) \le \delta n^2$, 
        \item $G-L$ can be made bipartite by removing at most $\delta n^2$ edges, and
        \item $d_{G}(v) \ge (1-\delta)n$ for all $v\in L$. 
    \end{enumerate}
    Let $V := V(G)$, $V':= V-L$, and let $V_1 \cup V_2 = V'$ be a bipartition such that the number of edges (in $G$) crossing $V_1$ and $V_2$ is maximized. 
    Define 
    \begin{align*}
        \mathcal{S} &:= \left\{e \in \binom{V}{3} \colon |e \cap L| \ge 1,\ |e\cap V_1| \le 1,\ |e\cap V_2| \le 1\right\}, \\
        \mathcal{S}' &:= \left\{e \in \binom{V}{3} \colon |e \cap L| = |e\cap V_1| = |e\cap V_2| = 1\right\}, \\
        \mathcal{H} &:= \mathcal{K}_{G},\quad
        \mathcal{B} := \mathcal{H}\setminus \mathcal{S}, \quad\mathrm{and}\quad
        \mathcal{M} := \mathcal{S}'\setminus  \mathcal{H}. 
    \end{align*}
    Let $m:= |\mathcal{M}|$ and $b:= |\mathcal{B}|$. 
    It follows from Statements~(i) and~(iii) above that 
    \begin{align}\label{APPENDIX:equ:GenTuran-odd-cycle-G[V1,V2]}
        |G[V_1, V_2]| \ge \frac{n^2}{4} - 2\delta n^2. 
    \end{align}
    Combined with Statement~(iv), for every $x_i \in L$ the intersection of links $L_{\mathcal{H}}(x_i)$ and $L_{\mathcal{S}'}(x_i)$  satisfies 
    \begin{align*}
        |L_{\mathcal{H}}(x_i) \cap L_{\mathcal{S}'}(x_i)|
        = |L_{\mathcal{H}}(x_i) \cap G[V_1, V_2]|
        \ge |G[V_1, V_2]| - \delta n \times n
        \ge \frac{n^2}{4} - 3\delta n^2. 
    \end{align*}
    Therefore, 
    \begin{align}\label{APPENDIX:equ:GenTuran-odd-cycle-m-upper}
        |\mathcal{M}|
        = \sum_{x_i \in L}\left(|L_{\mathcal{S}'}(x_i)| - |L_{\mathcal{H}}(x_i) \cap L_{\mathcal{S}'}(x_i)|\right)
        \le t \left(|V_1||V_2| - \left(\frac{n^2}{4} - 3\delta n^2\right)\right)
        \le 3\delta t n^2. 
    \end{align}
    Inequality~\eqref{APPENDIX:equ:GenTuran-odd-cycle-G[V1,V2]} with some simple calculations also imply that 
    \begin{align}\label{APPENDIX:equ:GenTuran-odd-cycle-Vi-size}
        \left(\frac{1}{2} - \sqrt{2\delta}\right)n 
            \le |V_i| 
            \le \left(\frac{1}{2} + \sqrt{2\delta}\right)n 
            \quad\text{for}\quad i\in \{1,2\}. 
        \end{align}
    Let 
    \begin{gather*}
        \tau := \frac{n}{200},\quad 
        D  := \left\{y\in V' \colon d_{\mathcal{H}}(y x_i) \ge \tau \text{ for all } x_i \in L \right\}, \quad \overline{D} := V'\setminus D,  \\
        D_i  := D \cap V_i \quad\text{and}\quad
        \overline{D}_i := V_i \setminus D_i \quad\text{for}\quad i\in \{1,2\}.  \quad
    \end{gather*}
    We also divide $\mathcal{M}$ further by letting 
    \begin{align*}
        \mathcal{M}_1 := \left\{e\in \mathcal{M} \colon e\cap \overline{D} \neq\emptyset\right\}, \quad\text{and}\quad
        \mathcal{M}_2 := \mathcal{M}\setminus \mathcal{M}_1. 
    \end{align*}
    Let $m_1 := |\mathcal{M}_1|$ and $m_2:= |\mathcal{M}_2|$. 
    \begin{claim}\label{APPENDIX:CLAIM:GenTuran-odd-cycle-D-bar}
        We have $m_1 \ge 49 n|\overline{D}|/100$ and  $|\overline{D}| \le 8 \delta t n$. 
    \end{claim}
    \begin{proof}
        Same as the proof of Claim~\ref{CLAIM:GenTuran-cycle-D-bar}.
    \end{proof}
    \begin{claim}\label{APPENDIX:CLAIM:GenTuran-odd-cycle-bad-edge-1}
        The $3$-graph $\mathcal{B}[V']$ does not contain an edge $e$ with $|e\cap D| \ge 2$. 
    \end{claim}
    \begin{proof}
        Suppose to the contrary that there exists an edge $\{v_1, v_2, v_3\} \in \mathcal{B}[V']$ such that $\{v_1, v_2\} \subset D$. 
        Let $D' := D\setminus \{v_1, v_2, v_3\}$. 
        Choose $t-1$ vertices $u_1, \ldots, u_{t-1} \in D'$. 
        It is easy to see that $v_1 x_1 u_1 x_2 \cdots u_{t-1} x_t v_2 v_1$ is a copy of $C_{k}$ in $\partial\mathcal{H}$. 
        Since all edges but $v_1v_2$ in this $C_{k}$ has codegree at least $\tau \ge 3k$ in $\mathcal{H}$, it follows from Fact~\ref{FACT:partial-embedding} that $C_{k}^{3} \subseteq \mathcal{H}$, a contradiction. 
    \end{proof}
    \begin{claim}\label{APPENDIX:CLAIM:GenTuran-odd-cycle-bad-edge-2}
        The $3$-graph $\mathcal{B}[V']$ does not contain two edges $e_1$ and $e_2$ such that 
        \begin{align}\label{equ:APPENDIX:CLAIM:GenTuran-odd-cycle-bad-edge-2}
            |e_1 \cap e_2| = 1, \quad 
            (e_1 \setminus e_2) \cap D \neq \emptyset, \quad\text{and}\quad
            (e_2 \setminus e_1) \cap D \neq \emptyset. 
        \end{align}
    \end{claim}
    \begin{proof}
        Suppose to the contrary that there exist two edges $e_1, e_2 \in \mathcal{B}[V']$ such that~\eqref{equ:APPENDIX:CLAIM:GenTuran-odd-cycle-bad-edge-2} holds. 
        Let $\{v_0\} := e_1 \cap e_2$ and 
        fix $v_i \in \left(e_i \cap D\right) \setminus \{v_0\}$ for $i\in \{1,2\}$. 
        Let $D' := D\setminus\left(e_1 \cup e_2\right)$. 
        Since $\tau = n/200 > |\overline{D}|+t+5$, it follows from the definition of $D$ that there exists a vertex $u_1 \in D'$ such that $\{v_1, u_1, x_1\} \in \mathcal{H}$. 
        Choose any $(t-2)$-set $\{u_2, \ldots, u_{t-1}\} \subseteq D'\setminus \{u_1\}$.
        It follows from the definition of $D$ again that $F:= u_1 x_2 \cdots u_{t_1} x_{t} v_2$ is a copy of $P_{2t-2}$ in $\partial\mathcal{H}$. 
        This $P_{2t-2}$ together with $u_1v_1v_0v_2$ (a copy of $P_3$) forms a copy of $C_{k}$ in $\partial\mathcal{H}$. 
        Since every edge in $F$ has codegree at least $t \ge 3k$ in $\mathcal{H}$, it follows from Fact~\ref{FACT:partial-embedding} that $C_{k}^3 \subseteq \mathcal{H}$, a contradiction.
    \end{proof}
    \begin{claim}\label{APPENDIX:CLAIM:GenTuran-odd-cycle-Maxdeg}
        For every $v\in V'$ we have 
        \begin{align*}
            \max\left\{|N_{G}(v) \cap D_1|, |N_{G}(v) \cap D_2|\right\} 
            \le \sqrt{3\delta}n. 
    \end{align*}
    \end{claim}
    \begin{proof}
        Same as the proof of Claim~\ref{CLAIM:GenTuran-cycle-Maxdeg}. 
    \end{proof}
    \begin{claim}\label{APPENDIX:CLAIM:GenTuran-odd-cycle-bad-deg-max}
        For every $i\in \{1,2\}$ and for every $v\in V_i$ we have $|N_{G}(v) \cap D_i| \le 9\sqrt{\delta}tn$. 
    \end{claim}
    \begin{proof}
        Same as the proof of Claim~\ref{CLAIM:GenTuran-bad-deg-max}. 
    \end{proof}
    For $i\in \{1,2,3\}$ let 
    \begin{align*}
        \mathcal{B}_i 
        := \left\{E\in \mathcal{B} \colon |E \cap \overline{D}| = i\right\}. 
    \end{align*}
    Since $\mathcal{B}_3$ is $C_{k}^3$-free, it follows from the definition of $C$ that 
    \begin{align}\label{APPENDIX:equ:GenTuran-odd-cycle-B3}
        |\mathcal{B}_3| \le C |\overline{D}|^2. 
    \end{align}
    \begin{claim}\label{APPENDIX:CLAIM:GenTuran-odd-cycle-B_2}
        We have $ |\mathcal{B}_2| \le {n |\overline{D}|}/{3}$.
    \end{claim}
    \begin{proof}
        It follows from the proof of Claim~\ref{CLAIM:GenTuran-cycle-B_2'}. 
    \end{proof}
    \begin{claim}\label{APPENDIX:CLAIM:GenTuran-odd-cycle-B_1}
        We have $|\mathcal{B}_1| \le 9\sqrt{\delta} t^2 n |\overline{D}|$. 
    \end{claim}
    \begin{proof}
        It follows from Claim~\ref{APPENDIX:CLAIM:GenTuran-odd-cycle-bad-edge-1} that every triple in $\mathcal{B}_1$ contains exactly one vertex from each set in $\{D_{i}, \overline{D}_i, L\}$ for some $i\in \{1,2\}$.  
        Then, by Claim~\ref{APPENDIX:CLAIM:GenTuran-odd-cycle-bad-deg-max}, we have 
        \begin{align*}
            |\mathcal{B}_1|
            \le t \times 9\sqrt{\delta}t n \times \left(|\overline{D}_1| + |\overline{D}_2|\right)
            = 9\sqrt{\delta} t^2 n |\overline{D}|, 
        \end{align*}
        which proves Claim~\ref{APPENDIX:CLAIM:GenTuran-odd-cycle-B_1}. 
    \end{proof}
    Let 
        \begin{gather*}
            \tau' := \frac{n}{3}
            \quad\text{and}\quad
            D_i' := \left\{v\in D_i \colon |N_{G}(v)\cap V_{3-i}| \ge \tau'\right\}
            \quad\text{for}\quad i\in \{1,2\}. 
        \end{gather*}
    Let $\overline{D'} := D\setminus D'$ and $\overline{D'}_i := D_i\setminus D_i'$ for $i\in \{1,2\}$.
        \begin{claim}\label{APPENDIX:CLAIM:GenTuran-odd-cycle-D'-in-D}
            We have $|\overline{D'}_i| \le 14\delta n$ for $i\in \{1,2\}$.   
        \end{claim}
        \begin{proof}
            Same as the proof of Claim~\ref{CLAIM:GenTuran-cycle-D'-in-D}.  
        \end{proof}
        \begin{claim}\label{APPENDIX:CLAIM:GenTuran-odd-cycle-GD1'D2'}
            We have $\left|G[D_1'] \cup G[D_2']\right| = 0$. 
        \end{claim}
        \begin{proof}
            It follows from Claim~\ref{APPENDIX:CLAIM:GenTuran-odd-cycle-bad-edge-1} and the proof of Claim~\ref{CLAIM:GenTuran-cycle-GD1'D2'}.
        \end{proof}
        \begin{claim}\label{APPENDIX:CLAIM:GenTuran-odd-cycle-B0}
            We have $|\mathcal{B}_0| \le 9\sqrt{\delta}t^2n |\overline{D'}|$. 
        \end{claim}
        \begin{proof}
            It follows from Claim~\ref{APPENDIX:CLAIM:GenTuran-odd-cycle-bad-edge-1} and the proof of Claim~\ref{CLAIM:GenTuran-even-cycle-B0}.
        \end{proof}
        \begin{claim}\label{APPENDIX:CLAIM:GenTuran-odd-cycle-M2}
            We have $m_2 \ge tn|\overline{D'}|/7$. 
        \end{claim}
        \begin{proof}
           Same as the proof of Claim~\ref{CLAIM:GenTuran-even-cycle-M2}. 
        \end{proof}
        By~\eqref{APPENDIX:equ:GenTuran-odd-cycle-B3}, Claims~\ref{APPENDIX:CLAIM:GenTuran-odd-cycle-B_2},~\ref{APPENDIX:CLAIM:GenTuran-odd-cycle-B_1},~\ref{APPENDIX:CLAIM:GenTuran-odd-cycle-B0}, and~\ref{APPENDIX:CLAIM:GenTuran-odd-cycle-M2}, we have 
            \begin{align*}
                |\mathcal{H}|
                & = |\mathcal{S}_{\mathrm{bi}}(n,t)| 
                + \sum_{i=0}^{4}|\mathcal{B}_i| 
                - |\mathcal{M}| -|\mathcal{M}'| \\
                & \le |\mathcal{S}_{\mathrm{bi}}(n,t)|  
                + 9\sqrt{\delta}t^2 n |\overline{D'}|
                + 9\sqrt{\delta}t^2n |\overline{D}|
                + \frac{n}{3}|\overline{D}| + C|\overline{D}|^2
                - \frac{49}{100}n|\overline{D}|
                - \frac{t}{7}n |\overline{D'}|\\
                & \le |\mathcal{S}_{\mathrm{bi}}(n,t)|
                - \left(\frac{tn}{7} - \sqrt{3\delta}tn\right) |\overline{D'}| 
                - \left(\frac{49}{100}n - 9\sqrt{\delta}t^2n 
                - \frac{n}{3}- C\times 8\delta tn\right)
                |\overline{D}| \\
                & \le |\mathcal{S}_{\mathrm{bi}}(n,t)|. 
            \end{align*}
            Note that the equality holds only if $|\overline{D}| = 0$ and $|\overline{D'}|=0$. 
            Further (simple) calculations show that equality holds iff $G\cong S(n,t)$, which proves Theorem~\ref{APPENDIX:THM:GenTuran-odd-Cycle-Exact}. 
\end{proof}
\end{appendix}
%
\end{document}